\documentclass[a4paper]{article}
\usepackage[utf8x]{inputenc}
\usepackage{amsthm}
\usepackage{float}
\usepackage{mathtools,nccmath}
\usepackage{amssymb}
\usepackage{amsmath}
\usepackage{amsfonts}
\usepackage{latexsym}
\usepackage{paralist}
\usepackage{enumitem}
\usepackage{multirow}
\usepackage{url}
\usepackage{fullpage}
\usepackage{bbold}
\usepackage{xcolor}
\usepackage{algorithmic}
\usepackage{tikz}
\usetikzlibrary{calc,fit,spy}
\usepackage{pgfplots}
\usetikzlibrary{arrows.meta}
\usepgfplotslibrary{fillbetween}
\pgfplotsset{width=7cm,
        compat=1.5.1,
        standard/.style={
        axis x line=middle,
        axis y line=middle,
        enlarge x limits=0.15,
        enlarge y limits=0.15,
        every axis x label/.style={at={(current axis.right of origin)},anchor=north west},
        every axis y label/.style={at={(current axis.above origin)},anchor=north east}
    }
}
\usepackage{multirow}

 \newcommand\ForAuthors[1]
 {\par\smallskip                     
  \begin{center}
   \fbox
   {\parbox{0.9\linewidth}
    {\raggedright\sc--- #1}
   }
  \end{center}
  \par\smallskip                     %
 }
  \newcommand\comment[1]{}

\title{Extensions of $\mathcal{KL}$ and Lyapunov Functions for Discrete-time Dynamical System \red{Peak} Analysis}

\author{Assalé Adjé\footnote{mail: assale.adje@univ-perp.fr}\\ LAboratoire de Modélisation et Pluridisciplinaire et Simulations\\
LAMPS\\ Université de Perpignan Via Domitia\\ France
}

\date{}

\def\funset{\mathbb{\Gamma}}
\def\fset{\Omega([0,1])}
\def\secfunh{\funset_{\rm fun}}
\def\secfunb{\funset_{\rm sc}}
\def\inv#1#2{\displaystyle{{#1}_{\left(#2\right)}^{-1}}}

\def\Fnu{\mathfrak{F}_{\nu}}
\def\mm{\mathbb 1}
\def\Idd{\mathrm{Id}}

\def\defNU#1{\mathcal{D}\left(#1\right)}
\def\newop#1{\overline{N_{#1}}}

\def\conv{\operatorname{conv}}
\def\rr{\mathbb R}

\def\rd{\rr^d}
\def\nn{\mathbb N}
\def\xin{X^\mathrm{in}}

\def\fin{\operatorname{Fin}\left(\xin\right)}
\def\fini#1{\operatorname{Fin}\left(#1\right)}
\def\osp#1{\overline{#1_{\xin}}}
\def\norm#1{\| #1\|}
\def\ivt{I_{\xin,T}^{\varphi}}
\def\scivt{\operatorname{SC}_{\xin,T}^{\varphi}}
\def\ccdom{\mathrm{SIC}}

\def\agmx{\operatorname{Argmax}}

\def\Max{\operatorname*{Max}}
\def\aff{\mathrm{Aff}}

\def\gs{\Delta^s}

\def\klcls{\mathcal{K}\mathcal{L}}
\def\kcls{\mathcal{K}}

\def\kclsi{\mathcal{K}_\infty}

\def\lcls{\mathcal{L}}
\def\klgen{\mathcal{KL}_{\rm gen}}

\def\bigkse{\mathbf{K}}
\def\bigks{\mathbf{K}^s}

\def\dom{\operatorname{dom}}

\def\nuopt{\nu_{\rm opt}}
\def\kopt{k_{\rm opt}}
\def\xopt{x_{\rm opt}}

\def\gls#1{{\mathrm R}_{#1}}
\def\gsls#1{\gls{#1}^s}
\def\IR{\mathbb{I}\left(\rr\right)}
\def\Idd{\mathrm{Id}}

\def\res{\mathcal{S}}
\def\PSD{\mathcal{PSD}}

\def\red#1{#1}
\def\blue#1{#1}

\def\afct{\mathfrak{a}}
\def\bfct{\mathfrak{b}}
\def\xll{x^{ll}}
\def\xul{x^{ul}}
\def\xlu{x^{lu}}
\def\xuu{x^{uu}}
\def\philes{\varphi_{\ell}}
\def\xles{X_\ell}
\def\xinles{\xin_{\ell}}
\def\theles{\theta_\ell}
\def\epsles{\overline{\varepsilon}}
\def\omeles{\omega}
\def\omelesopt{\omeles_{\rm opt}}
\def\lyaples{V_{\ell}}
\def\osples{\overline{{\lyaples}_{\xinles}}}
\def\osps#1#2{\overline{{#1}_{#2}}}

\def\lexp{L_e}

\newtheorem{assumption}{Assumption}
\newtheorem{proposition}{Proposition}
\newtheorem{defi}{Definition}
\newtheorem{theorem}{Theorem}
\newtheorem{corollary}{Corollary}
\newtheorem{problem}{Problem}
\newtheorem{lemma}{Lemma}
\newtheorem{remark}{Remark}
\newtheorem{example}{Example}

\begin{document}
\maketitle

\begin{abstract}
In this paper, we extend two classes of functions that are classically involved in asymptotic stability analyses.\red{The purpose of this extension is} to study a maximization problem on the reachable values of a discrete-time dynamical system. This particular maximization problem is called a \red{peak} computation problem. The problem \red{consists in finding} a couple composed of an initial state and a time that maximizes a given function over states. The paper focuses on the time component of the optimal solution, which is an integer as the time is discrete. \red{We apply a method developed in previous papers to compute an upper bound of the greatest index maximizer of a real sequence. This previous method is based on a formula that requires a pair of a strictly increasing and continuous function on $[0,1]$ and a convergent geometric sequence that provides an upper bound of the analyzed sequence. This pair is proven to exist}. However, in practice, the computation cannot be \red{done in a general setting}. In this paper, we developed two alternative methods. The first is based on discontinuous and non-strictly increasing/decreasing $\mathcal{KL}$-like functions called $\klgen$ functions. We prove that the existence of a $\klgen$ upper bound is equivalent to the existence of a pair of a strictly increasing and continuous function on $[0,1]$ and a convergent geometric sequence. The construction of a strictly increasing continuous function from a $\klgen$ function requires an extension of the famous Sontag's lemma. Finally, we construct a new type of Lyapunov function that is well designed for our \red{peak} computation problem. These functions are called Opt-Lyapunov functions. Opt-Lyapunov functions are \red{suitable as} we establish an equivalence theorem between the existence of an Opt-Lyapunov function and that of a pair of a strictly increasing and continuous function on $[0,1]$ and a convergent geometric sequence. The construction of \red{an} Opt-Lyapunov function from a pair of a strictly increasing and continuous function on $[0,1]$ and a convergent geometric sequence is inspired by the Yoshizawa construction of Lyapunov functions.
\end{abstract}


\section{Introduction}

In this paper, we study the problem of {\it \red{peak}} computation for discrete-time dynamical systems in finite-dimensional state spaces. A \red{peak} computation problem can be viewed as a particular maximization problem for which the constraint set is the reachable value set of a discrete-time dynamical system. The objective function of this optimization problem is a function of the states. This particular maximization problem can be used to model various situations in engineering, economics, physics, and biology... For example, in population dynamics, for a predator-prey system, one may be interested in the maximal size of each species, even if the system asymptotically oscillates between several equilibria. This problem boils down to maximizing each coordinate of the state variable separately over the system modeling the population dynamics. 

Besides its importance in analyzing critical concrete situations, the literature on \red{peak} analysis problems in the community of dynamical systems analysis is \red{quite limited}. The \red{peak} computation problem appears in a discrete-time setting for linear systems, where the objective function is the Euclidean norm~\cite{ahiyevich2018upper}. This problem also appears in the analysis of stable linear systems with general quadratic objective functions~\cite{DBLP:journals/jota/Adje21,adje13052025}. A slightly different version of the problem was encountered in~\cite{ahmadi2024robust}. The difference arises from the fact that the state variable is constrained to remain in a given set. The problem was studied in a continuous-time setting and algebraically in a polynomial environment in~\cite{miller2020peaksafety}. In ~\cite{miller2020peaksafety}, the authors address the problem of verifying safety properties (see e.g., ~\cite{belta2017formal} for verification problems on discrete-time systems) for dynamical systems and formulate the safety verification problem as a \red{peak} computation problem. Safety verification problems aim to prove that the trajectories of a system are contained in a safe set. A safety verification problem can be translated into a \red{peak} computation problem when the safe set is a sublevel of \red{sufficiently regular function}. In this case, the objective function of the \red{peak} computation problem is associated with the sublevel set. A safety analysis from the \red{peak} computation point of view was also proposed for switched systems ~\cite{adje2017proving} and for program analysis in~\cite{adje2015property}. For the safety analysis of dynamical systems, one can think that a reachability analysis (see e.g., ~\cite{rakovic2006reachability} and references therein) equivalent to a feasibility analysis in our context is sufficient to prove some properties of the dynamical system. However, maximality in our context is \red{synonymous} with criticality. The optimal value in this situation can be viewed as the value which penalizes the most the system with respect to some criteria and 
defines the \red{extreme condition} in which the system can enter.

In~\cite{miller2020peaksafety} the authors are interested in computing an upper bound of the \red{peak} optimal value for the continuous case. The techniques used are based on sum-of-squares (SOS) and linear matrix inequality (LMI) formulations (the interested reader can consult~\cite{Lasserre_2015} for those subjects). The same goal motivated the authors of~\cite{ahiyevich2018upper}. They limited their study to linear systems and norm objective functions and directly used semidefinite programming (see~\cite{ben2001lectures} for further details on the subject). Reference ~\cite{ahmadi2024robust} cannot be fairly compared \red{because} of the additional constraint. Moreover, the studied systems are linear or piecewise linear, and the objective functions are linear. The goal of the current paper is to provide the theoretical bases to solve exactly a \red{peak} computation problem in a general way. The developments made in this paper can be viewed as a theoretical generalization of works done in~\cite{DBLP:journals/jota/Adje21,adje13052025,adje2025kllyap} \red{and a specialization of}~\cite{adje2025maximizationrealsequences} \red{ to specific sequences of optimal values arising in peak analysis computations}. 

In this paper, we propose to study \red{peak} computation problems from \red{extensions of} classical tools of dynamical systems stability theory (for example, ~\cite[Chap. 4]{elaydiintroduction} for discrete time). This approach is based on a simple observation: the existence of peaks depends on the asymptotic behavior of the system. The importance of stability analysis is time-tested, and stability analysis is a prerequisite for the implementation of numerical methods. Although general stability properties can be difficult to obtain, asymptotic stability properties can be easily presented and understood from $\mathcal{KL}$ bounds (e.g., ~\cite{doi:10.1080/10236190902817844} or~\cite{NESIC20041025}), that is, the norms of the states are globally bounded from above by a $\mathcal{KL}$ function. However, proving that a candidate $\mathcal{KL}$ function is an upper bound over the norms of the state must be done for all integers, which can be difficult. In contrast, Lyapunov functions are functional certificates for system stability. Some results guarantee the existence of Lyapunov functions: a system is globally asymptotically stable if and only if there exists a Lyapunov function (see, e.g., \cite{khalil2002nonlinear}). Being a Lyapunov function is a timeless property that depends only on the dynamics. In a few advantageous situations, Lyapunov functions can be computed using numerical and optimization solvers (see the survey~\cite{giesl2015review}). 

Formally, the data of a \red{peak} computation problem are decomposed into an autonomous discrete-time dynamical system and an objective function $\varphi$ over states. The dynamical system consists of a set of initial conditions $\xin$ and the dynamics of the system, that is, a function $T$ that maps a state to a new state. The associated \red{peak} computation problem aims to maximize $\varphi(T^k(x))$ under the constraints $x\in \xin$ and $k$ is a natural integer. \red{An} optimal solution of a \red{peak} computation problem is then a couple $(\xopt,\kopt)$ where $\xopt\in\xin$ and $\kopt$ is a natural integer. Note that when an integer $k$ is fixed, maximizing $\varphi(T^k(x))$ over $x\in\xin$ is a standard maximization problem. The main difficulty of the \red{peak} computation problem is to compute an optimal $\kopt$, that is, an integer $n$ for which the supremum over $x\in\xin$ of $\varphi(T^n(x))$ is the greatest. \red{The optimal value of a peak computation problem can be interpreted as the sequence of the optimal value of standard maximization problems where the integer $k$ is fixed. From this interpretation, the integer element $\kopt$ of an optimal solution becomes an index for which the associated term maximizes the sequence. Therefore, we can apply the method developed in~\cite{adje2025maximizationrealsequences}. It consists in computing an upper bound of the (greatest) index maximizer of a real sequence that is bounded from above. The computation of this upper bound is based on a formula involving a pair $(h,\beta)$ where $h$ is a strictly increasing continuous function on $[0,1]$ and $\beta$ is a scalar belonging to $(0,1)$. The pair $(h,\beta)$ is also constrained to form a sequence $(h(\beta^k))_{k\in\nn}$ greater than the sequence of the suprema of $\varphi\circ T^k$ over $\xin$. To be applicable, the pair $(h,\beta)$ has to be \emph{useful}, that is, $h(0)$ must be strictly smaller than, at least, a supremum of $\varphi\circ T^k$ over $\xin$. In~\cite{adje2025maximizationrealsequences}, we proved that such a pair $(h,\beta)$ exists and can be  even chosen to be optimal, meaning that the pair permits finding the greatest optimal integer. In practice,  this pair is not computable}. Therefore, we propose two additional methods for constructing useful pairs. First, we extend the class of $\mathcal{KL}$ functions to the class of functions that are possibly discontinuous, unsigned (on the state variable but nonnegative on the time variable), increasing on states, and decreasing on time. We call this class of functions $\klgen$. As for classical $\mathcal{KL}$ functions, we use $\klgen$ functions as upper bounds on $\varphi(T^k(x))$ (for the classical case with $\klcls$, $\varphi$ is a norm). We establish an equivalence theorem, where we prove that a useful pair $(h,\beta)$ can be constructed from a $\klgen$ upper bound function, and we can construct $\klgen$ upper bounds from useful pairs $(h,\beta)$. The construction of useful pairs $(h,\beta)$ from $\klgen$ upper bound functions relies on an extension of Sontag's result~\cite[Prop. 7]{SONTAG199893} for $\klgen$. Finally, we propose a new concept of Lyapunov functions
adapted for the peaks computation problem, namely, the Opt-Lyapunov functions. The word adapted means that the existence of an Opt-Lyapunov function is equivalent to the existence of a useful pair $(h,\beta)$. Once again, the proofs of all implications are constructive. The construction of an Opt-Lyapunov function from a useful pair $(h,\beta)$ is inspired by the Yoshizawa construction of Lyapunov functions~\cite{yoshi}. The main difference between classical Lyapunov functions and this new concept is that we allow Opt-Lyapunov functions to have several zeros and take infinite values. These two new possibilities make the use of classical comparison functions irrelevant (see e.g., ~\cite{DBLP:journals/mcss/Kellett14} for a comprehensive survey). In our context, when we have access to classical Lyapunov functions, we can link, \red{when $\varphi$ is continuous}, $\varphi$ with classical Lyapunov functions using comparison functions. Without this link, we must develop a concept of weaker relations between $\varphi$ and an Opt-Lyapunov function. In this paper, we introduce the notion of a certificate of compatibility to overcome the loss of links.  

In summary, this study makes two main contributions.
\begin{itemize}  
	\item \red{We generalize the concept of $\klcls$ functions to $\klgen$ adapted to peaks computation problems for which we can formulate peaks computation problems as a $\klgen$ upper bounds existence}.
	\item \red{We introduce a new concept of Lyapunov functions well designed for peaks computation problem for which the comparison functions approach is replaced by a compatibility notion}.  
\end{itemize}   

The paper is organized as follows. Section~\ref{statement} presents the \red{peak} computation problem addressed in this paper and the main difficulty related to the search for an integer as an optimal solution \red{of the problem}. In Section~\ref{sec:leslieexample}, we introduce an illustrative example to which we apply the novel techniques in the following sections. This example is the simplest version of the Leslie model for studying the dynamics of a population decomposed into age classes. Then, in Section~\ref{klsection}, we introduce the new class $\klgen$. We prove that the existence of $\klgen$ upper bounds on $\varphi(T^k(x))$ is equivalent to the existence of a useful pair. In Section~\ref{newlyapunov}, we present the concept of Opt-Lyapunov functions
well-designed for our maximization problem. The classical comparison functions associated with classical Lyapunov functions are replaced by a new notion called compatibility. We also prove that the existence of a compatible Opt-Lyapunov function is equivalent to the existence of a useful pair. In Section~\ref{summarysec}, we summarize the results obtained, highlighting the concrete constructions between useful pairs of strictly increasing and continuous functions on $[0,1]$, convergent geometric sequences, $\klgen$ class, and Opt-Lyapunov functions.  
\vspace{0,2cm}

\noindent {\bf Notations}: $\rr$ stands for the set of reals, $\rr_+$ the set of nonnegative reals, $\rr^*$ the set of nonzero reals and $\rr_+^*$ the set of strictly positive reals. The vector space of vectors of $d$ reals is denoted by $\rd$. The set of natural integers is denoted by $\nn$ whereas $\nn^*$ denotes the set of nonzero natural integers.

\section{Problem Statement And Motivated Example}
\label{statement}
\subsection{The motivation : a discrete-time peak computation problem}
\subsubsection{Definition of the problem}
Let us consider a discrete-time dynamical system $(\xin,T)$ on $\rd$ where
\begin{itemize}
\item $\xin$ is the nonempty subset of $\rd$ of initial conditions;
\item $T$ is a nonzero self-map on $\rd$ that updates the state-variable.  
\end{itemize}
The system is autonomous and evolves without control, disturbances, or external inputs. In this paper, we are interested in the reachable values that maximize a given function. Then, given $\varphi:\rd\to\rd$, we define the associated {\it discrete-time peak computation problem} as follows:
\begin{equation}
\label{pcpprob}
\tag{DPCP}
\begin{array}{lcl}
\displaystyle{\Max_{x,k}} & &\displaystyle{\varphi(T^k(x))}\\
           & \text{ s. t.}& \displaystyle{x\in\xin, k\in\nn}
           \end{array} 
\end{equation}
where $T^k$ is the $k$-fold composition of $T$ if $k$ is not null, and the identity if $k=0$.
 
Like any optimization problem, to solve a problem of the form~\eqref{pcpprob}, we must compute
\begin{itemize}
\item its optimal value, i.e., $\sup\{\varphi(T^k(x)): x\in\xin,\ k\in\nn\}$;
\item its optimal solutions, i.e., couples $(\overline{x},\overline{k})\in \xin\times \nn$ such that\[\varphi(T^{\overline{k}}(\overline{x}))=   \sup\{\varphi(T^k(x)): x\in\xin,\ k\in\nn\}.\]
\end{itemize}
An optimal solution $(x,k)$ of Problem~\eqref{pcpprob} is thus composed of an initial condition $x$ leading to a peak and its associated date $k$ of peak-realization.

Without loss of generality, we can suppose that $\varphi(0)=0$; otherwise, we change $\varphi$ for $\varphi-\varphi(0)$ and then we add $\varphi(0)$ at the end of the optimization procedure.
This is formulated in Assumption~\ref{assum:varphi}.
\begin{assumption}
\label{assum:varphi}
$\varphi(0)=0$.
\end{assumption}

\subsubsection{A problem on integers}
The optimal value of Problem~\eqref{pcpprob} can be described using a sequence of optimal values. Indeed, we can introduce a static classical optimization problem in which integer $k$ is fixed as follows:
\begin{equation}
\label{pcpstatic}
\tag{$P_k$}
\begin{array}{lcl}
\displaystyle{\Max_{x}} & &\displaystyle{\varphi(T^k(x))}\\
           & \text{ s. t.}& \displaystyle{x\in\xin}
           \end{array} 
\end{equation}
Then, we define the sequence $\nu=(\nu_k)_{k\in\nn}$ for all $k\in\nn$ by:
\begin{equation}
\label{nudef}
\nu_k:=\sup\{\varphi(T^k(x)): x\in\xin\}
\end{equation}
and its supremum:
\begin{equation}
\label{nuopt} 
\nuopt:=\sup\{\nu_k:k\in\nn\}.
\end{equation}
Naturally, if the optimal value of Problem ~\eqref{pcpstatic} is infinite, which is the same as $\nu_k=+\infty$ for some $k\in\nn$, then the optimal value of Problem~\eqref{pcpprob} is infinite. We assume that the following holds: 
\begin{assumption}
\label{mainassum}
For all $k\in\nn,\ \nu_k=\sup\left\{\varphi(T^k(x)) : x\in\xin\right\}$ is finite.
\end{assumption}
A simple situation where Assumption~\ref{mainassum} holds is when $\xin$ is a bounded set, $T$ maps bounded sets to bounded sets, and $\varphi$ is upper semicontinuous. \blue{The existence of an optimal solution for Problem~\eqref{pcpstatic} is actually not mandatory to ensure Assumption~\ref{mainassum}.}

From these notations, $\nuopt$ is equal to the optimal value of Problem~\eqref{pcpprob}, whereas $\nu_k$ is the optimal value of static optimization Problem ~\eqref{pcpstatic}. Solving Problem~\eqref{pcpprob} involves computing $\nuopt$, an integer $\overline{k}$ such that $\nu_{\overline k}=\nuopt$ and an optimal solution $x\in\xin$ of Problem~\eqref{pcpstatic}. In this paper, we focus on the characterization and computation of the integer $\overline{k}$. Static problems of the form~\eqref{pcpstatic} are standard problems in optimization theory. The computation of an optimal solution for these problems relies on an optimization solver when a suitable solver is available. In theory, some special instances of Problem~\eqref{pcpstatic} can be solved numerically, more or less efficiently. For example, when $\varphi \circ T^k$ is concave and $\xin$ is convex, numerous methods exist (see e.g., ~\cite{ben2001lectures,bertsekas2015convex}), or when $\varphi\circ T^k$ is polynomial and $\xin$ is a basic semialgebraic set (see e.g., ~\cite{anjos2011handbook,Lasserre_2015}). In practice, the difficulty of resolving Problem~\eqref{pcpstatic} precedes the optimization procedure and remains in the writing of the optimization problem itself. Indeed, obtaining a closed form for $T^k$ or at least for $T^k(\xin)$ can be an ambitious task. For example, for logistic sequences, closed forms are only known for particular cases (see, e.g.,\cite{gutierrez2010}). Some specific difference equations or systems require intensive efforts to obtain closed forms for $T^k$ (see, e.g.,~\cite{ELSAYED2012378,articleH,stevic2014representation} and references therein). Similarly, obtaining an exact description of $T^k(\xin)$ when $\xin$ is an infinite set incurs a high computational cost. For example, in~\cite{magron2015semidefinite}, the authors approximate the polynomial image of a basic semialgebraic set obtained from the $L^1$ limit of a sequence of semidefinite programs. In some specific cases, $T^k$ can be described explicitly. When $T$ is linear or conjugate to a linear map (i.e., $T=B^{-1}\circ A\circ B$ for a bijection $B$ and a linear map $A$), we can compute $T^k$ using an eigendecomposition if $T$ is diagonalizable (or in Jordan form). When $T$ is linear, $T^k(\xin)$ is always a polytope if $\xin$ is so. Hence, when $T$ is linear and $\xin$ is a polytope, Problem~\eqref{pcpstatic} becomes a linear program. While the theory developed in this paper is independent of the practical constructions of $T^k$ or $T^k(\xin)$, the computational illustration proposed in this paper is limited to an unstable linear system combined with a nonlinear objective function $\varphi$. A more general practical approach will be considered in future works.

The approach developed in~\cite{adje2025maximizationrealsequences} consists in computing an upper bound of the greatest integer maximizer of Problem~\eqref{pcpprob}. In~\cite{adje2025kllyap}, we employed the notion \emph{stopping integer}, which is an integer $K$ such that 
\begin{equation}
\label{eq:stopping}
\nuopt=\max\{\nu_k : k=0,\ldots,K\} \text{ and } \sup_{k> K}\nu_k< \nuopt.
\end{equation}
We define the set of stopping integers and the smallest stopping integer as follows:
\[
\gs_\nu:=\left\{k\in\nn: \max_{0\leq j\leq k} \nu_j>\sup_{j>k} \nu_j\right\}\text{ and } \bigks_u:=\inf\gs_\nu.
\]
The set of stopping integers can be empty, and we can characterize its nonemptiness from the analysis of the limit superior of $\nu$ and thus introduce the set:
 \[
\gsls{\nu}:=\left\{k\in\nn: \nu_k>\limsup_{n\to +\infty} \nu_n\right\}.
\]
Moreover, we define the set of integer maximizers of Problem~\eqref{pcpprob} as
\[
\agmx(\nu):=\{k\in\nn: \nu_k=\nuopt\}.
\]
Recall that $\nuopt$ is finite if and only if the limit superior of $\nu$ is finite or equal to $-\infty$. This basic analysis can be completed using the sets $\gs_\nu$ and $\gsls{\nu}$.
\begin{proposition}
\label{argmaxsimple}
If $\nuopt<+\infty$. The following assertions are made:
\begin{enumerate}
\item $\gs_\nu\neq\emptyset\iff \gsls{\nu}\neq \emptyset$;
\item $\gs_\nu=\emptyset\iff \sup_{k\in\nn} \nu_k=\limsup_{n\to +\infty} \nu_n$;
\item If $\gs_\nu\neq\emptyset$ then $\bigks_\nu=\max\agmx(\nu)$. 
        \end{enumerate}
\end{proposition}
One way to compute an element of $\gsls{\nu}$ is, given a term $\nu_k$, to look for an integer $K\in\nn$ such that
\[
\nu_j<\nu_k,\text{ for all } j\geq K.
\]
Such integers $K$ cannot be smaller than $\bigks_\nu$ and are thus stopping integers. The problem is to convert the term $\nu_k$ into a stopping integer. The most natural sequences allowing this conversion (which is a bijection) between terms and indices are the geometric sequences. However, assuming that $\nu$ is a geometric sequence is extremely restrictive. Instead of directly using geometric sequences, we use bijective deformations of positive convergent geometric sequences as the upper bounds of $\nu$. 

Let $\IR$ be the set of all intervals on $\rr$ \blue{with nonempty interiors}. We define, for a function $f$ on reals, the set of intervals where $f$ is strictly increasing and continuous that is,
\[
\ccdom(f)=\{I\in\IR : f:I\mapsto f(I) \text{ is strictly increasing and continuous }\}. 
\]
If $I\in\ccdom(f)$, $f$ is bijective from $I$ to $f(I)$ and we denote by $\inv{f}{I}$ the inverse of $f$ on $I$. Note that $\inv{f}{I}$ is also strictly increasing and continuous from $f(I)$ to $I$. Finally, we pay particular attention to functions that are strictly increasing and continuous on $[0,1]$ and we write 
\[
\fset:=\{h:\rr\mapsto \rr : [0,1]\in\ccdom(h)\}.
\] 
A convergent geometric sequence is required to complete the upper bound on $\nu$ and we introduce
\[
\funset(\nu):=\{(h,\beta)\in \fset \times (0,1): \nu_k\leq h(\beta^k),\ \forall\, k\in\nn\} 
\]
and
\[ 
\secfunh(\nu):=\{h\in \fset : \exists\, \beta\in (0,1) \text{ s.t. } (h,\beta)\in \funset(\nu)\}.
\]
Finally, we introduce for $h\in\secfunh(\nu)$:
\[ 
\secfunb(\nu,h):=\{\beta\in (0,1) : (h,\beta)\in \funset(\nu)\}.
\]
From $(h,\beta)\in\funset(\nu)$, we can define
\begin{equation}
	\label{eq:fnu}
\Fnu(k,h,\beta):=\dfrac{\ln \inv{h}{[0,1]}(\nu_k)}{\ln(\beta)} \text{ if } \nu_k>h(0)
\end{equation}
for which we have
\[
j\geq \lfloor \Fnu(k,h,\beta)\rfloor \implies \nu_j< \nu_k
\]
The point requiring special consideration is the strict positivity of $\inv{h}{[0,1]}(\nu_k)$ or, equivalently, the fact that $\nu_k>h(0)$. This depends on the choice of $h$ and we define the notion of \emph{usefulness}.
\begin{defi}[Useful Strictly Increasing Continuous Functions]
\label{useful}
Suppose that $\gsls{\nu}\neq \emptyset$. A function $h\in \secfunh(\nu)$ is said to be useful for $\nu$ if the set 
\begin{equation}
\label{residual}
\res(\nu,h):=\{k\in\nn : \nu_k>h(0)\}
\end{equation}
is nonempty. 

By extension, a pair $(h,\beta)\in\funset(\nu)$ is useful for $\nu$ if $h$ is useful for $\nu$.
\end{defi}
\begin{remark}
Note that as $h\in\funset(\nu)$ is continuous, we have $\limsup_{k\to +\infty} \nu_k\leq h(0)$ and thus $\res(\nu,h)\subset\gsls{\nu}$.
\end{remark}
One of the main results of~\cite{adje2025maximizationrealsequences} is the existence of a useful element in $\funset(\nu)$ when $\gsls{\nu}$ is nonempty.
\begin{theorem}
\label{mainexistence}
The following statements are true.
\begin{enumerate}
\item $\nuopt<+\infty$ if and only if $\funset(\nu)\neq \emptyset$;
\item $\gsls{\nu}\neq \emptyset$ if and only if there exists $(h,\beta)\in\secfunh(\nu)$ that is useful for $u$.
\end{enumerate}
\end{theorem}
\subsubsection{Problem statement}
Even if we have proved that $\funset(\nu)$ is nonempty, the proof furnished in~\cite{adje2025maximizationrealsequences} exhibits an affine function $h\in\secfunh(\nu)$ for which the slope and y-intercept depend on $\bigks_\nu$ and $\nuopt$. This leads to considering alternative approaches to construct an element in $\funset(\nu)$.
\begin{problem}
\label{computehpb}
Construct a couple $(h,\beta)\in \funset(\nu)$ useful for 
$\nu$ when $\gsls{\nu}\neq \emptyset$.
\end{problem}
In~\cite{adje2025kllyap}, we proposed to build $(h,\beta)\in\funset(\nu)$ from a $\klcls$ stability certificate or a Lyapunov function. Problems of the form~\eqref{pcpprob} can have a finite optimal value and are such that $\gsls{\nu}$ is nonempty without admitting a classical $\klcls$ stability certificate or a classical Lyapunov function. This is the case for the example developed in Section~\ref{sec:leslieexample}. To go further Problem~\ref{computehpb}, we are interested in a full characterization of the elements of $\funset(\nu)$ that subsumes the previous results based on classical stability tools.
Thus, the main contribution of this paper is the characterization of the elements of $\funset(\nu)$ from an extension of $\klcls$ and the Lyapunov functions. We will prove that solving Problem~\ref{computehpb} is equivalent to constructing a weaker version of $\klcls$ certificate and a weaker version of the Lyapunov functions. We will prove that we can construct an element $(h,\beta)\in\funset(\nu)$ from any extended $\klcls$ certificate and that an extended $\klcls$ certificate can be constructed from any $(h,\beta)\in\funset(\nu)$. The same results hold for our extended class of Lyapunov functions.

\subsection{Classical \protect$\klcls$ and Lyapunov approaches}

	The use of $\klcls$ and Lyapunov functions outside the scope of stability theory seems to be atypical and we bring some intuitions about the use in the context of solving Problem~\ref{computehpb}. In fact, classical $\klcls$ and Lyapunov functions furnish more or less naturally a function in $\fset$ and $\beta\in (0,1)$. To obtain an element of $\funset(\nu)$ from a $\klcls$ certificate or a Lyapunov function requires a restrictive inequality or the continuity of $\varphi$. From the positivity of $\klcls$ and Lyapunov functions, usefulness holds if and only if the sequence $\nu$ has a positive term. All the details on the approach can be found in~\cite{adje2025kllyap}.

\subsubsection{The classical \protect$\klcls$ approach}

First, recall that a $\klcls$ function $\gamma:\rr_+\times \rr_+\mapsto \rr_+$ is such that
\begin{enumerate}
\item For all $t\in\rr_+$, $s\mapsto \gamma(s,t)\in\kclsi$ i.e. the set of strictly increasing continuous functions $\rr_+\to \rr_+$ that are zero at zero and whose the limit at $+\infty$ is $+\infty$;
\item For all $s\in\rr_+$, $t\mapsto \gamma(s,t)\in\lcls$ i.e. the set of decreasing continuous functions $\rr_+\to \rr_+$ whose the limit at $+\infty$ is $0$;  
\end{enumerate} 
The dynamical system $(\xin,T)$ admits a $\klcls$ certificate if and only if there exists $\gamma\in\klcls$ such that for all $x\in\xin$ and all $k\in\nn$
\begin{equation}
\label{eq:kclsclassic}
\norm{T^k(x)}\leq \gamma(\norm{x},k)
\end{equation}
If the inequality~\eqref{eq:kclsclassic} holds, the famous result of Sontag~\cite[Proposition 7]{SONTAG199893} directly provides a potential pair $(h,\beta)$. Recall that the result asserts that for any $\gamma\in\klcls$  there exists $\theta_1,\theta_2\in\kclsi$ verifying $\gamma(s,t)\leq \theta_1(\theta(s)e^{-t})$ for all $(s,t)\in\rr_+\times \rr_+$. Then, if $\xin$ is bounded, we define
\begin{equation}
\label{eq:alasontag}
h:s \to \theta_1(\theta_2(s\sup_{x\in\xin} \norm{x}))\in\fset.
\end{equation}
The difficulty is to exploit the inequality~\eqref{eq:kclsclassic} to obtain $(h,e^{-1})\in\funset(\nu)$. This requires an additional assumption on $\varphi$. For example, if $\varphi$ is continuous, we obtain the desired upper bound from the famous Hahn/Khalil result~\cite{hahn1967stability} and \cite[Lemma 4.3/Appendix C.4]{khalil2002nonlinear}. This result asserts that for any positive definite\footnote{a function is positive definite if it is strictly positive except at 0 where it is null} continuous coercive\footnote{a function $f$ is said to be coercive (or radially unbounded) if $\lim_{\norm{x}\to +\infty} f(x)=+\infty$} function $f$ there exist $\overline{\alpha_1},\overline{\alpha_2}\in\kclsi$ satisfying:
\begin{equation}
\label{khalil}
\overline{\alpha_1}\circ \norm{\cdot}\leq f\leq \overline{\alpha_2}\circ \norm{\cdot}.
\end{equation}
From Assumption~\ref{assum:varphi} $\max\{\varphi,\norm{\cdot}\}$ is positive definite, coercive, and continuous if $\varphi$ is. Hence
\begin{equation}
\label{eq:hfromkl}
(\alpha_\varphi\circ h,e^{-1})\in\funset(\nu) \text{ where } \alpha_\varphi\in\kclsi\text{ is such that }
\varphi\leq \max\{\varphi,\norm{\cdot}\}\leq \alpha_\varphi\circ \norm{\cdot}.
\end{equation}
We can relax the continuity assumption on $\varphi$ and instead imposing that, for some $\psi:\rd\to \rr_+$ such that $\sup_{x\in\xin}\psi(x)<+\infty$
\begin{equation}
\label{eq:extendedkl}
\varphi(T^k(x))\leq \gamma(\psi(x),k),\text{ for all } x\in\xin \text{ and } k\in\nn.
\end{equation}
This consists in replacing in the inequality~\eqref{eq:kclsclassic} the norm on the left-hand side by $\varphi$ and the norm on the right-hand side by a nonnegative function $\psi$ with a strictly positive supremum on $\xin$. In this case 
\begin{equation}
\label{eq:hfromextkl}
(s\mapsto \theta_1(\theta_2(s\sup_{x\in\xin} \psi(x))),e^{-1})\in\funset(\nu).
\end{equation}
Finally, as stated above, usefulness in the context of $\klcls$ certificate is equivalent to the existence of a strictly positive term for $\nu$.    
\subsubsection{The Lyapunov approach}
In~\cite{adje13052025}, we consider Lyapunov functions to solve Problem~\eqref{pcpprob}. In~\cite{adje13052025}, the dynamical system was a stable affine system, and the objective function $\varphi$ was a quadratic function. Recently, in~\cite{adje2025kllyap}, we generalized the approach for all discrete-time systems that admit a Lyapunov function. Recall that a function $V:\rd \to \rr_+$ is a (global) Lyapunov function if and only if
\begin{equation}
\label{eq:positivitylyap}
\alpha_1 \circ \norm{\cdot}\leq V\leq \alpha_2 \circ \norm{\cdot} \text{ for some }\alpha_1,\alpha_2\in\kclsi
\end{equation}
and 
\begin{equation}
\label{eq:decreaselyap}
V\circ T\leq \lambda V \text{ and  for some } \lambda\in (0,1)
\end{equation}
Note that~\eqref{eq:positivitylyap} implies that $V$ is positive definite.

Inequality~\eqref{eq:decreaselyap} implies that 
\begin{equation}
\label{eq:lyaptoh}
V(T^k(x))\leq \lambda^k V(x) \text{ for all }x\in\xin.
\end{equation}
If $\xin$ is bounded and not reduced to $\{0\}$, then $\sup_{x\in\xin} V(x)$ is finite and strictly positive. Hence, similarly to the $\klcls$ approach, we define
\[
h:s\mapsto s \sup_{x\in\xin} V(x)\in \fset.
\] 
To use $h$ as an element of $\secfunh(\nu)$, we need to make additional assumptions. We can simply impose that $\varphi\leq V$ and from~\eqref{eq:lyaptoh}
\begin{equation}
\label{eq:hfromlyap}
(h,\lambda)\in\funset(\nu).
\end{equation} 
If $\varphi$ is continuous, the comparison between $\varphi$ and $V$ relies on the Hahn-Khalil result described in Eq. {khalil}. If $\varphi$ is continuous then, using the same notation as in the $\klcls$ approach
\begin{equation}
\label{eq:hfromextlyap}
(\alpha_\varphi \circ \alpha_1^{-1}\circ h,\lambda)\in\funset(\nu)
\end{equation}
where $\alpha_1$ is the $\kclsi$ function appearing in~\eqref{eq:positivitylyap}.

In ~\cite{adje2025kllyap}, we slightly generalized the Lyapunov approach to allow local Lyapunov functions. Local means that the inequalities depicted in~\eqref{eq:positivitylyap} and~\eqref{eq:decreaselyap} hold only on an (forward) invariant set for the dynamics $T$. Moreover, for $h\in\secfunh(\nu)$ and $k\in\res(\nu,h)$, the function $\beta\mapsto \Fnu(k,h,\beta)$ defined in Equation~\eqref{eq:fnu} increases. Therefore, we should consider the smallest element of $\beta\in\secfunb(\nu,h)$. For a Lyapunov function $V$ for $T$, we should consider the \emph{operator ratio}
\begin{equation}
	\label{eq:opratio}
N_T(V):=\sup_{x\neq 0} \dfrac{V(T(x))}{V(x)}
\end{equation}	 
as the scalar $\lambda$ in Equations~\eqref{eq:hfromlyap} and~\eqref{eq:hfromextlyap}. The study of the suitable properties of operator ratios in our context was developed in~\cite{adje2025kllyap}. 
\subsection{Problem statement continues}
In Figure~\ref{figintro}, we insist on the fact that the existence of a $\klcls$ certificate or a Lyapunov function combined with an additional assumption on $\varphi$ only implies the existence of a pair $(h,\beta)\in\funset(\nu)$. In this paper, we exhibit two proper classes of functions ensuring direct and converse theorems for the condition $\gsls{\nu}\neq\emptyset$. The first class extends the $\klcls$ class of functions. We obtain the equivalence between the existence of a weaker $\klcls$ certificate and $(h,\beta)\in\funset(\nu)$. The second class extends the set of Lyapunov functions for which the existence of a weaker Lyapunov function is equivalent to the existence of $(h,\beta)\in\funset(\nu)$. The constructed pair $(h,\beta)$ is useful when $\gsls{\nu}\neq \emptyset$.
\begin{figure}[h]
\fbox{
\begin{minipage}{0.97\textwidth}
\begin{center}
\begin{tikzpicture}
\draw (-5.8,1) ellipse (4.5cm and 1.5cm);
\begin{scope}
\node[draw,text width=3cm,text centered] at (-8,1) {Lyapunov\\ or\\ $\klcls$ certificate};
\draw (-6,1) node{+};
\node[draw,text width=3cm,text centered] at (-4,1) {$\varphi$ smaller than the stability certificate\\ or \\ $\varphi$ continuous};
\end{scope}
\draw[thick,double,->] (-1.5,0.5) -- (1,-1);
\node[draw,text width=3cm,text centered] at (2.8,-1.8) {$\exists (h,\beta)\in\funset(\nu)$ useful for $\nu$};
\draw (-1.5,-1.1) node[below]{\cite[Th.2]{adje2025maximizationrealsequences}};
\draw (-0.2,0) node[right]{\cite[Th.4,Th.5,Th.6,Th.7]{adje2025kllyap}};
\draw[thick,double,<->] (-4,-1.8) -- (1,-1.8);
\node[draw,text width=2cm,text centered] at (-5.3,-1.78) {$\gsls{\nu}\neq \emptyset$};
\draw[dashed] (-10,-5.4) rectangle (1,-2.4);
\begin{scope}
\node[draw,text width=3cm,text centered] at (-8,-4) {weaker Lyapunov\\ functions\\ / \\ weaker $\klcls$\\ functions};
\draw[thick,double,<->] (-6.3,-4.2) -- (0.9,-2.7);
\node[text width=3cm,text centered] at (-0.1,-5.1) {in this paper};
\end{scope}
\end{tikzpicture}
\end{center}
\end{minipage}
}
\caption{Previous results obtained using $\klcls$ and Lyapunov functions and the results developed in this paper.}
\label{figintro}
\end{figure}
   
\section{A running example from Leslie population evolution models}
\label{sec:leslieexample}
In this paper, we illustrate the new tools that we develop using a running example. This example comes from population ecology and is a model for studying the evolution of a population according to some period. The model studied here is called a \emph{Leslie model} and was introduced in 40s~\cite{6201d988-ef79-3e59-8ac9-055f74401218,9d8eed9f-80b1-3385-89ce-e466360a8bb4}.  
\subsection{Introduction}
Leslie population models analyze the growth of the  population of a species structured in age classes. A class represents a stage of development of the species (juvenile, sub-adult, adult, etc.). From this structuration, the population is modeled as a vector of the sizes of the population of each age class. The evolution of the population is computed from a matrix called the \emph{Leslie matrix}. The vector of the population sizes at a given year (or another time step) is the image of the vector of the population sizes of the previous year by the Leslie matrix. The coefficients of the first row are the fertility rates associated with each age class. Hence, those coefficients compute the number of individuals in the first age class that can be generated from each age class. The next rows attribute the survival rates to each age class. This rate aims to compute the number of individuals that pass from the current age class to the next. Hence, the survival rate belongs to $[0,1]$. We also note that this survival rate is the unique nonzero coefficient of the row. 

Usually, the questions raised about Leslie population models are about the asymptotic behavior of the population, the maximal growth rates of the population of each class, and the asymptotic proportions of the population of each age class with respect to the total population. The question concerning the proportions can be addressed in a different way: what is the proportion of the juvenile (the first class) with respect to the others? In this paper, we are interested in the maximum proportion.    
To illustrate our approach in the simplest way, we consider two age-class Leslie population models. This means that we only consider the juvenile and adult populations, and we are interested in the maximum proportion of juveniles over adults. 

\subsection{Peak computation problem description}
We consider a matrix $L$ parameterized by scalars $a,b$ and $c$, defined as follows:
\[
L=\begin{pmatrix}
a & b\\
c & 0
\end{pmatrix}
\]
The scalar $a$ represents the number of juveniles that can be generated per period by one juvenile, whereas $b$ represents the number of juveniles that can be generated per period by one adult. The scalar $c$ measures the proportion of juveniles that become adults at the end of the period. A value of 0 means that adults disappear after the period considered. We suppose that one juvenile and one adult can produce strictly more than one juvenile. Moreover, we recall that $c$ encodes the percentage of juveniles that become adults and thus cannot exceed one. We suppose that some juvenile individuals survive, and so $c$ cannot be null. More formally, we assume that
\begin{equation}
	\label{eq:condlesliecoef}
a\geq 1,\ b>1 \text{ and } c\in (0,1)
\end{equation}
To deal with several possible initial configurations simultaneously, we define
\[
\xinles:=[l_1,u_1] \times [l_2,u_2]\text{ where } 0<l_1<u_1 \text{ and } 0<l_2<u_2.
\]
In the sequel, we use the following notation for the vertices of $\xin$:

\begin{equation}
	\xll :=(l_1,l_2)^\intercal,\ \xul:=(u_1,l_2)^\intercal,\ \xlu:=(l_1,u_2)^\intercal\text{ and } \xuu:=(u_1,u_2)^\intercal
\end{equation}
As the proportion of the juvenile over the adults is concerned, our objective function is 
defined as follows:
\[
\philes: \rr^2\ni x=(x_1,x_2)\mapsto\left\{ \begin{array}{cr}\dfrac{x}{y} & \text{ if } y\neq 0\\
\\
 0 & \text{ if } y=0
\end{array}\right.
\] 
Let us introduce $\pi_1:\rr^2\ni x\mapsto x_1$ and $\pi_2:\rr^2\ni x\mapsto x_2$. Then we obtain the following peak computation problem  
\begin{equation}
\label{eq:runningdpcp}
\tag{$DPCP_{\ell}$}
\begin{array}{rc}
\displaystyle{\Max_{x,k}}& \philes(L^k x)\\
             s.t. &\left\{\begin{array}{l}
                    x\in\xinles\\
                    k\in\nn
                    \end{array}
                    \right.
\end{array}
=
\begin{array}{rc}
\displaystyle{\Max_{x\in\rr^2,k}}& \dfrac{\pi_1(L^k x)}{\pi_2(L^k x)}\\
 \\
             s.t. &\left\{\begin{array}{l}
                    x\in\xinles\\
                    k\in\nn
                    \end{array}
                    \right.
\end{array}
\end{equation}
Then we define:
\[
\omeles=(\omeles_k)_{k\in\nn} \text{ where }\omeles_k=\sup_{x\in\xin} \philes(L^k x)\text{ and }\omelesopt=\sup_{k\in\nn} \omeles_k.
\]

\subsection{A rapid preliminary analysis}
\subsubsection{Asymptotic divergence}
\label{subsubsec:rapidanal}
The eigenvalues of $L$ are
\[
\lambda_1:=\dfrac{a-\sqrt{a^2+4bc}}{2} \text{ and } \lambda_2:=\dfrac{a+\sqrt{a^2+4bc}}{2}.
\]
From~\eqref{eq:condlesliecoef} it is easy to see $\lambda_2>1$, $\lambda_1<0$ and $\lambda_1/\lambda_2\in (-1,0)$. The system is thus not asymptotically stable and does not admit a $\klcls$ upper bound or a classical Lyapunov function. 

Moreover, for any starting point $x\in\xinles$, $(\pi_1(L^k x))_{k\in\nn}$ and $(\pi_2(L^k x))_{k\in\nn}$ are strictly increasing (from $k=2$ for $\pi_2$) and both tend to $+\infty$ as $k$ goes to $+\infty$. We give details about this in the Appendix.
\subsubsection{Existence of a stopping integer}
We introduce the following two auxiliary functions:
\begin{equation}
	\label{eq:abdef}
\afct:\rr^2\ni x\mapsto \lambda_2x_2-c x_1 \text{ and } \bfct:\rr^2\ni x\mapsto cx_1-\lambda_1 x_2\enspace. 
\end{equation}
Note that for all $x\in (\rr_+^*)^2$, for all $k\in\nn$, 
\begin{equation}
\label{eq:abprop}
\afct(x)\left(\dfrac{\lambda_1}{\lambda_2}\right)^k+\bfct(x)=cx_1\left(1-\left(\dfrac{\lambda_1}{\lambda_2}\right)^k\right)+\lambda_1 x_2\left(\left(\dfrac{\lambda_1}{\lambda_2}\right)^{k-1}-1\right)
\end{equation}
and thus is strictly positive as $\lambda_1<0$ and $\lambda_1/\lambda_2\in (-1,0)$.

Now, we show that $\omelesopt$ is finite and $\gsls{\omeles}\neq \emptyset$. Using an eigendecomposition of $L$, $\afct$ and $\bfct$, we have for all $x\in\xinles$ and all $k\in\nn$
\begin{equation}
\label{eq:eigenvarphi}
\philes(L^k x)=\dfrac{\lambda_1^{k+1}\afct(x)+\lambda_2^{k+1}\bfct(x)}{c\left(\lambda_1^k\afct(x)+\lambda_2^k\bfct(x)\right)}
=\dfrac{\lambda_2}{c}\dfrac{\afct(x)\left(\dfrac{\lambda_1}{\lambda_2}\right)^{k+1}+\bfct(x)}{\afct(x)\left(\dfrac{\lambda_1}{\lambda_2}\right)^k+\bfct(x)}.
\end{equation}
From~\eqref{eq:abprop}, both the numerator and denominator are strictly positive. 

As $\xinles$ is the Cartesian product of intervals, we can compute the maximum of $\philes \circ L^k$ from the study of the sign of the partial derivatives. We conclude that for all $k\in\nn$
\begin{equation}
	\label{eq:nuksolleslie}
\omeles_{2k}=\max_{x\in \xinles} \philes(L^{2k} x)=\philes(L^{2k} \xul)\text{ and }\omeles_{2k+1}=\max_{x\in \xinles} \philes(L^{2k+1} x)=\philes(L^{2k+1}\xlu).
\end{equation}
From~\eqref{eq:eigenvarphi}, we observe that for all $x\in(\rr_+^*)^2$, $\lim_{k\to +\infty}\philes(L^k x)=\lambda_2/c$. Hence, $\limsup_{k\to +\infty} \omeles_k=\lim_{k\to +\infty} \omeles_k=\lambda_2/c<+\infty$. We conclude that $\omelesopt<+\infty$.

To show that $\gsls{\omeles}\neq \emptyset$, we have to check whether $\omeles_k>\lambda_2/c$ for some $k\in\nn$. This can be deduced from the properties of $\afct$. Let $x^*\in\xinles$. If $\afct(x^*)<0$, we get $\omeles_0\geq \philes(x^*)>\lambda_2/c$ and $0\in\gsls{\omeles}$. Now, we observe that $\afct\circ L =\lambda_1\afct $ on $\rr\times\rr_+^*$. Hence, if $\afct(x^*)>0$, we have $\afct(L x^*)<0$, $\omeles_1\geq \philes(L x^*)>\lambda_2/c$ and thus $1\in\gsls{\omeles}$. If $\afct(x^*)=0$. From the equality $\afct\circ L=\lambda_1\afct$, we have for all $k\in\nn$, $\afct(L^k x^*)=0$. In this case, $\philes(L^k x^*)=\lambda_2/c$ for all $k\in\nn$. To conclude that $\gsls{\omeles}\neq \emptyset$, it suffices to show that $\xinles$ cannot be included in $\afct^{-1}(\{0\})$. By definition, $\xinles$ has a nonempty interior, whereas $\afct^{-1}(\{0\})$ has an empty one; hence $\afct(x)\neq 0$ for some $x\in\xinles$.
Finally, to compute $\omelesopt$, the initial vectors $x$ are removed such that $\afct(x)=0$ and we introduce
\begin{equation}
\label{eq:newinitial}
\xles:=\xinles\cap\{x\in\rr^2 : x_1/x_2\neq \lambda_2/c\}=\xin\cap \{x\in\rr^2: \alpha(x)\neq 0\}.
\end{equation}
\subsubsection{An alternative simpler solution}
\label{subsubsec:alternativesol}
We propose a simple alternative approach to solve the peak computation problem~\eqref{eq:runningdpcp} as follows. We simply use the fact that for all $k\in\nn$ and all $x\in\xinles$,
\[
\philes(L^{2k+2} x)-\philes(L^{2k} x)=\dfrac{\lambda_2}{c}\dfrac{\left(\dfrac{\lambda_1}{\lambda_2}\right)^{2k} \afct(x)\bfct(x)\left(1+\left(\dfrac{\lambda_1}{\lambda_2}\right)^{3}-\dfrac{\lambda_1}{\lambda_2}-\left(\dfrac{\lambda_1}{\lambda_2}\right)^{2}\right)}{\left(\left(\dfrac{\lambda_1}{\lambda_2}\right)^{2k+1}\afct(x)+\bfct(x)\right)\left(\left(\dfrac{\lambda_1}{\lambda_2}\right)^{2k}\afct(x)+\bfct(x)\right)}.
\]
From~\eqref{eq:abprop}, the denominator of $\philes(L^{2k+2} x)-\philes(L^{2k} x)$ is still strictly positive. The numerator is strictly negative when $\afct(x)$ is positive. When $\afct(x)$ is strictly positive, we consider instead $\philes(L^{2k+3} x)-\philes(L^{2k+1}x)=\philes(L^{2k+2} Lx)-\philes(L^{2k}L x)$ which is strictly negative from the equality $\afct\circ L=\lambda_1 \afct$. Then we conclude that, if $\afct(x)<0$, $(\philes(L^{2k} x))_{k\in\nn}$ is strictly decreasing and, if $\afct(x)>0$, $(\philes(L^{2k+1} x))_{k\in\nn}$ is strictly decreasing. Finally, this leads to $\omelesopt=\max\{\omeles_0,\omeles_1\}$.

\subsection{The illustration of our work} 
\label{subsec:numericaldata}
For the running example, we will construct
\begin{enumerate}
\item a weaker $\klcls$ certificate;
\item a weaker Lyapunov function.
\end{enumerate} 
for the discrete-time peak computation problem~\eqref{eq:runningdpcp}. These constructions will be completed independently of the results obtained in Subsubsection~\ref{subsubsec:alternativesol}. To be more explicit, we will apply our techniques to explicit numerical data. The following $2\times 2$ Leslie matrix has been found in~\cite{RODIN1988299}.
\begin{equation}
	\label{eq:nummatrix}
L:=\begin{pmatrix}
	1 & 4\\
	0.5 & 0
\end{pmatrix}.  
\end{equation}  
From this matrix, we obtain
\[
a=:1,\ b:=4,\ c:=0.5, \lambda_1:=-1 \text{ and } \lambda_2:=2.
\]
Hence, $\lambda_2/c=4$, $(\rr_+^*)^2\ni x\mapsto \afct(x):=2 x_2-0.5x_1$ and $(\rr_+^*)^2\ni x\mapsto \bfct(x):=0.5 x_1+x_2$.

In our approach, the choice of the initial condition set is important. We propose a specific initial conditions set for which $\afct$ takes positive and negative values.
We thus define:
\[
\xin:=[10;40]\times [5;50]
\]
that leads to
\[
\xll:=(10,5)^\intercal,\ \xlu:=(10,50)^\intercal,\ \xul:=(40,5)^\intercal\text{ and } \xuu:=(40,50)^\intercal.
\]

\section{Extended \protect$\klcls$ Theory for Discrete-Time Systems Maximization}
\label{klsection}

In this paper, we are interested in the resolution of Problem~\eqref{pcpprob}, for which the underlying dynamical system $x_{k+1}=T(x_k)$ is not necessarily asymptotically stable, as shown in the example provided in Section~\ref{sec:leslieexample}. However, Inequality~\eqref{eq:kclsclassic} holds if and only if the system $x_{k+1}=T(x_k), x_0\in\rd$ is said to be {\it uniformly global asymptotic stability} (UGAS) (see, for example~\cite[Prop. 1]{NESIC20041025} or~\cite[Th. 2]{doi:10.1080/10236190902817844}). On the other hand,  Inequality~\eqref{eq:extendedkl} provides useful information on the asymptotic behavior of the system: the limit superior of $\nu_k$ is smaller than $\inf_{k\in\nn} \gamma(\sup_{x\in\xin} \psi(x),k)$ and thus finite. However, Inequality~\eqref{eq:extendedkl} forces our sequence $\nu$ to have a positive term to obtain the usefulness of the pair proposed in~\eqref{eq:hfromextkl}. Therefore, Inequality~\eqref{eq:extendedkl} constitutes our starting point. Thus, we define the class of functions $\klgen$ that extends the class of functions $\klcls$. The class $\klgen$ permits the construction of a useful element $(h,\beta)\in\funset(\nu)$ when $\gsls{\nu}\neq \emptyset$ and the analog of Inequality~\eqref{eq:extendedkl} is verified by replacing $\gamma$ with a $\klgen$ function. Contrary to the classical $\klcls$ certificate approach, we prove a converse theorem: if $\gsls{\nu}\neq \emptyset$, then the analog of Inequality~\eqref{eq:extendedkl} holds for some $\gamma\in\klgen$. We also provide an extension to the famous Sontag's result~\cite[Proposition 7]{SONTAG199893} recalled above for the $\klgen$ class. It allows to obtain an element $h\in\funset(\nu)$ such that $e^{-1}\in\secfunb(\nu,h)$ when the analogue of~\eqref{eq:extendedkl} is satisfied. 

We recall that the data are as follows: 
\begin{itemize}
\item an initial conditions set $\xin\subseteq \rd$;
\item a nonzero map $T:\rd\mapsto \rd$;
\item a function $\varphi:\rd\mapsto \rr$ such that $\varphi(0)=0$. 
\end{itemize}
We recall that we define $\nuopt=\sup\{\nu_k:k\in\nn\}$ where $\nu_k=\sup\{\varphi(T^k(x):x\in\xin\}$. Naturally, we suppose that Assumption~\ref{mainassum} holds. To simplify the notation of the maximal value on $\xin$ of a function $f:\rd\mapsto \rr\cup\{+\infty\}$, we introduce the following new notation:
\[
\osp{f}:=\sup\{f(x) : x\in\xin\}
\]
and the set of functions that are bounded from above on $\xin$:
\[
\fin:=\{f:\rd\mapsto \rr\cup\{+\infty\} : \osp{f}<+\infty\}.
\]
With those notations, $\nu_k=\osp{(\varphi\circ T^k)}$ and by following Assumption~\ref{mainassum}, we have for all $k\in\nn$, $\varphi\circ T^k \in\fin$.     

\subsection{Revisiting \protect$\mathcal{KL}$ functions : the class \protect$\klgen$}
\label{newkcls}

One of the limitations of using classical $\klcls$ functions in~\eqref{eq:extendedkl} is the positivity of $\klcls$ functions. To ensure the usefulness of the element $(h,\beta)$ of Equation~\eqref{eq:hfromextkl}, our sequence $\nu$ must have a positive term. Then, we should relax the positivity on the state-variable part of $\klcls$ functions (and not on the time variable). We also relax the strict increase (or decrease) and continuity required to be a $\mathcal{KL}$ function. 

\begin{defi}[The Class $\klgen$]
A function $\gamma:\rr\times \rr_+\mapsto \rr$ belongs to $\klgen$ if
\begin{itemize}
  \item for all $t\in\rr_+$, $s\mapsto \gamma(s,t)$ is \red{increasing};
 \item for all $s\in\rr$, $t\mapsto \gamma(s,t)$ is \red{decreasing}.
\end{itemize}
\end{defi}

\begin{example}[A pertinent $\klgen$ for Leslie models]
	For the example introduced in Section~\ref{sec:leslieexample}, we compute a $\klgen$ upper bound for $(\xles,L,\philes)$. In this example, we construct a family of $\klgen$ functions for which some of them will be $\klgen$ upper bounds for $(\xles,L,\philes)$. For this construction, we introduce the following auxiliary function. 
	\begin{equation}
		\label{eq:Hfct}
	H:\{(x,y)\in\rr^2: x+y\neq 0\}\ni (u,v)\mapsto \dfrac{\lambda_2}{c} \dfrac{\dfrac{\lambda_1}{\lambda_2}v+u}{v+u}.
	\end{equation}
	We observe that we have for all $x\in \xles$ (defined in~\eqref{eq:newinitial}) and all $k\in\nn$
	\begin{equation}
		\label{eq:varphiH}
		\philes(L^k x)=\dfrac{\lambda_1^{k+1}\afct(x)+\lambda_2^{k+1}\bfct(x)}{c\left(\lambda_1^k\afct(x)+\lambda_2^k\bfct(x)\right)}=
		\dfrac{\lambda_2}{c}\dfrac{\dfrac{\lambda_1}{\lambda_2} \left(\dfrac{\lambda_1}{\lambda_2}\right)^k+\dfrac{\bfct(x)}{\afct(x)}}{\left(\dfrac{\lambda_1}{\lambda_2}\right)^k+\dfrac{\bfct(x)}{\afct(x)}}=H\left(\dfrac{\bfct(x)}{\afct(x)},\left(\dfrac{\lambda_1}{\lambda_2}\right)^k\right).
	\end{equation}
	From~\eqref{eq:varphiH}, we propose to construct a $\klgen$ function $\rr\times \rr_+\ni (s,t)\mapsto \gamma(s,t)$ from $H$. Indeed, the final purpose of the constructed $\klgen$ function is to provide an upper bound of $\philes\circ L^k$ on $\xles$. From the analysis of the sign of the partial derivatives of $H$, we remark that
	\begin{itemize}
		\item for all $v>0$ (resp. $v<0$), $u\mapsto H(u,v)$ strictly increases on $(-\infty,-v)$ and on $(-v,+\infty)$ (resp. strictly decreases on $(-\infty,-v)$ and on $(-v,+\infty)$);
		\item for all $u<0$ (resp. $u>0$), $v\mapsto H(u,v)$ strictly increases on $(-\infty,-u)$ and on $(-u,+\infty)$ (resp. strictly decreases on $(-\infty,-u)$ and on $(-u,+\infty)$). 
	\end{itemize} 
	From these monotonicity properties, $H$ is a candidate to be a $\klgen$ function on negative values of the state-variable $s$ and with positive second arguments. Therefore, we will consider $|\lambda_1/\lambda_2|^t$ rather than $(\lambda_1/\lambda_2)^t$ as second argument in $H$. For positive values $s$, the $\klgen$ function $\gamma(s,t)$ cannot be equal to $H(s,|\lambda_1/\lambda_2|^t)$ and we will propose a strictly increasing affine function. Moreover, $H(s,|\lambda_1/\lambda_2|^t)$ is not defined when $s=-|\lambda_1/\lambda_2|^t$ and we have to avoid those points. Finally, the main objective of this $\klgen$ function is to provide an upper bound over $\philes\circ L^k$ on $\xles$. Unfortunately, for all $s\in (-|\lambda_1/\lambda_2|^t,0)$, $H(s,|\lambda_1/\lambda_2|^t)$ is strictly negative and thus does not provide an upper bound for $\philes\circ L^k$ which is strictly positive on $\xles$. We also remark that $s\mapsto H(s,|\lambda_1/\lambda_2|^t)$ is not bounded from above on $(-\infty,-|\lambda_1/\lambda_2|^t)$. Consequently, we artificially bound it, allowing only values $s$ smaller than $-|\lambda_1/\lambda_2|^t-\varepsilon$ for a given strictly positive $\varepsilon$. 
	We will put constraints over $\varepsilon$ when we impose the $\klgen$ function to be an upper bound of $\philes\circ L^k$ on $\xles$. Finally, given $\varepsilon>0$, $H$ is exploited at $(s,|\lambda_1/\lambda_2|^t)$ only when $s<-|\lambda_1/\lambda_2|^t-\varepsilon$. Otherwise, we use an affine function increasing in $s$ and in $|\lambda_1/\lambda_2|^t$ proved to be greater than $H(-|\lambda_1/\lambda_2|^t-\varepsilon,|\lambda_1/\lambda_2|^t)$. Finally, we get a family of $\klgen$ functions defined for all $\varepsilon>0$ as follows,
	\begin{equation}
		\label{eq:klgenepsdef}
		\gamma_\varepsilon:\rr\times\rr_+\ni (s,t)\mapsto \left\{\begin{array}{lr}
			H\left(s,\left|\dfrac{\lambda_1}{\lambda_2}\right|^t\right) & \text{ if } s\leq -\left|\dfrac{\lambda_1}{\lambda_2}\right|^t-\varepsilon\\
			\left(s+\left|\dfrac{\lambda_1}{\lambda_2}\right|^t+\varepsilon\right)+H\left(-\left|\dfrac{\lambda_1}{\lambda_2}\right|^t-\varepsilon,\left|\dfrac{\lambda_1}{\lambda_2}\right|^t\right) & \text{ otherwise }  
		\end{array}\right.
	\end{equation}
	We prove in the Appendix that for all $\varepsilon>0$, $\gamma_\varepsilon\in\klgen$. 
	
	Note that we could replace $H(-|\lambda_1/\lambda_2|^t-\varepsilon,|\lambda_1/\lambda_2|^t)$ by a strictly positive function increasing with respect to $|\lambda_1/\lambda_2|^t$ and greater than $H(-|\lambda_1/\lambda_2|^t-\varepsilon,|\lambda_1/\lambda_2|^t$. The current choice of affine function makes $\gamma_\varepsilon$ continuous, which is not mandatory to belong to $\klgen$. 
	
	Applying ~\eqref{eq:klgenepsdef} to the numerical data of Subsection~\ref{subsec:numericaldata}, we have
	\[
	\gamma_\varepsilon:\rr\times\rr_+\ni (s,t)\mapsto \left\{\begin{array}{lr}
		4 \dfrac{0.5^{t+1}+s}{0.5^t+s} & \text{ if } s\leq -(0.5)^t-\varepsilon\\
		\\
		s+0.5^t(1+6\varepsilon^{-1})+\varepsilon+4 & \text{ if } s>-0.5^t-\varepsilon
	\end{array}
	\right.
	\]
	\qed
\end{example}
\begin{defi}[$\klgen$-$\fin$ upper bounds]
A pair $(\gamma,\theta)\in\klgen\times \fin$ is a $\klgen$-$\fin$ upper bound for $(\xin,T,\varphi)$ if and only if 
\[
\forall\, k\in\nn,\ \forall\, x\in\xin,\ \varphi(T^k(x))\leq \gamma(\theta(x),k) 
\]
Moreover, $(\gamma,\theta)$ is said to be useful for $\nu$ such that $\gsls{\nu}\neq \emptyset$ if and only if the set
\[
\left\{k\in \gsls{\nu} : \inf_{t\geq 0} \gamma(\osp{\theta},t)<\nu_k\right\}
\]
is nonempty.
\end{defi}
\begin{example}[A useful $\klgen$-$\fini{\xles}$ upper bound for $(\xles,L,\philes)$]
	\label{ex:klgenfinleslie}
	We now propose, as $\gsls{\omeles}\neq \emptyset$, a useful $\klgen$-$\fini{\xles}$ upper bound for the peak computation problem introduced in~\eqref{eq:runningdpcp}. We show that for well-chosen $\varepsilon$, $\gamma_\varepsilon$ introduced in~\eqref{eq:klgenepsdef} combined with $\theles:=-|\bfct/\afct|$ that belongs to $\fini{\xles}$ provide a $\klgen$-$\fini{\xles}$ upper bound for $(\xles,L,\philes)$.
	
	The choice of $\varepsilon$ depends on supremum of $\theles$ over $\afct^{-1}((-\infty,0))$ and over $\afct^{-1}((0,+\infty))$.
	Using the partial derivatives of $-|\bfct/\afct|$, we conclude that
	\[
		\sup_{\substack{\afct(x)<0\\ x\in X}}\theles(x)=\left\{
		\begin{array}{lr}
			\theles(\xul)& \text{ if }\afct^{-1}((-\infty,0))\neq \emptyset\\
			-\infty &  \text{ otherwise}
		\end{array}\right.
		\text{ and } 
		\sup_{\substack{\afct(x)>0\\ x\in X}} \theles(x)=\left\{
		\begin{array}{lr}
			\theles(\xlu) & \text{ if } \afct^{-1}((0,+\infty))\neq \emptyset\\
			-\infty &  \text{ otherwise}
		\end{array}\right.
	\]
	For all $x\in \xles$ such that $\afct(x)<0$, we have $\theles(x)<-1$. Moreover, the set $\afct^{-1}((-\infty,0))$ is nonempty if and only if $\afct(\xul)<0$. Thus, the supremum of $\theles$ over $\afct^{-1}((-\infty,0))$ is strictly smaller than -1. For all $x\in \xles$ such that $0<\afct(x)<\bfct(x)$, we also have $\theles(x)<-1$. However, if the set $\{x\in \xles: \afct(x)\geq \bfct(x)\}$ (included in $\afct^{-1}((0,+\infty))$) is nonempty, for all $\varepsilon>0$, we can find $x\in \xles$ such that $0<\afct(x)<\bfct(x)$ and $-1-\varepsilon\leq \theles(x)$. Note that $\{x\in \xles: \afct(x)\geq \bfct(x)\}$ is nonempty if and only if $\afct(\xlu)\geq \bfct(\xlu)$ (or equivalently $\theles(\xlu)\geq -1$). Finally, for all $x\in \xles$ such that $\afct(x)\geq \bfct(x)$, we have $\theles(x)\geq -1$. From ~\eqref{eq:abprop}, we also have $-|\lambda_1/\lambda_2|=\lambda_1/\lambda_2\geq \theles(x)$ for all $x\in \xles$ such that $\afct(x)\geq \bfct(x)$.
	
	As a corollary, we have
	{\everymath{\displaystyle}
		\[
		\osps{\theles}{\xles}=\max_{x\in \xles} \theles(x)=\left\{
		\begin{array}{lr}
			\theles(\xul) & \text{ if } \afct(\xlu)\leq 0\\
			\theles(\xlu)& \text{ if } \afct(\xul)\geq 0\\
			\max\{\theles(\xul),\theles(\xlu)\} & \text{ if } \afct(\xul)<0\text{ and } \afct(\xlu)>0
		\end{array}\right. .
		\]
	}
    From the computation of $\osps{\theles}{\xles}$, we define:
	\begin{equation}
		\label{eq:varepssupthe}
	\varepsilon^{<0}=\left\{\begin{array}{lr} \dfrac{1}{2}(-1-\theles(\xul)) & \text{ if } \afct(\xul)<0\\
		+\infty & \text{ if } \afct(\xul)\geq 0\end{array}\right.
	\text{ and } 
	\varepsilon^{>0}=\left\{\begin{array}{lr} \dfrac{1}{2}(-1-\theles(\xlu)) & 0<\afct(\xlu)<\bfct(\xlu)\\
		\dfrac{1}{2}\left(\dfrac{\lambda_1}{\lambda_2}-\theles(\xlu)\right) & \text{ if } \afct(\xlu)\geq\bfct(\xlu)\\
		+\infty & \text{ if } \afct(\xlu)\leq 0\end{array}\right.
	\end{equation}
	and 
	\begin{equation}
		\label{eq:finalesples}
	\epsles=\min\{\varepsilon^{<0},\varepsilon^{>0}\}.
	\end{equation}
	Finally, we have for all $x\in \xles$ and all $k\in\nn$
	\begin{equation}
		\label{eq:finalklgenfin}
		\philes(L^k(x))\leq \gamma_{\epsles}\left(\theles(x),k\right)\text{ where } \theles=-|\bfct/\afct|.
	\end{equation}
	We insist on the fact that $\gamma_{\epsles}\in\klgen$ and $\theles\in\fini{\xles}$. The details of the proof of~\eqref{eq:finalklgenfin} is provided in the Appendix.
	
	From sufficiently large values of $t$, $\gamma_{\epsles}(\osps{\theles}{\xles},t)=H(\osps{\theles}{\xles},|\lambda_1/\lambda_2|^t)$ and thus 
	$\inf_{t\geq 0} \gamma_{\epsles}(\osps{\theles}{\xles},t)=H\left(\osps{\theles}{\xles},0\right)=\lambda_2/c$. It follows that $\{k\in\gsls{\omeles}:\inf_{t\geq 0} \gamma_{\epsles}(\osps{\theles}{\xles},t)<\omeles_k\}=\gsls{\omeles}$ and finally
	the pair $(\gamma_{\epsles},\theles)$ is useful.
	
	Applying our computations to the numerical data of Subsection~\eqref{subsec:numericaldata}, we get
	\[
	\afct(\xlu)=95>0 \text{ and } \afct(\xul)=-10<0 
	\]
	As $\bfct(\xlu)=55$ and $\bfct(\xul)=25$, we have
	\[
	\osps{\theles}{\xles}=\max\{\theles(\xlu),\theles(\xul)\}=\max\{-11/19,-2.5\}=-11/19
	\]
	and as $\afct(\xlu)>\bfct(\xlu)$, 
	\[
	\epsles=\min\{0.5(-1+2.5),0.5(-0.5+11/19)\}=3/76
	\]
	Finally, we obtain:
	\[
	\philes({\lexp}^k x)\leq 	\gamma_{\epsles_e}(\theles(x),0.5^t)=\left\{\begin{array}{lr}
		4 \dfrac{0.5^{k+1}+\theles(x)}{0.5^k+\theles(x)} & \text{ if } \theles(x)\leq -(0.5)^t-\dfrac{3}{76}\\
		\\
		\theles(x)+153\times 0.5^t+\dfrac{307}{76} & \text{ if } \theles(x)>-0.5^t-\dfrac{3}{76}
	\end{array}
	\right..
	\]
	\qed
	\end{example}

Following the notations introduced earlier, we have the following equivalence:
\[
\nuopt<+\infty \iff \sup_{k\in\nn} \varphi\circ T^k\in\fin.
\]
We recall that:
\[
\nuopt<+\infty \iff \limsup_{k\to +\infty} \nu_k=\limsup_{k\to +\infty}\sup_{x\in\xin} \varphi(T^k(x))=\limsup_{k\to +\infty}\ \osp{(\varphi\circ T^k)}\in\rr\cup\{-\infty\}
\]
We can establish the first result, which states that the existence of $\klgen$-$\fin$ upper bounds is equivalent to the fact that the sequence $\nu$ is bounded from above. This is an adaptation of classical results dealing with stability and $\mathcal{KL}$ upper bounds.  
\begin{theorem}[$\funset(\nu)$ to $\klgen$]
\label{htoklbounds}
The sequence $\nu=(\nu_k)_{k\in\nn}$ \red{satisfies $\nuopt<+\infty$} if and only if there exists a $\klgen$-$\fin$ upper bound $(\gamma,\theta)$ for $(\xin,T,\varphi)$. If, moreover, $\gsls{\nu}\neq\emptyset$ then $(\gamma,\theta)$ can be chosen to be useful for $\nu$.
\end{theorem}

\begin{proof}
Suppose that there exists a $\klgen$-$\fin$ upper bound $(\gamma,\theta)$ for $(\xin,T,\varphi)$. As $s\mapsto \gamma(s,t)$ is increasing for all $t\geq 0$, we have
for all $x\in\xin$, for all $k\in\nn$, $\gamma(\theta(x),k)\leq \gamma(\osp{\theta},k)$ and thus
$\nu_k=\sup\{\varphi(T^{k}(x)):x\in\xin\}\leq \gamma(\osp{\theta},k)$ for all $k\in\nn$. Now, taking the limsup on both sides leads to \[\limsup_{k\mapsto +\infty} \nu_k\leq \inf_{k\in\nn} \gamma(\osp{\theta},k)\leq \gamma(\osp{\theta},0)<+\infty\] as $t\mapsto \gamma(s,t)$ is decreasing for all $s\in\rr$. \red{This proves that $\nuopt<+\infty$}. 

Suppose that $\nu=(\nu_k)_{k\in\nn}$ has a finite supremum. From \red{the first statement of Th.~\ref{mainexistence}, $\funset(\nu)$} is nonempty and let $(h,\beta)$ be in $\funset(\nu)$. Then, we define for all $s\in\rr$ and $t\in\rr_+$, \[\gamma(s,t):=\min\{s,h(\beta^t)\}\] and for all $x\in\rd$, \[\theta(x):=\sup_{k\in\nn} \varphi(T^k(x)).\] It is clear that $\gamma\in\klgen$ as $\beta \in (0,1)$. Now by assumption, $\nuopt=\sup_{n\in\nn} \nu_n\in\rr$. Moreover, we have \[\nuopt=\sup_{n\in\nn} \sup_{x\in\xin} \varphi(T^n(x))=\sup_{x\in\xin} \sup_{n\in\nn} \varphi(T^n(x))= \osp{\theta}.\] By definition of $\theta$, we have for all $x\in\xin$, for all $k\in\nn$, $\varphi(T^k(x))\leq \theta(x)$. By definition of $(h,\beta)$, we have for all $k\in\nn$, $\nu_k\leq h(\beta^k)$. This is the same as for all $x\in\xin$ and for all $k\in\xin$, $\varphi(T^k(x))\leq h(\beta^k)$. We conclude that for all $x\in\xin$ and for all $k\in\nn$\[\varphi(T^k(x))\leq \min\{\theta(x),h(\beta^k)\}.\] Note that for \red{the} function $\gamma$ constructed as above, for all $s\in\rr$ \[\inf_{t\geq 0} \gamma(s,t)=\inf_{t\geq 0}\min\{s,h(\beta^t)\}=\min\{s,h(0)\}.\] This follows from the continuity of $h$, the fact $h$ is strictly increasing and $\beta\in (0,1)$.

If moreover, $\gsls{\nu}\neq\emptyset$, then from the second statement \red{of} \blue{Th~\ref{mainexistence}}, there exists $h\in\secfunh(\nu)$ useful for $\nu$. We can define $\gamma$ using this function $h$. Now, let $k\in\res(\nu,h)$. As $\nu_k\leq \nuopt=\osp{\theta}$, we conclude that \[\inf_{t\geq 0} \gamma(\osp{\theta},t)=\min\{\osp{\theta},h(0)\}=h(0)\] which ends the proof.  
\end{proof}
In the proof of Th.~\ref{htoklbounds}, we exploit \red{Th.~\ref{mainexistence} which links} the nonemptiness of the sets $\funset(u)$, $\res(u,h)$ and $\gsls{u}$. We can reformulate Th~\ref{htoklbounds} as Corollary~\ref{klvershsimple}.

\begin{corollary}
\label{klvershsimple}
The set $\funset(\nu)$ is nonempty if and only if there exists a $\klgen$-$\fin$ upper bound for $(\xin,T,\varphi)$.
If, moreover, there exists $h\in\secfunh(\nu)$ useful for $\nu$ then there exists a useful $\klgen$-$\fin$ upper bound for $(\xin,T,\varphi)$.
\end{corollary}

In Th~\ref{htoklbounds} and in Corollary~\ref{klvershsimple}, we construct a useful $\klgen$-$\fin$ upper bound for $(\xin,T,\varphi)$ from a useful $(h,\beta)\in\funset(\nu)$. Now, \red{we have to check} whether it is possible to construct a useful $(h,\beta)\in\funset(\nu)$ from a useful $\klgen$-$\fin$ upper bound for $(\xin,T,\varphi)$. Thus, we provide a positive response. The result is based on  an extension of Sontag's result for the class $\klgen$. This permits to fix the value $\beta\in (0,1)$ to only consider problem on functions in $\fset$.  

\begin{theorem}[$\klgen$ to $\funset(\nu)$]
\label{KLversh}
Suppose that there exists a $\klgen$-$\fin$ upper bound for $(\xin,T,\varphi)$. Then there exists $h\in\fset$ such that $(h,e^{-1})\in\Gamma(\nu)$. If moreover, the  $\klgen$-$\fin$ upper bound for $(\xin,T,\varphi)$ is useful for $\nu$ then $h$ can be chosen to be useful for $\nu$. 
\end{theorem}

\begin{lemma}
\label{dectoregular}
Let $f:\rr_+\to \rr$ be a decreasing function and let $m\geq \inf_{\rr_+} f$ ($m$ can be equal to $-\infty$ if $\inf_{\rr_+} f$ is). There exists a function $g:\rr_+\mapsto \rr$ that is strictly decreasing and continuous such that:
\begin{enumerate}
\item $f\leq g$;
\item $\inf_{\rr_+} g=m$.
\end{enumerate} 
\end{lemma}
\begin{proof}
If $m\geq f(0)$, the function $x\mapsto g(x)=m+1/(x+1)$ satisfies the above properties. Note that this case includes the case where $f$ is constant on $\rr_+$. 

Now assume that $f(0)>m\geq \inf_{\rr_+} f$. We denote for all couple of pairs $\{(x,y),(z,t)\}$ where $x,y,z,t \in\rr$ by $\aff\{(x,y),(z,t)\}$ the unique affine function $\ell:\rr\mapsto \rr$ such that $\ell(x)=y$ and $\ell(z)=t$. We define the functions $F_0$ and $F_n$ for all $n\in\nn^*$ as follows:
\[
F_0=f(0)\mm_{[0,1]}
\text{ and }
F_n=
f(n-1)\mm_{\left[n;n+\frac{1}{2}\right]}+\aff\left\{\left(n+\frac{1}{2};f(n-1)\right);\left(n+1;f(n)\right)\right\}\mm_{\left[n+\frac{1}{2};n+1\right]}
\]
Finally, we define, for all $x\in\rr_+$, the function $F$:
\[
F(x)=\sum_{n\in\nn} F_n(x) + \dfrac{1}{x+1} \text{ and note that} \sum_{n\in\nn} F_n(x)=\left\{
\begin{array}{lr}
f(0) & \text{ if } x \in [0,1]\\
f(\lfloor x\rfloor -1) & \text{ if } x\in [\lfloor x\rfloor ,\lfloor x\rfloor+1/2]\\
\ell\left(x\right) & \text{ if } x\in (\lfloor x\rfloor +1/2,\lfloor x\rfloor+1)
\end{array}
\right.
\]
where $\ell=\aff\left\{\left(\lfloor x\rfloor+1/2;f(\lfloor x\rfloor-1)\right);\left(\lfloor x\rfloor+1;f(\lfloor x\rfloor)\right)\right\}$.

The function $\sum_{n\in\nn} F_n$ is clearly continuous and decreasing (as $f$ does) on $\rr_+$. The function $\sum_{n\in\nn} F_n$ is not strictly decreasing, but adding $x\mapsto (x+1)^{-1}$ makes $F$ strictly decreasing and continuous. We also have $f\leq \sum_{n\in\nn} F_n$ because $f$ is decreasing. As $x\mapsto (x+1)^{-1}$ is strictly positive, we have $f\leq F$. 
Finally, we have $\inf_{\rr_+} F=\inf_{\rr_+} f$. Indeed, for all $n\in\nn^*$ and all $x\in [n,n+1)$, $F_n\leq f(n-1)$, it implies that $\inf_{\rr_+} F= \inf_{x\in \rr_+}\inf_{n\in\nn^*} F_n(x)\leq \inf_{n\in\nn^*} f(n-1)=\inf_{\rr_+} f$ (as $f$ decreases). The reverse inequality holds as $f\leq F$. 

The function $F$ does not necessarily satisfy \blue{$\inf_{\rr_+} F=m$}. Hence, we define for all $x\in\rr_+$ as 
\[
g(x)=\left\{
\begin{array}{lr}
\max\{F(x),(x+1)^{-1}+m\} & \text{ if } m>\inf_{\rr_+} f\\
F(x) & \text{ if } m=\inf_{\rr_+} f
\end{array}
\right.
\]
Suppose that $m>\inf_{\rr_+} f$. The function $g$ is clearly strictly decreasing and continuous and satisfies $f\leq g$ as $F$ does. As $m>\inf_{\rr_+} f=\inf_{\rr_+} F$, there exists $x_0\red{\in\rr_+}$ such that $F(x)<m$ for all $x\geq x_0$ and then $g(x)= (x+1)^{-1}+m$ for all $x\geq x_0$. This implies that $\inf_{\rr_+} g=m$. 
\end{proof}

\begin{lemma}[Sontag's extension]
\label{alasontag}
Let $\gamma\in\klgen$. Let $s\in\rr$ and let $m\geq \inf_{t\geq 0} \gamma(s,t)$ be finite. Then there exists $h\in\Omega([0,1])$ such that for all $t\geq 0$, $\gamma(s,t)\leq h(e^{-t})$ and $h(0)=m$.
\end{lemma}

\begin{proof}
By definition, $t\mapsto \gamma(s,t)$ is decreasing and applying Lemma~\ref{dectoregular}, there exists a function $\sigma:\rr_+\mapsto \rr$ strictly decreasing and continuous such that $\gamma(s,t)\leq \sigma(t)$ for all $t\geq 0$ and $\inf_{\rr_+} \sigma=m$. Let us define for all $r\in (0,1]$, $h(r)=\sigma(-\ln r)$ and $\lim_{x\to 0^+}h(x)=\lim_{y\to +\infty}\sigma(y)=m$. Then, defining $h(0)=m$ makes $h$ continuous at 0. Finally, $h$ belongs to $\Omega([0,1])$.       
\end{proof}

\begin{proof}[Proof of Th.~\ref{KLversh}]
By assumption, there exists a $\klgen$-$\fin$ upper bound $(\gamma,\theta)$ for $(\xin,T,\varphi)$. Let $x\in\xin$ and $k\in\nn$. This implies that $\varphi(T^k(x))\leq \gamma(\osp{\theta},k)$. By Lemma~\ref{alasontag}, there exists $h\in\Omega([0,1])$ such that $\gamma(\osp{\theta},k)\leq h(e^{-k})$ for all $k\in\nn$. This leads to $(h,e^{-1})\in\funset(\nu)$. If there exists $k\in\nn$ such that $\inf_{t\in\rr}\gamma(\osp{\theta},t)<\nu_k$, we can choose, as it is stated in Lemma~\ref{alasontag}, $h$ such that $h(0)=m\in (\inf_{t\in\rr}\gamma(\red{\osp{\theta}},t),\nu_k)$. 
\end{proof}

\begin{example}[An element of $\funset(\omeles)$ from $(\gamma_{\epsles},\theles)$]
From Th~\ref{KLversh}, we can build an element of $\funset(\omeles)$ from the $\klgen$-$\fini{\xles}$ upper bound obtained in Example~\ref{ex:klgenfinleslie}. The situation is simpler in Example~\ref{ex:klgenfinleslie} as we get directly a scalar in $(0,1)$ from the expression of $\gamma_{\epsles}$.
 
Let 
\[
h:[0,1]\ni x\mapsto \left\{\begin{array}{lr}\displaystyle{\gamma_{\epsles}\left(\osps{\theles}{\xles},\dfrac{\ln(x)}{\ln(\beta)}\right)}&\text{ if } x\in(0,1]\\
	\dfrac{\lambda_2}{c} & x=0
	\end{array}\right.
	\text{ and } \beta=\left|\dfrac{\lambda_1}{\lambda_2}\right|
\]	
where $\gamma_\varepsilon$ is defined in~\eqref{eq:klgenepsdef}, $\epsles$ and $\theles$ are defined in Example~\ref{ex:klgenfinleslie}. As $t$ tends to $+\infty$, we have  $\osps{\theles}{\xles}<-|\lambda_1/\lambda_2|^t-\epsles$ for all sufficiently large $t$. It follows that $\lim_{x\mapsto 0^+} h(x)=\lim_{z\mapsto 0^+} H(\osps{\theles}{\xles},z)=\lambda_2/c$. Thus, the function $h$ is continuous at 0. From~\eqref{eq:finalklgenfin}, we can deduce that $(h,\beta)\in\funset(\omeles)$. 
Furthermore, we have, $h([0,1])=[\lambda_2/c,\gamma_{\epsles}(\osps{\theles}{\xles},0)]$ that contains all terms $\omeles_k$ greater than the limit superior of $\omeles$ and 
\[
[\lambda_2/c,\gamma_{\epsles}(\osps{\theles}{\xles},0)]\ni y\mapsto \inv{h}{[0,1]}(y)=\left\{
\begin{array}{lr}
\dfrac{\osps{\theles}{\xles}(\lambda_2-cy)}{cy-\lambda_1} & \text{ if } y\leq \dfrac{\epsles\lambda_1-\osps{\theles}{\xles}(\lambda_2-\lambda_1)}{c\epsles}\\
 \dfrac{\epsles(cy-c\osps{\theles}{\xles}-c\epsles-\lambda_2)}{c\epsles+\lambda_2-\lambda_1}&\text{ if }y>\dfrac{\epsles\lambda_1-\osps{\theles}{\xles}(\lambda_2-\lambda_1)}{c\epsles}
\end{array}
\right..
\]
Then we can use Formula~\eqref{eq:fnu} that becomes, if $\omeles_k>\lambda_2/c$:
\[
\dfrac{\ln(\inv{h}{[0,1]}(\omeles_k))}{\ln\left(\left|\dfrac{\lambda_1}{\lambda_2}\right|\right)}.
\]

Applying those computations to the numerical data of Subsection~\ref{subsec:numericaldata}, we get, as $\osps{\theles}{\xles}=-11/19$
\[
\left[4,\dfrac{11891}{76}\right]\ni y\mapsto \inv{h}{[0,1]}(y)=\left\{
\begin{array}{lr}
	-\dfrac{11}{19}\dfrac{(2-0.5y)}{0.5y+1} & \text{ if } y\leq 86\\
	\\
	\dfrac{y}{153}-\dfrac{263}{11628}&\text{ if }y>86
\end{array}
\right..
\]
Recall that $\omeles_0=\philes(\xul)=40/5=8$ and then
\[
\dfrac{\ln(\inv{h}{[0,1]}(8))}{-\ln(2)}=\dfrac{\ln(22/57)}{-\ln(2)}\approx 1.3735
\]
This means that the maximum of terms $(\omeles_k)_k$ is either $\omeles_0$ or $\omeles_1$. Now as $\omeles_1=\philes(L\xlu)=42$. This implies that $\omelesopt=42$. Moreover, we find that 
\[
\dfrac{\ln(\inv{h}{[0,1]}(42))}{-\ln(2)}=\dfrac{\ln\left(\dfrac{11\times 19}{19\times 22}\right)}{-\ln(2)}=1.
\]
\qed
\end{example}

\subsection{Summary of the \protect$\klgen$ extension}

We extend the class of $\klcls$ functions that classically serve for dynamical systems stability. In our context, our $\klgen$ class constitutes the main part of $\klgen$-$\fin$ upper bound of $\varphi(T^k(x))$. We prove that such upper bounds exist if and only if $\nu$ has a finite supremum. Moreover, if $\gsls{\nu}\neq \emptyset$, we can construct a function $h\in\secfunh(\nu)$ useful for $\nu$ such that $e^{-1}\in\secfunb(\nu,h)$.

Unfortunately, in practice, proving that a pair $(\gamma,\theta)$ is a $\klgen$-$\fin$ upper bound for $(\xin,T,\varphi)$ seems to be as difficult as proving that a pair $(h,\beta)$ belongs to $\funset(\nu)$. Thus, we propose a new type of Lyapunov function for which the conditions do not involve the entire trajectory of the analyzed dynamical system.

\section{Lyapunov type Functions for Discrete-time Systems Maximization}
\label{newlyapunov}

Classically, Lyapunov functions (i.e., satisfying conditions~\eqref{eq:positivitylyap} and~\eqref{eq:decreaselyap}) are used as certificates of stability (see ~\cite{elaydiintroduction} or~\cite{kelley2001difference}). They present a constructive approach to prove the stability of a dynamical system through converse Lyapunov theorems, which prove that stability notions are equivalent to the existence of Lyapunov functions~\cite{khalil2002nonlinear}. For some linear, polynomial, or piecewise polynomial systems, Lyapunov functions can be computed from numerical optimization solvers based on linear or semidefinite programming (for example, see~\cite{giesl2015review} and references therein). \red{This numerical opportunity has been exploited for specific peak computation problems. Indeed, in~\cite{adje13052025} quadratic Lyapunov functions were used to solve our problem when the dynamical system is stable and affine and the objective function is quadratic. Next, in~\cite{adje2025kllyap}, we have proved 
that more general (local) Lyapunov functions can be used to solve peak computation problems. In these previous works, we applied direct-like Lyapunov theorems. In Lyapunov stability theory, after obtaining direct theorems, we consider converse theorems. The question raised in this case is, which stability notions imply the existence of a Lyapunov function? In our context, the converse theorem relies on the following question: which properties of classical Lyapunov functions can be removed to obtain an equivalence between the existence of a useful pair $(h,\beta)\in\funset(\nu)$ and the existence of a weaker version of the Lyapunov function? In this section, we reply to this answer.} First, in Subsection~\ref{generallyapunov}, we define a more general class of Lyapunov functions \red{called {\bf Opt-Lyapunov} functions}. The existence of these functions does not ensure the stability of the system, but only the finiteness of the limit superior of the sequence $\nu$ defined in Eq. ~\eqref{nudef}. Then, in Subsection~\ref{compatibility}, we discuss the notion of usefulness and how it constrains the Opt-Lyapunov functions. In the classical case, usefulness constrains the sequence $\nu$ to have some positive terms. We replace the functions $\alpha_1,\alpha_\varphi$ in \eqref{eq:hfromextlyap} by a \emph{certificate of compatibility}. Finally, in Subsection~\ref{directconverselyap}, we establish an equivalence theorem between the existence of those Opt-Lyapunov functions and the existence of an element $(h,\beta)\in\funset(\nu)$. 

\subsection{Generalized Lyapunov Functions for Discrete-time Systems Maximization}
\label{generallyapunov}

\subsubsection{Definitions and Useful Facts}
We recall that the \textit{domain} of function $g:\rd\mapsto (-\infty,+\infty]$ is the set of $x\in\rd$ such that  $f(x)<+\infty$. We denote by $\dom(f)$ the domain of a extended real valued function.

We extend the class of classical Lyapunov functions as follows: An Opt-Lyapunov function is still nonnegative but can be either strictly positive or can have several zeros and can take infinite values. The supremum of an Opt-Lyapunov function on the initial conditions set is non-null but smaller than 1. Opt-Lyapunov functions keep the decrease condition of classical Lyapunov functions.   
\begin{defi}[Opt-Lyapunov functions]
\label{lyapdef}
A function $V$ is said to be an Opt-Lyapunov function for $(\xin,T)$ if and only if
\begin{enumerate}
\item $V:\rd\mapsto [0,+\infty]$;
\item $\osp{V}\in (0,1]$;
\item $V\circ T\leq \lambda V$ for some $\lambda\in (0,1)$.
\end{enumerate}
\end{defi}
\begin{example}[Opt-Lyapunov function for the Leslie Model]
	\label{ex:opt-lyap-leslie}
	Returning to the Leslie model example introduced in Section~\ref{sec:leslieexample}. As $\lambda_2>1$, the system is unstable, that is, $\norm{x_k}$ goes to $+\infty$ as $k$ tends to $+\infty$. A classical Lyapunov function does not exist for this system. However, an Opt-Lyapunov function exists. We propose, for $C>0$,
	\begin{equation}
		\label{eq:opt-lyap-leslie}
	\rr^2\ni x\mapsto \lyaples(x)=\left\{\begin{array}{cr}\dfrac{C}{x_1^2+x_2^2} & \text{ if } x_1,x_2>0\\
		+\infty & \text{ otherwise}
	\end{array}\right..
	\end{equation}
	Let us take $x\in\rr^2$. If $Lx$ has a nonpositive coordinate (and thus $\lyaples(Lx)=+\infty$), then either $x_1\leq 0$ or $x_2\leq 0$ yielding to $\lyaples(x)=+\infty$. Therefore, $\lyaples(Lx)\leq \lambda \lyaples(x)$ for any $\lambda\in (0,1)$. Now, assume that both coordinates of $Lx$ are strictly positive. This implies that $x_1>0$. We can have $x_2\leq 0$. In this case, $\lyaples(x)=+\infty$ and $\lyaples(Lx)\leq \lambda \lyaples(x)$ for any $\lambda\in (0,1)$. Finally, suppose that $Lx$ and $x$ have strictly positive coordinates. We have from ~\eqref{eq:abdef}
	\[
	\dfrac{\lyaples(Lx)}{\lyaples(x)}=\dfrac{x_1^2+x_2^2}{(ax_1+bx_2)^2+c x_1^2}=\dfrac{x_1^2+x_2^2}{(a^2+c^2)x_1^2+b^2x_2^2+2ab x_1 x_2}<\dfrac{x_1^2+x_2^2}{(a^2+c^2)x_1^2+b^2x_2^2}<1.
	\]
	 This inequality is not sufficient to prove condition (3) of Definition~\ref{lyapdef} as the supremum of $\lyaples(Lx)/\lyaples(x)$ might be equal to 1. We will compute the exact value of the supremum of $\lyaples(Lx)/\lyaples(x)$ over the open positive orthant of $\rr^2$. This supremum is strictly smaller than 1.
	 
	 We can adjust the constant $C$ to have $\osples\in (0,1]$ even if we will see that the finiteness of $\osples$ is sufficient to obtain the second statement of Definition~\ref{lyapdef}. As $\lyaples$ is marginally strictly decreasing we have $\osples=C/(l_1^2+l_2^2)$. Thus any $C\in (0,l_1^2+l_2^2]$ can be chosen.

	Applying those computations to the numerical data of Subsection~\ref{subsec:numericaldata}, we get for all $x\in(\rr_+^*)^2$, \[\lyaples(x)=125/x_1^2+x_2^2\] where we choose $C=l_1^2+l_2^2$ to obtain $\osps{\lyaples}{\xinles}=1$.
	 \qed
\end{example}
\begin{proposition}
\label{fixlyap}
If $V$ is an Opt-Lyapunov function for $(\xin,T)$ then 
\begin{enumerate}
\item  for all $k\in\nn$, $T^k(\xin)\subseteq V^{-1}([0,1])\subseteq \dom(V)$;
\item for all fixed points $x$ of $T$, $V(x)\in \{0,+\infty\}$.
\end{enumerate}
\end{proposition}
\begin{proof}
{\itshape 1}. From the second statement for all $x\in\xin$, $V(x) \in [0,1]$. From the third statement and the fact that $V$ is nonnegative, we have, for all $x\in\xin$, for all $k\in\nn$, $0\leq V(T^k(x))\leq V(x)\leq 1$. 

{\itshape 2}. Let $x$ be a fixed point of $T$. If $V(x)=+\infty$, there is nothing to prove. Suppose that $V(x)$ is finite. From the third statement, $0\leq V(T(x))=V(x)\leq \lambda V(x)$ for some $\lambda\in (0,1)$. Hence, $V(x)(1-\lambda)= 0$ and we conclude that $V(x)=0$.   
\end{proof}

\begin{corollary}
\label{coronofix}
If $\xin$ is included in the fixed point set of $T$ then a function $V:\rd\mapsto [0,+\infty]$ such that $V\circ T\leq \lambda V$ for some $\lambda\in (0,1)$ cannot satisfy $\osp{V}>0$.
\end{corollary}
Corollary~\ref{coronofix} extends the fact that if a classical Lyapunov function exists for a self-map $T$, $T$ cannot have a nonzero fixed point. If an Opt-Lyapunov function for $(\xin,T)$ exists, $\xin$ contains a vector $x$ such that $x\neq f(x)$.

In fact, we can equivalently replace the second statement of Def.~\ref{lyapdef} by $\osp{V}$ is finite and strictly positive. Indeed, under this condition, we recover an Opt-Lyapunov function for $(\xin,T)$ with $\osp{W}\in (0,1]$ by taking $W:=V/\red{\osp{V}}$.

\begin{proposition}
\label{similaroptlyap}
The existence of a function $W:\rd\mapsto [0,+\infty]$ such that 
\begin{enumerate}
 \item $\osp{W}$ is strictly positive and finite ;
 \item $W\circ T\leq \rho \circ W$ for a function $\rho:[0,+\infty]\to [0,+\infty]$ such that:
 \begin{itemize}
  \item $\rho(x)=0 \iff x=0$ and $\rho(x)=+\infty \iff x=+\infty$;  
  \item $\rho_{\mid \rr_+}\in \kcls$;
  \item $\rho(x)<x$ for all $x\in\rr_+^*$.
  \end{itemize}  
\end{enumerate}
is equivalent to the existence of an Opt-Lyapunov function for $(\xin,T)$.
\end{proposition}

\begin{lemma}
\label{poslemma}
Let $\gamma$ be strictly greater than 1. Let $f\in\kclsi$ verifying for all $x>0$, $f(x)>x$. Then, there exists a function $g\in\kcls$ that satisfies $g(f(x))=\gamma g(x)$ for all $x\in\rr_+$.
\end{lemma}

Lemma~\ref{poslemma} is a slightly modified version of~\cite[Lemma 25]{DBLP:journals/mcss/Kellett14}, for which $f-\Idd$ is supposed to be in $\kclsi$. To be self-contained, we propose in the Appendix a proof inspired by the one provided in~\cite{DBLP:journals/mcss/Kellett14}.

\begin{lemma}
\label{lemmafact}
Let $\lambda\in (0,1)$. Let $f\in\kcls$ verifying for all $x>0$, $f(x)<x$. Then, there exists a function $g\in\kcls$ satisfying $g(f(x))=\lambda g(x)$ for all $x\in\rr_+$.
\end{lemma}

Again, a proof of Lemma~\ref{lemmafact} is given in~\cite{DBLP:journals/mcss/Kellett14}. A proof is provided in the Appendix for the convenience of the reader.

\begin{proof}[Proof of Prop.~\ref{similaroptlyap}]
Suppose that such $W$ and $\rho$ exist, then let $\lambda\in (0,1)$. From Lemma~\ref{lemmafact}, there exists $g$ such that $g(\rho(x))=\lambda g(x)$ for all $x\in\rr_+$. By extension, we define $g(+\infty)=+\infty$. Then as $g$ is increasing, we have $g(W(T(x)))\leq g(\rho(W(x))=\lambda g(W(x))$ for all $x\in\rd$. Now, as $g$ is strictly increasing, there exists some $x\in\xin$ such that $g(W(x))$ is strictly positive and $\osp{g\circ W}$ is strictly positive and finite. Finally, $(g\circ W)/\osp{g\circ W}$ is an Opt-Lyapunov function for $(\xin,T)$. 

The reverse implication is obvious with $\rho(x)=\lambda x$ for all $x\in \rr_+$.
\end{proof}

Prop.~\ref{lyapclassicisopt} shows that the set of Opt-Lyapunov functions contains the \red{set of} classical Lyapunov functions.
\begin{proposition}[Classical is Opt]
	\label{lyapclassicisopt}
	If $V$ is a classical Lyapunov function for $T$ and $\xin$ is \red{bounded} and not reduced to $\{0\}$, then $V/\osp{V}$ is an Opt-Lyapunov function for $(\xin,T)$.
\end{proposition}

\begin{proof}
	Let $V$ be a classical Lyapunov function for $T$. If $\xin$ is not reduced to $\{0\}$, as $V$ is positive definite, $V(x)>0$ for some $x\in\xin$. From ~\eqref{eq:positivitylyap}, if $\xin$ is bounded, $\osp{V}$ is finite. This implies that $V/\osp{V}$ satisfies the second statement in Definition~\ref{lyapdef}. Moreover, it is obvious that $V/\osp{V}$ is positive definite and so verifies the first statement of Definition~\ref{lyapdef}. It is easy to see that $V/\osp{V}$ satisfies the third statement of Definition~\ref{lyapdef} from \eqref{eq:decreaselyap} and $\osp{V}>0$. 
\end{proof}	

\subsubsection{Operator Ratios From Proper Positive Semidefinite Functions}
\label{oppsd}

We extend the concept of operator ratios developed in~\cite{adje2025kllyap} (specialized at Eq.~\eqref{eq:opratio} for Lyapunov functions) for which the arguments are positive semidefinite functions rather than positive definite functions.  The arguments can also be functions that take infinite values. Therefore, let us denote by $\PSD$ the set of nonzero proper positive semidefinite functions from $\rd$ to $[0,+\infty]$, that is,
\begin{equation}
\label{psddef}
\PSD:=\{f:\rd\mapsto [0,+\infty] : \exists\, x\in\rd,\ f(x)\in\rr_+^*\}
\end{equation}
We can extend the operator ratio function to $\PSD$ functions. 
\begin{defi}[Operator Ratios PSD Based]
\label{extopdef}
Let $U:\rd \mapsto \rd$. We associate to $U$ the operator ratio function $\newop{U}:\PSD \mapsto [0,+\infty]$ defined as follows: 
\[
\red{\PSD\ni P\mapsto} \newop{U}(P):=\sup_{P(x)\in \rr_+^*} \dfrac{P(U(x))}{P(x)}
\]
\end{defi}
Recall that $\Idd$ denotes the identity function $\rd\ni x:\mapsto x$. It is straightforward to see that if $P\in\PSD$, $\newop{\Idd}(P)=1$. 

The operator function is well-defined, but the inequality $\newop{U^k}(P)\leq \newop{U}(P)^k$ can fail to be true, as shown in Example~\ref{counterex}.

\begin{example}[A counterexample to the power commutation]
\label{counterex}
Let us consider the functions $U:x \mapsto -x $ and $P:x\mapsto \max\{x,0\}$. We have $\newop{U}(P)=\sup_{P(x)>0} P(U(x))/P(x)=\sup_{x>0} P(-x)/x=\sup_{x>0} 0/x=0$ whereas $\newop{U^2}(P)=\sup_{x>0} P(x)/x=\sup_{x>0} x/x=1$. Finally, we have $\newop{U^2}(P)>\newop{U}(P)^2$.
\end{example}

To obtain the inequality $\newop{U^k}(P)\leq \newop{U}(P)^k$, we have to consider a smaller class of operators $P$ and we introduce for a self-map $U:\rd\mapsto \rd$, the set $\defNU{U}$ of operators defined as follows:
\begin{equation}
\defNU{U}:=
\{P\in\PSD : P(U(x))=0\implies P(U^2(x))=0\}
\end{equation}
We remark that $P\in\defNU{U}$ if and only if $P(U^2(x))\in (0,+\infty] \implies P(U(x)) \in (0,+\infty]$. Moreover, if $P\in\defNU{U}$, this leads to for all $k\geq 1$, $P(U^{k+1}(x))\in (0,+\infty]\implies P(U^{k}(x))\in (0,+\infty]$ and we also have if $P(U^k(x))\in (0,+\infty]$ for some $k\geq 1$ then $P(U^j(x))\in (0,+\infty]$ for all $1\leq j\leq k$.

\begin{example}[$P$ not in $\defNU{U}$]  
	\red{The functions $U$ and $P$ are the ones defined in Example~\ref{counterex}}. Let $x>0$ be strictly positive. Then, $P(U(x))=P(-x)=0$. However, $P(U^2(x))=P(x)=x>0$. So $P\notin\defNU{U}$. 
\end{example}

\begin{proposition}
\label{compoop}
Let $U:\rd\mapsto \rd$ and $P\in\defNU{U}$. Then, for all $k\in\nn^*$, $\newop{U^k}(P)\leq \left(\newop{U}(P)\right)^k$ and if $\newop{U}(P)$ is finite this inequality holds for $k=0$. 
\end{proposition}

\begin{proof}
For $k=0$, the result holds as $\newop{\Idd}(P)=1$ and there exists $x\in\rd$ such that $P(x)>0$. If $\newop{U}(P)$ is finite, $(\newop{U}(P))^0$ is defined and equal to 1. 

Now let $k\in\nn^*$. If $\newop{U}(P)=+\infty$, then the inequality is obviously true. Suppose now that $\newop{U}(P)<+\infty$. From the defintion of $\newop{U}(P)$, we have for all $x\in\rd$ and all $j\in\nn^*$ such that $P(U^{j-1}(x))\in\rr_+^*$, $P(U^{j}(x))\leq \newop{U}(P) P(U^{j-1}(x))$. From the fact that $P\in\defNU{U}$, we also have for all $x\in\rd$ and all $j\in\nn^*$ such that $P(U^{j}(x))=0$, $P(U^{j+1}(x))\leq \newop{U}(P) P(U^{j}(x))$ as it implies that $P(U^{j+1}(x))=0$. Finally, we have, for all $x\in\rd$ such that $P(x)>0$, $P(U^{k}(x))\leq \newop{U}(P) P(U^{k-1}(x))$. By dividing by $P(x)$ and taking the supremum, we obtain $\newop{U^k}(P)\leq \newop{U}(P)\newop{U^{k-1}}(P)$ and we conclude by recurrence.
\end{proof}

\begin{proposition}
\label{op01}
If $V$ is an Opt-Lyapunov function for $(\xin,T)$, then $V\in\defNU{U}$ and $\newop{T}(V)\in [0,1)$.
\end{proposition}

\begin{proof}
First, $V^{-1}(\rr_+^*)$ is nonempty from the second statement of Def. ~\ref{lyapdef}. The third statement implies that $V(T^2(x))\leq \lambda V(T(x))$ for all $x\in\rd$ for some $\lambda\in (0,1)$ and thus $V(T^2(x))>0$ implies that $V(T(x))>0$. We conclude that $V\in\defNU{T}$. The third statement also implies that $V(T(x))\leq \lambda V(x)$ for all $x\in\dom(V)$ for some $\lambda\in (0,1)$ and thus $N_T(V)\leq\lambda<1$. It can be equal to 0 if and only if for all $x$ such that $V(x)>0$, $V(T(x))=0$.   
\end{proof}

\begin{example}[Operator ratio for $\lyaples$]
We compute the operator ratio of $L$ at the Opt-Lyapunov function $\lyaples$ proposed in Equation~\eqref{eq:opt-lyap-leslie}.
\[
\newop{L}(\lyaples)=\sup_{x\in {\lyaples}^{-1}(\rr_+^*)}  \dfrac{\lyaples(Lx)}{\lyaples(x)}=\sup_{x\in \lyaples^{-1}(\rr_+^*)} \dfrac{x^\intercal x}{x^\intercal L^\intercal L x}=\left(\inf_{x_1>0,\ x_2>0} \dfrac{x^\intercal L^\intercal L x}{x^\intercal x}\right)^{-1} 
\]
Finally, we can use a characterization of \emph{copositive matrices} (see for example ~\cite{doi:10.1137/090750391} and references therein) to compute $\newop{L}{\lyaples}$. Indeed, we have:  
\[
\inf_{\substack{x_1>0\\ x_2>0}} \dfrac{x^\intercal L^\intercal L x}{x^\intercal x}=\sup\left\{\lambda\geq 0: x^\intercal L^\intercal L x-\lambda x^\intercal x\geq 0 \text{ on } (\rr_+^*)^2\right\}=\sup\left\{\lambda\geq 0: L^\intercal L -\lambda\Idd \text{ is copositive}\right\}
\]

From~\cite[Proposition 2.1]{doi:10.1137/090750391}, $L^\intercal L x-\lambda\Idd=\begin{pmatrix} a^2+c^2-\lambda & ab \\ ab & b^2-\lambda\end{pmatrix}$ is copositive if and only if $a^2+c^2-\lambda\geq 0 $, $b^2-\lambda\geq 0$ and $\sqrt{(a^2+c^2-\lambda)(b^2-\lambda)}+ab\geq 0$. Since $a\geq 1$ and $b>1$, the third inequality is a consequence of the first two conditions. We conclude that:
\[
\newop{L}(V)=\sup_{x\in V^{-1}(\rr_+^*)} \dfrac{x^\intercal x}{x^\intercal L^\intercal L x}=\dfrac{1}{\min\{a^2+c^2,b^2\}}
\]

Applying the computations to the numerical data of Subsection~\ref{subsec:numericaldata}, we have \[\newop{L}(\lyaples)=\dfrac{4}{5}.\]
\qed
\end{example}
To prove that a function $V\in\PSD$ is an Opt-Lyapunov function for $(\xin,T)$, we could use any value greater $\newop{T}(V)$. The computation of $\newop{T}(V)$ ensures that the best value possible is obtained. This is important in the context of using the function $\Fnu(k,h,\beta)$ with an Opt-Lyapunov function because $\beta\mapsto \Fnu(k,h,\beta)$ is decreasing.


\red{For classical Lyapunov functions, the extended version (on positive semidefinite functions) of operator ratios (Definition~\ref{extopdef}) coincides with the positive definite version (recalled in~\eqref{eq:opratio}).}
\begin{proposition}
\label{nopnewop}
Let $U:\rd\mapsto \rd$. Suppose that there exists a classical Lyapunov function $V$ for $U$. Then, $V\in\defNU{U}$ and $\newop{U}(V)=N_U(V)$.
\end{proposition}

\begin{proof}
If there exists a Lyapunov function $V$ for $U$, then $U(0)=0$. Moreover, $V$ is positive definite. Thus if $V(U^2(x))>0$ then $U^2(x)\neq 0$ and thus $U(x)\neq 0$ hence $V(U(x))>0$ and $V\in\defNU{U}$. The equality between $\newop{U}(V)$ and $N_U(V)$ follows readily from the fact that $\{y\in\rd : V(y)>0\}=\rd\backslash \{0\}$. 
\end{proof}

\subsection{\protect$(\xin,T,\varphi)$-certificate of compatibility and \protect$(\xin,T,\varphi)$-compatibility}
\label{compatibility}
The construction of $h\in\secfunh(\nu)$ from a classical Lyapunov function $V$ is based on the existence of a function $\alpha\in\kclsi$ such that $\varphi\leq \alpha\circ V$. The existence of this latter function $\alpha$ is either assumed (with $\alpha=\Idd$) in Equation~\eqref{eq:hfromextlyap} or in Equation~\eqref{eq:hfromlyap} comes from the continuity of $\varphi$ and the definition of $V$ (see Equation~\eqref{eq:positivitylyap}). However, when $V$ is an Opt-Lyapunov function for $(\xin,T)$, $V$ is neither positive definite nor continuous. Thus, we cannot guarantee the existence of $\alpha\in\kclsi$ such that $\alpha\circ \varphi\leq V$. We must use a different class of functions to find a similar relation between $\varphi$ and $V$. 
First, we propose to restrict the interval in which a function $\alpha$ should be continuous and strictly increasing. Indeed, the analysis only concerns the values $\{\varphi(T^k(x)),k\in\nn, x\in \xin\}$ and $\{V(x), x\in \xin\}$. Hence, we introduce the following interval:
\[
\ivt:=\overline{\operatorname{conv}}(\{ \varphi(T^k(x)) : k\in\nn,\ x\in\xin\})\in\IR
\] 
\begin{proposition}
\label{ivtprop}
The following assertions hold:
\begin{enumerate}
\item For all $k\in\nn$, $\nu_k\in\ivt$. If the values are finite, $\limsup_{n\in\nn} \nu_n$ and $\liminf_{n\in\nn} \nu_n$ also belong to $\ivt$.
\item
\[
\ivt=\left[\inf_{k\in\nn} \nu_k, \sup_{k\in\nn} \nu_k\right]\cap \rr
\]
\end{enumerate}
\end{proposition}

\begin{proof}
Let $k\in\nn$. There exists for all $n\in\nn$, $x_n\in\xin$ such that $\nu_k-1/n\leq \varphi(T^k(x_n))\leq \nu_k$ and thus the sequence $(\varphi(T^k(x_n)))_{n\in\nn}$ converges to $\nu_k$ and thus $\nu_k\in\ivt$. The values, if they are finite, $\limsup_{k\in\nn} \nu_k$ and $\liminf_{k\in\nn} \nu_k$ are limits of subsequences of $(\nu_k)_{k\in\nn}$ and they belongs to $\ivt$. From the definition of $\limsup$ (resp. $\liminf$), if $\limsup_{k\in\nn} \nu_k$ (resp. $\liminf_{k\in\nn} \nu_k$) is infinite, the value $\sup_{k\in
\nn} \nu_k$ (resp. $\inf_{k\in\nn} \nu_k$ is equal to $+\infty$ (resp. $-\infty$).
 
If one of the values $\inf_{k\in\nn} \nu_k$ or $\sup_{k\in\nn} \nu_k$ is not finite, the infinity passes to the related bound of $\ivt$. If those values are both finite, we can extract sequences in $\ivt$ which converge respectively to $\inf_{m\in\nn} \nu_m$ and $\sup_{m\in\nn} \nu_m$. For example, we can define $(\varphi(T^{k_n}(x_{k_n})))_n\subset \ivt$ such that for all $n\in\nn$, 
\[
\inf_{m\in\nn} \nu_m-1/n\leq \nu_{k_n}-1/n \leq \varphi(T^{k_n}(x_{k_n}))\leq \nu_{k_n}\leq \inf_{m\in\nn} \nu_m+1/n.
\] 
This sequence converges to $\inf_{m\in\nn} \nu_m$.
\end{proof}

Let $V$ be an Opt-Lyapunov function for $(\xin,T)$. To pass from $\alpha\circ \varphi \leq V$ to $\varphi\leq \inv{\alpha}{J}\alpha\circ V$, $\alpha$ must be surjective and strictly increasing as a function from $J$ to $\{V(x), x\in\xin\}\subseteq [0,1]$. From the earlier discussion, $J$ must also contain $\ivt$. Therefore, we define the following set of functions
\[
\scivt:=
\{\alpha:\rr\mapsto\rr : \exists\, I\in \IR \text{ s.t. }\alpha(I)=[0,1] \text{ and} \operatorname{conv}(\ivt\cup I)\in\ccdom(\alpha)\}
\]
\begin{proposition}
\label{uniqueI}
Let $\alpha\in\scivt$. There is only one interval $I$ such that $\alpha(I)=[0,1]$ and $I\in\ccdom(\alpha)$.
\end{proposition}
\begin{proof}
Let us consider $\alpha\in\scivt$ and $I$ such that $\alpha(I)=[0,1]$. As $\alpha$ is continuous and strictly increasing on $I$, $\inv{\alpha}{I}$ exists and is also continuous and strictly increasing from $[0,1]$ to $I$. Hence, $I$ is compact. Moreover, such an interval $I$ is unique. Otherwise, consider $I,J\in \ccdom(\alpha)$ such that $I\neq J$ and $\alpha(I)=\alpha(J)=[0,1]$. Suppose there exists $x\in I$ such that $x\leq \min J$.  We have $\alpha(x)<\alpha(\min J)=0$ but $x\in I$ and $\alpha(x)\geq \alpha(\min I)=0$. 
\end{proof}

We present some useful properties of functions belonging to $\scivt$ and their relation to $\nu$.
\begin{proposition}
\label{alphalimsup}
For all $\alpha\in\scivt$ such that there exists an Opt-Lyapunov function $V$ for $(\xin,T)$ satisfying $\alpha(\varphi(T^k(x)))\leq V(T^k(x))$ for all $x\in\xin$ and all $k\in\nn$, we have 
$\alpha(\limsup_{k\to +\infty} \nu_k)\leq 0$. 
\end{proposition}

\begin{proof}
Let $x\in\xin$ and $k\in\nn$. By assumption, $\alpha(\varphi(T^k(x)))\leq V(\varphi(T^k(x)))\leq \lambda^k V(x)$. Taking the supremum over $\xin$ and from the increasing condition and the continuity of $\alpha$, $\alpha(\nu_k)\leq \lambda^k \osp{V}$. Taking the limsup, we get $\limsup_{k\mapsto +\infty} \alpha(\nu_k)=\alpha(\limsup_{k\mapsto +\infty} \nu_k)\leq 0$. 
\end{proof}

\begin{corollary}
\label{lyaptolambda}
Suppose there exist $\alpha\in\scivt$ and an Opt-Lyapunov function $V$ for $(\xin,T)$ such that $\alpha(\varphi(T^k(x)))\leq V(T^k(x))$ for all $x\in\xin$ and $k\in\nn$. 
Then: 
\begin{enumerate}
\item $\nuopt<+\infty$;
\item If $\gsls{\nu}=\emptyset$, for all $\alpha(\varphi(T^k(x)))\leq 0$ for all $x\in\xin$ and all $k\in\nn$.
\end{enumerate}
\end{corollary}
\begin{proof}
The first statement follows readily from Prop.~\ref{alphalimsup} and the fact that $\limsup_{k\to +\infty} \nu_k<+\infty$ iff $\nuopt<+\infty$. The second statement follows from Prop.~\ref{argmaxsimple} and Prop.~\ref{alphalimsup}, since we have, if $\gsls{\nu}=\emptyset$, for all $x\in\xin$ and all $k\in\nn$, $\varphi(T^k(x))\leq \sup_{k\in\nn} \nu_k=\limsup_{n\to +\infty} \nu_n$.
\end{proof}

The set $\scivt$ defines the class of functions that we must consider to obtain relations $\alpha\circ \varphi \leq V$ when $V$ is an Opt-Lyapunov function for $(\xin,T)$. Another important consideration to take into account is \red{the} usefulness of functions $h\in\secfunh(\nu)$ based on elements $\alpha$ of $\scivt$. An Opt-Lyapunov function for $(\xin,T)$ is always nonnegative and if $\alpha(\nu_k)\leq 0$ for all $k\in\nn$, \red{$\alpha\circ \varphi\leq V$ obviously holds}. \red{However, a function $\alpha$ nonpositive on the terms $\nu_k$ whereas $\gsls{\nu}$ is nonempty is not useful. To be useful}, $\alpha(\nu_k)$ must be strictly positive for some $k\in\nn$. By Prop.~\ref{alphalimsup} and Corollary~\ref{lyaptolambda}, we have automatically 
$k\in\gsls{\nu}$. We propose a special definition for the particular functions $\alpha\in\scivt$.

\begin{defi}[$(\xin,T,\varphi)$-Certificate of compatibility]
Let $g\in\PSD$ such that $\{T^{k}(x): x\in\xin, k\in\nn\}\subseteq \dom(g)$. A function $\alpha\in \scivt$ such that $\alpha(\nu_k)>0$ for some $k\in\gsls{\nu}$ and $\alpha(\varphi(T^k(x)))\leq g(T^k(x))$ for all $k\in \nn$ and all $x\in\xin$ is called a $(\xin,T,\varphi)$-{\it certificate of compatibility} for $g$.
\end{defi}


The problem is to decide which conditions on $\xin$, $T$, $\varphi$ and $g\in\PSD$ such that $\{T^k(x): x\in\xin, k\in\nn\}\subseteq \dom(g)$ ensure the existence of $(\xin,T,\varphi)$-certificate of compatibility for $g$. Therefore, we introduce the concept of $(\xin,T,\varphi)$-compatibility.

\begin{defi}[$(\xin,T,\varphi)$-compatibility]
A function $g\in\PSD$ such that $\{T^k(x): x\in\xin, k\in\nn\}\subseteq \dom(g)$ is said to be $(\xin,T,\varphi)$-compatible at $k\in\gsls{\nu}$ if and only if there exist $\varepsilon>0$ and $\eta>0$ such that for all $x\in\xin$ and for all $j\in\nn$, verifying $\varphi(T^j(x))>\nu_k-\eta$, we have $g(T^j(x))>\varepsilon$. 

We say that $g$ is $(\xin,T,\varphi)$-compatible if there exists some $k\in\gsls{\nu}$ at which $g$ is $(\xin,T,\varphi)$-compatible.
\end{defi}
\begin{example}[$\lyaples$ in~\eqref{eq:opt-lyap-leslie} is $(\xinles,L,\philes)$-compatible]
	\label{ex:lyaplescomp}
	We have to prove that there exists $k\in\gsls{\omeles}$, $\varepsilon>0$ and $\eta>0$ such that for all $x\in\xin$, for all $j\in\nn$ such that $\philes(L^jx)>\omeles_k-\eta$, we have $\lyaples(L^jx)>\varepsilon$.
	
	Recall that $\overline{\lim}_n \omeles_n=\lambda_2/c$ and thus $k\in\gsls{\omeles}$ iff $\omeles_k>\lambda_2/c$. Let $k\in\gsls{\omeles}$ and let $\eta>0$ any real such that $\omeles_k-\eta>\lambda_2/c$. Let $K\in\nn$ such that $\omeles_j\leq \lambda_2/c$ for all $j>K$. Let $j\in\nn$ and $x\in\xinles$ such that $\philes(L^j(x))> \omeles_k-\eta$. It follows that $j\in\{0,\ldots,K\}$. From the discussion in Subsection~\ref{subsubsec:rapidanal}, $\pi_2(L^j x)\leq \max\{\pi_2(x),\pi_2(L^K x)\}$. Moreover, as for all $x\in\rr^2$ and all $k\in\nn$, $\pi_2(L^k x)=(\lambda_2-\lambda_1)^{-1}(\lambda_1^k\afct(x)+\lambda_2^k\bfct(x))$, it is easy to see that  $\pi_2(L^j x)\leq \max\{u_2,\pi_2(L^K \xuu)\}$. We conclude that, for all $x\in\xin$ and all $j\in\nn$ such that $\philes(L^j x)> \omeles_k-\eta$, we have 
	\[
	\begin{array}{ll}
		\lyaples(L^j x)=\dfrac{C}{(\pi_1(L^jx))^2+(\pi_2(L^jx))^2}&=\dfrac{C}{(\pi_2(L^jx))^2((\philes(L^jx))^2+1)}\\
		&\geq \dfrac{C}{(\max\{u_2,\pi_2(L^{K} \xuu))\})^2(\omelesopt^2+1)}:=\varepsilon>0
	\end{array}
	\]
	This proves that $\lyaples$ is an Opt-Lyapunov function $(\xinles,L,\philes)$-compatible.
	\qed
\end{example}

To be $(\xin,T,\varphi)$-compatible at $k\in\gsls{\nu}$ implies that if $T^k(x)$ is $\eta$-optimal for $\varphi$ then $g(T^k(x))$ is strictly positive. We extend the strict positivity of $g$ to all orbits of $T$ starting in $\xin$ greater than $\nu_k-\eta$. 

\begin{proposition}
\label{compatiblecone}
If $g$ is $(\xin,T,\varphi)$-compatible at $k\in\gsls{\nu}$, it is $(\xin,T,\varphi)$-compatible at every $j\in\gsls{\nu}$ such that $\nu_j\geq \nu_k$.
\end{proposition}

\begin{proof}
Suppose that $g$ is $(\xin,T,\varphi)$-compatible at $k\in\nn$. Let $\varepsilon>0$ and $\eta>0$ such that for all $x\in\xin$ and for all $j\in\nn$, verifying $\varphi(T^j(x))> \nu_k-\eta$, we have $g(T^j(x))>\varepsilon$. Let $j\in\nn$ such that $\nu_j\geq \nu_k$. Let $x\in\xin$, and $l\in\nn$ such that $\varphi(T^l(x))> \nu_j-\eta$. This implies that $\varphi(T^l(x))> \nu_j-\eta\geq \nu_k-\eta$ and thus $g(T^l(x))>\varepsilon$.
\end{proof}

The existence of $(\xin,T,\varphi)$-certificate of compatibility for $g$
is actually equivalent to the $(\xin,T,\varphi)$-compatibility of $g$.
\begin{theorem}
\label{compatibilityth}
A function $g$ is $(\xin,T,\varphi)$-compatible if and only if there exists $(\xin,T,\varphi)$-certificate of compatibility for $g$.
\end{theorem}

\begin{proof}
Suppose that there exists $\alpha\in \scivt$ such that $\alpha(\nu_k)>0$ for some $k\in\gsls{\nu}$ and $\alpha(T^k(x))\leq g(T^k(x))$ for all $k\in\nn$ all $x\in \xin$. Let $k\in\gsls{\nu}$ such that $\alpha(\nu_k)>0$. From Prop.~\ref{compatiblecone}, we can suppose that $\nu_k>\inf_{n\in\nn} \nu_n$ and thus there exists $\zeta>0$ such that $(\nu_k-\zeta,\nu_k)\subset \ivt$. As $\alpha$ is continuous and strictly increasing on $\ivt$, for $\varepsilon:=\alpha(\nu_k)/2$, there exists $\eta\in (0,\zeta]$, such that $0<\nu_k-s\leq \eta$ implies that $\varepsilon<\alpha(s)<\alpha(\nu_k)+\varepsilon$. Now let us consider $j\in\nn$ and $x\in\xin$ such that $\varphi(T^j(x))>\nu_k-\eta$. We have either $\nu_k>\varphi(T^j(x))$ or $\varphi(T^j(x))\geq \nu_k$. In the first case, we have, by hypothesis, $g(T^j(x))\geq \alpha(\varphi(T^j(x)))>\varepsilon$. In the second case, as $\alpha$ is increasing and $g(T^j(x))\geq \alpha(\varphi(T^j(x)))\geq \alpha(\nu_k)>\varepsilon$.

Now, assume that $g$ is $(\xin,T,\varphi)$-compatible at some $k\in \gsls{\nu}$. Then, from Prop.\ref{compatiblecone}, $g$ is also $(\xin,T,\varphi)$-compatible at $k=\bigkse_\nu$. Let $\varepsilon>0$ and $\eta>0$ such that for all $x\in\xin$ and for all $j\in\nn$, verifying $\varphi(T^j(x))> \nuopt-\eta$, we have $g(T^j(x))>\varepsilon$. Then we can define: $\alpha : s\mapsto s-\nuopt+\min\{\eta,\varepsilon\}$. Let $x\in\xin$ and $j\in\nn$. Then either $\varphi(T^j(x))\leq \nuopt-\min\{\eta,\varepsilon\}$ or $\varphi(T^j(x))> \nuopt-\min\{\eta,\varepsilon\}$. For the first case, $\alpha(\varphi(T^j(x)))\leq 0\leq g(T^j(x)))$. For the second, $0< \alpha(\varphi(T^j(x)))\leq \min\{\eta,\varepsilon\}$. By assumption on $g$, $g(T^j(x))>\varepsilon$. Note that the set $\{(x,j)\in\xin\times \nn :\varphi(T^j(x))>\nuopt-\min\{\eta,\varepsilon\}\}$ is nonempty. Finally, as $\alpha$ is strictly increasing, for all $k\in\gsls{\nu}$ such that $\nu_k\geq \nuopt-\min\{\eta,\varepsilon\}$, we have $\alpha(\nu_k)>0$. Note that the set $\{k\in\nn : \nu_k\geq \nuopt-\min\{\eta,\varepsilon\}\}$ is nonempty as $[\bigkse_\nu,\bigks_\nu]\cap \nn$ is included in it.
\end{proof}

\begin{example}[A $(\xinles,L,\philes)$-certificate of compatibility for $\lyaples$]
	\label{ex:leslyapcertif}
	From Th~\ref{compatibilityth} and Example~\ref{ex:lyaplescomp}, there exists a $(\xinles,L,\philes)$-certificate of compatibility for $\lyaples$. We will see in Theorem~\ref{directoptlyap} that this certificate is required to construct an element $(h,\beta)\in\funset(\omeles)$ from an Opt-Lyapunov function. 
	
	Thus, we construct a certificate of compatibility for $\lyaples$ proposed in~\eqref{eq:opt-lyap-leslie}.
	From the definition of $\lyaples$, we can easily establish a link between $\lyaples$ and $\philes$. Indeed, for all $x\in(\rr_+^{*})^2$:
	\[
	\lyaples(x)=\dfrac{C}{x_1^2+x_2^2}=\dfrac{C}{x_1^2(1+(\philes(x))^{-2})}.
	\]
	The functions of the form $s\mapsto C/(M^2(1+s^{-2}))$ are strictly increasing and continuous. One way to use these functions as lower bounds of $\lyaples$ is to find an upper bound of the set $\{\pi_1(L^k x) : x\in\xinles, k\in\nn\}$. However, from the discussion in~\ref{subsubsec:rapidanal}, $\{\pi_1(L^k x) : x\in\xin, k\in\nn\}$ is unbounded. Moreover, of the form $s\mapsto C/(M^2(1+s^{-2}))$ are strictly positive, whereas a certificate of compatibility must take negative values for all $k\in\nn$ such that $\omeles_k<\overline{\lim_n} \omeles_n$. 
	
	Then, the certificate is defined as a piecewise function on $\rr$. The pieces are one interval and its complement. The interval has a lower bound that is strictly greater than the limit superior. Hence, the certificate takes the form $s\mapsto C/(M^2(1+s^{-2}))$ on the interval. The definition of the interval permits us to consider a bounded subset of $\{\pi_1(L^k x): (x,k)\in\xinles\times \nn\}$ and thus compute $M$. On the complement of the interval, our certificate is an affine form that is proved to be a lower bound of $\lyaples$.   
	
	First, from the analysis of $\afct$, we have either $\omeles_0\in\gsls{\omeles}$ or $\omeles_1\in\gsls{\omeles}$. Then, we define \[z:=\max\{\omeles_0,\omeles_1\}.\]
	We have $z>\lambda_2/c$. 
	Let $\eta>0$ be such that $z-\eta>\lambda_2/c$ (for example $\eta=0.5(z-\lambda_2/c)$). As $\lambda_2/c$ is the limit of $(\omeles_k)_k$ and $z-\eta>\lambda_2/c$, there exists $K\in\nn$ such that $\omeles_k< z-\eta$ for all $k>K$. Then, for all $j\in\nn$ such that $\omeles_j\geq z-\eta$, we have $j\leq K$. We conclude that for all $j\in\nn$ and for all $x\in \xinles$ such that $\philes(L^j x)\geq z-\eta$, we have $\pi_1(L^j x)\leq \sup_{x\in\xinles} \pi_1(L^K x)$. Then, it suffices to compute at least an upper bound of this integer $K$. We define the function $N:\afct^{-1}(\rr\backslash\{0\})\to\mathbb Z$ for all $x\in \afct^{-1}(\rr\backslash\{0\})$ as follows:
	\[
	N(x):=
		\left\lfloor\dfrac{\ln\left(\dfrac{\bfct(x)(c(z-\eta)-\lambda_2)}{|\afct(x)|(c(z-\eta)-\lambda_1)}\right)}{\ln(\left|\lambda_1/\lambda_2\right|)}\right\rfloor
	\]
	For a given $x\in \xles$, the integer $N(x)$, if nonnegative, determines an upper bound of the maximal integer $k$ such that $\philes(L^k x)\geq z-\eta$. 
	Hence, the integer $\overline{N}$ defined as follows
		\[
	\overline{N}:=\sup_{x\in \xles} N(x)=\left\{
	\begin{array}{lr}
		N(\xul) & \text{ if } \afct(\xlu)\leq 0\\
		N(\xlu) & \text{ if } \afct(\xul)\geq 0\\
		\max\{N(\xlu),N(\xul)\} & \text{ if } \afct(\xlu)> 0 \text{ and } \afct(\xul)< 0
	\end{array}		
	\right.
	\]
	is an upper bound for $K$. Moreover, it can be proved that $\overline{N}$ is nonnegative. Then, we can compute $M=\max_{x\in \xinles}\pi_1(L^{\overline{N}} x)=\pi_1(L^{\overline{N}} \xuu)$. Actually, $\overline{N}$ is an upper bound of the greatest maximizer $\bigkse_\omega$. However, we continue as the goal is to produce a $(\xinles,L,\philes)$-certificate of compatibility. We are able to define: 
	\[
	\alpha_1:s\mapsto \dfrac{C}{M^2(1+s^{-2})}.
	\]
	Now let us take any $\gamma$ greater than
	\[
	\max\left\{\max_{s\in [z-\eta,z]} \alpha_1'(s),\dfrac{\lyaples(x^*)}{\eta}\right\}
	\]
	where $x^*\in\xinles\cup L(\xinles)$ is the vector such that $z=\philes(x^*)$.
	Note that as $\alpha_1$ is strictly increasing and $\lyaples$ is strictly positive, $\gamma$ is strictly positive.
	It follows that $\alpha_1(s)\geq \gamma(s-z)+\alpha_1(z)$. Moreover, there exists $u\in [z-\eta,z)$ such that $\gamma(u-z)+\alpha_1(z)=0$. We conclude that 
	\[
	\alpha:s\mapsto \left\{
	\begin{array}{lr}
		\dfrac{C}{M(1+s^{-2})} & \text{ if } s\geq z\\
		\gamma(s-z)+\dfrac{C}{M^2(1+z^{-2})} & \text{ if } s<z
	\end{array}
	\right.
	\]
	is a $(\xles,L,\philes)$-certificate of compatibility for $\lyaples$. 
	As for all $s\in [z-\eta,z]$, we have
	\[
	\alpha_1^{'}(s)=\dfrac{2Cs^{-3}}{(M^2(1+s^{-2}))^{2}}\leq \dfrac{2C(z-\eta)^{-3}}{(M^2(1+z^{-2})},
	\]
	a candidate $\gamma$ could be  
	\[
	\max\left\{\dfrac{2C(z-\eta)^{-3}}{(M^2(1+z^{-2}))^2},\dfrac{\lyaples(x^*)}{\eta}\right\}.
	\]
	
	Applying those computations to our numerical data of Subsection~\ref{subsec:numericaldata}, we have $z=\max\{\omeles_0,\omeles_1\}=42$ and 
	\[
	\eta=0.5(42-4)=19
	\]
	We have $\bfct(\xlu)=55$, $\afct(\xlu)=95$, $\bfct(\xul)=25$ and $\afct(\xul)=-10$. Hence, 
	\[
	N(\xlu)=	\left\lfloor\dfrac{\ln\left(\dfrac{319}{665}\right)}{-\ln(2)}\right\rfloor
	=1\text{ and } N(\xul)=	\left\lfloor\dfrac{\ln\left(\dfrac{19}{10}\right)}{-\ln(2)}\right\rfloor
	=-1
	\]
	Thus, the integer $\overline{N}$ is 1. We remark that the integer $N(\xul)=-1$ which means that to have $\philes(L^k \xul)\geq 42-19=23$, $k$ has to be negative. Indeed, for all $k\in\nn$, $\philes(L^k \xul)< 23$. We could stop here as $\overline{N}$ is an upper bound of the greatest integer maximizer of \eqref{eq:runningdpcp}.  
	
	We have to compute $M=\pi_1(L\xuu)^2=240^2=57600$. We have $x^*=L\xlu=(210,5)^\intercal$. Then 
	\[
	\dfrac{\lyaples(x^*)}{\eta}=\dfrac{125}{(19(210^2+5^2))}=\dfrac{1}{6707}\text{ and } \dfrac{2C(z-\eta)^{-3}}{M(1+z^{-2})^2}= \dfrac{250\times 23^{-3}}{57600(1+42^{-2})^2}=\dfrac{21609}{60644708120}
	\]
	 and thus $\gamma=1/6707$.
		
We conclude that 
 	\[
 \alpha:s\mapsto \left\{
 \begin{array}{lr}
 	\dfrac{125}{57600(1+s^{-2})} & \text{ if } s\geq 42\\
 	\\
 	\dfrac{(s-42)}{6707}+\dfrac{49}{22592} & \text{ if } s<42
 \end{array}
 \right.
 \]
	
	\qed
\end{example}
\subsection{Direct and Converse Opt-Lyapunov Theorems}
\label{directconverselyap}

We present both direct and converse Opt-Lyapunov theorems using the $(\xin,T,\varphi)$-compatibility notion. More precisely, the direct version aims to prove that if a $(\xin,T,\varphi)$-compatible Opt-Lyapunov function exists (and such that $\newop{T}(V)>0$), then we can construct a couple $(h,\beta)\in \funset(\nu)$ that is useful for $\nu$. The converse version exhibits a $(\xin,T,\varphi)$-compatible Opt-Lyapunov function when a couple $(h,\beta)\in \funset(\nu)$ is useful for $\nu$. Moreover, we can produce a $(\xin,T,\varphi)$-certificate of compatibility for this Opt-Lyapunov function.

We note that from Prop.~\ref{fixlyap}, if $V$ is an Opt-Lyapunov for $(\xin,T)$ then $\{T^k(x): k\in\nn, x\in\xin\}\subseteq \dom(V)$. 

\begin{proposition}
\label{zeronewop}
Assume that there exists a $(\xin,T,\varphi)$-compatible Opt-Lyapunov $V$ such that $N_V(T)=0$. Then $\bigks_\nu=0$ and $\nuopt=\osp{\varphi}$.
\end{proposition}

\begin{proof}
If $N_V(T)=0$ then for all $x\in\rd$ such that $V(x)>0$, $V(T(x))=0$. This also implies that $V(T^k(x))=0$ for all $k\in\nn^*$. Let $\alpha$ be a $(\xin,T,\varphi)$-certificate of compatibility. Thus, for all $k\in\nn^*$, $\alpha(\varphi(T^k(x)))\leq V(T^k(x))=0$ and as $\alpha$ is increasing and continuous, $\alpha(\nu_k)\leq 0$ for all $k\in\nn^*$. For some $j\in\gsls{\nu}$, $\alpha(\nu_j)>0$, this leads to $j=0$ and as $\alpha$ is strictly increasing and continuous on $\ivt$ then for all $k\in\nn^*$, $\nu_0>\nu_k$ and $\bigks_\nu=0$. This also implies that $\nuopt=\osp{\varphi}$.
\end{proof}

\begin{theorem}[Direct Opt-Lyapunov Theorem]
\label{directoptlyap}
Assume that there exists a $(\xin,T,\varphi)$-compatible Opt-Lyapunov $V$ such that $N_V(T)>0$. Let $\alpha$ be a $(\xin,T,\varphi)$-certificate of compatibility for $V$. Let $J=\overline{\operatorname{\conv}}(\ivt\cup I)$ where $\alpha(I)=[0,1]$. We define the function 
\begin{equation}
\label{eq:directlyaph}
h:\red{[0,1]\ni} x\mapsto \inv{\alpha}{J}\left(x\osp{V}\right).
\end{equation}
Then $(h,N_V(T))\in\funset(\nu)$ is useful for $\nu$.
\end{theorem}

\begin{proof}
The function $\inv{\alpha}{J}$ is strictly increasing and continuous on $[0,1]$. Thus, we have, for all $x\in\xin$ and all $k\in\nn$,
\[
\varphi(T^k(x))\leq \inv{\alpha}{J}(V(T^{k}(x)).
\] 
By Def.~\ref{extopdef} and Prop.~\ref{compoop}, $V(T^k(x))\leq (N_V(T))^k V(x)\leq N_V(T)^k \osp{V}$. As $\inv{\alpha}{J}$ is strictly increasing, we conclude that
\[
\nu_k=\sup_{x\in\xin} \varphi(T^k(x))\leq \inv{\alpha}{J} \left(N_V(T)^k \osp{V}\right).
\] 
Note that $x\mapsto \inv{\alpha}{J}\left( x\osp{V} \right)$ is strictly increasing and continuous on $J\supset [0,1]$ as $\inv{\alpha}{J}$ is, and $\osp{V}$ is not zero. Moreover, by Prop.~\ref{op01} and hypothesis, $N_V(T)\in (0,1)$. As $\alpha$ is a $(\xin,T,\varphi)$-certificate of compatibility for $V$, there exists $j\in\gsls{\nu}$ such that $\alpha(\nu_j)>0$. As $\inv{\alpha}{J}$ is strictly increasing and $\alpha(I)\ni 0$, $\nu_j>\inv{\alpha}{J}(0)=\inv{\alpha}{J}\left( 0\times \osp{V} \right)=h(0)$.
\end{proof}

We establish a converse Lyapunov theorem. The construction of the Opt-Lyapunov function for $(\xin,T)$ is inspired by the one proposed by Yoshizawa~\cite{yoshi}.

\begin{theorem}[Converse Opt-Lyapunov Theorem]
\label{converseoptlyap}
Let $(h,\beta)\in\funset(\nu)$ be useful for $\nu$. Let us define 
\[
\rr\ni s\mapsto \red{\mathfrak{m}}(s):=\min\left\{h(1),\max\left\{s,h(0)\right\}\right\}. 
\]
Then, the function $V\in\PSD$ is defined as follows:
\[
\rd\ni x\mapsto V(x)=\sup_{k\in\nn}\beta^{-k} \inv{h}{[0,1]}\left(  \mathfrak{m}\left(\varphi(T^k(x))\right)\right)
\]
is an Opt-Lyapunov function for $(\xin,T)$. Moreover, the function $\hat{h}$ is defined as follows:
\[
\hat{h}:\rr\mapsto \left\{
\begin{array}{lr}
s-h(0) & \text{ if } s\leq h(0)\\
\inv{h}{[0,1]}(s) & \text{ if } s\in (h(0),h(1)]
\end{array}
\right.
\]
is a $(\xin,T,\varphi)$-certificate of compatibility for $V$ making the function $V$ $(\xin,T,\varphi)$-compatible. 
\end{theorem}

\begin{proof}
First, we prove that the function $V$ is an Opt-Lyapunov function for $(\xin,T)$.

Let us prove that $V$ takes its values in $[0,+\infty]$. Let $x\in\rd$. As we have for all $s\in\rr$ $h(1)\geq  \mathfrak{m}(s)\geq h(0)$, we have, as $\inv{h}{[0,1]}$ is increasing, $1\geq \inv{h}{[0,1]}( \mathfrak{m}(\varphi(T^k(x))))\geq 0$ for all $k\in\nn$. We conclude, as $\beta\in(0,1)$ that $V(x)\geq 0$ and $V(x)=+\infty$ for all $x\in\rd$ such that $\varphi(T^k(x))\geq h(1)$ for all $k$ sufficiently large. 

By definition of $\funset(\nu)$, we have for all $x\in\xin$ and all $n\in\nn$, $\varphi(T^n(x))\leq h(\beta^n)<h(1)$ and so 
\[
V(x)\leq \sup_{k\in\nn} \beta^{-k} \inv{h}{[0,1]}(h(\beta^k))=\sup_{k\in\nn} \beta^{-k} \beta^k=1.
\] 
This implies that $\osp{V}\in [0,1]$. By assumption, $(h,\beta)\in\funset(\nu)$ and $\res(\nu,h)$ is nonempty. Let us take $k\in\nn$ such that $h(1)> \nu_k>h(0)$. Therefore, there exists $\varepsilon>0$ such that $\nu_k-\varepsilon>h(0)$. For this $\varepsilon$, there exists $x\in\xin$ such that $\nu_k\geq \varphi(T^k(x))\geq \nu_k-\varepsilon$. Finally, for this $x$, 
\[
V(x)\geq \beta^{-k} \inv{h}{[0,1]}(\varphi(T^k(x))>\inv{h}{[0,1]}(h(0))=0
\]
 as $\inv{h}{[0,1]}$ strictly increases. We conclude that $\osp{V}\in (0,1]$.
 Finally, we must prove that $V\circ T \leq \beta V$. Let $y\in\rd$. \red{If $V(y)=+\infty$, there is nothing to prove. Assume now $V(y)< +\infty$.} We have:
\[
\begin{array}{rl}
V(T(y))
=&\displaystyle{\sup_{k\in\nn} \beta^{-k} \inv{h}{[0,1]}\left( \mathfrak{m}(\varphi(T^{k+1}(y)))\right)}\\
=&\displaystyle{\sup_{j\in\nn^*} \beta^{-j+1} \inv{h}{[0,1]}\left( \mathfrak{m}(\varphi(T^{j}(y)))\right)}\\
=&\displaystyle{\beta \sup_{j\in\nn^*} \beta^{-j} \inv{h}{[0,1]}\left( \mathfrak{m}(\varphi(T^{j}(y)))\right)}\\
\leq &\displaystyle{\beta \sup_{j\in\nn} \beta^{-j} \inv{h}{[0,1]}\left( \mathfrak{m}(\varphi(T^{j}(y)))\right)=\beta V(y)}
\end{array}
\]
We conclude that $V$ is an Opt-Lyapunov function for $(\xin,T)$.

Now we prove that $\hat{h}$ is a $(\xin,T,\varphi)$-certificate of compatibility for $V$. Clearly, the function $\hat{h}$ is continuous and strictly increasing on $(-\infty,h(0))$ and $(h(0),h(1)]$. It is also continuous at $s=h(0)$ as $\lim_{s\to h(0)^+} \inv{h}{[0,1]}(s)=0$ and $s\mapsto s-h(0)$ tends to $0$ as $s$ tends to $h(0)$. Moreover for all $s< h(0)$, $\hat{h}(s)< 0\leq \hat{h}(t)$ for all $t\geq h(0)$. This proves that $\hat{h}$ is continuous and strictly increasing on $(-\infty,h(1)]$. Now, as $h\in\secfunh(\nu)$, $\ivt\subseteq (-\infty,h(1))$. Moreover, taking $I=[h(0),h(1)]$, we have $\hat{h}(I)=[0,1]$ and $\overline{\operatorname{\conv}}(\ivt\cup I)\in\ccdom(\hat{h})$. We conclude that $\hat{h}\in\scivt$. Now, we prove that $\hat{h}(\varphi(T^k(x)))\leq V(T^k(x))$ for all $x\in\xin$ and all $k\in\nn$. Therefore, let $x\in\xin$ and $k\in\nn$. Note that as $(h,\beta)\in \funset(\nu)$, we have $\varphi(T^k(x))\leq h(1)$. We have 
{\everymath{\displaystyle}
\[
\begin{array}{l}
V(T^k(x))=\\
\sup_{n\in\nn} \beta^{-n} \inv{h}{[0,1]}\left(\max\left\{\varphi(T^{k+n}(x)),h(0)\right\}\right)\geq \inv{h}{[0,1]}\left(\max\left\{\varphi(T^{k}(x)),h(0)\right\}\right)\geq \inv{h}{[0,1]}\left(\varphi(T^{k}(x))\right)
\end{array}
\]
}
 if $\varphi(T^{k}(x))\geq h(0)$. 
If $\varphi(T^k(x))<h(0)$, then $V(T^k(x))\geq \inv{h}{[0,1]}(h(0))=0$ and $\hat{h}(\varphi(T^k(x)))\leq 0$. We conclude that $\hat{h}(\varphi(T^k(x)))\leq V(T^k(x))$ for all $x\in\xin$ and all $k\in\nn$. As $\res(\nu,h)\neq \emptyset$, there exists $k\in\nn$, such that $h(1)>\nu_k>h(0)$. As $\inv{h}{[0,1]}$ is strictly increasing on $[h(0),h(1)]$, we have $\hat{h}(\nu_k)=\inv{h}{[0,1]}(\nu_k)>0$. This ends the proof that $\hat{h}$ is a $(\xin,T,\varphi)$-certificate of compatibility of $V$.
\end{proof}

\begin{example}
	As we choose to take $\gamma$ equals 
	\[
	\max\left\{\dfrac{2(\omega-\eta)^{-3}}{(M(1+\omega^{-2})},\dfrac{V(x^*)}{\eta}\right\},
	\]
	the function
	 \[
	\alpha:s\mapsto \left\{
	\begin{array}{lr}
		\dfrac{C}{M{}(1+s^{-2})} & \text{ if } s\geq z\\
		\gamma(s-z)+\dfrac{C}{M^2(1+z^{-2})} & \text{ if } s<z
	\end{array}
	\right.
	\]
	is a $(\xles,L,\philes)$-certificate of compatibility for $\lyaples$. 
	Finally, $(h,\beta)\in\funset(\omeles)$ where:
	\[
	h=\alpha^{-1} \text{ and } \beta=\dfrac{1}{\min\{a^2+c^2,b^2\}}.
	\]
	This leads to the formula for all $k\in\nn$ such that $\omeles_k>\alpha^{-1}(0)$
	\[
	\dfrac{\ln(\alpha(\omeles_k)/\osples)}{-\ln(\min\{a^2+c^2,b^2\})}.
	\]
	
	Applying our computations to the numerical data of Subsection~\ref{subsec:numericaldata} leads to
	$\alpha^{-1}(0)=\dfrac{1757}{64}>8=\omeles_0$. Then, we cannot use $\alpha$ at $\omeles_0$. Finally,
	\[
	\dfrac{\ln(\alpha(\omeles_1)/\osples)}{-\ln(\min\{a^2+c^2,b^2\})}=\dfrac{\ln\left(\dfrac{49}{22592}\right)}{\ln\left(\dfrac{4}{5}\right)}=27.49
	\] 
	This means that from the certificate of compatibility, we assume that the maximal term of $\omeles$ is one of the 28th first terms. This result is sound but not optimal. Indeed, the maximal term is $\omeles_1$.
	
	We see that the Opt-Lyapunov approach is less precise than $\klgen$-$\fini{\xles}$ approach. For the $\klgen$-$\fini{\xles}$ approach, we have to construct an upper bound of $\philes(L^k x)$. This bas been done directly from the function $H$ that characterizes $\philes(L^k x)$. This seems to explain why the upper bound is tight and thus the function $h\in\funset(\omeles)$ obtained provides better results. Opt-Lyapunov function $\lyaples$ defined in Eq.~\eqref{eq:opt-lyap-leslie} does not depend on $\philes$. The function depends only on the dynamics $L$ and initial conditions set. The link between $\lyaples$ and $\philes$ is established in Example~\ref{ex:lyaplescomp} and becomes explicit in Example~\ref{ex:leslyapcertif}. More generally, a certificate of compatibility is the function that relates a Lyapunov function to the orbits $\varphi(T^k x))$. 
	\qed
\end{example}

\section{A Technical Summary}
\label{summarysec}
\red{
In this paper, we have proposed a general method to solve optimization problems of the form~\eqref{pcpprob}. More precisely, focusing on the integer component of the optimal solution, it is necessary to compute an index $k\in\nn$ for which the term of the sequence $\nu$ (defined in~\eqref{nudef}) is maximal. The method is based on a function $\Fnu$ defined in~\cite{adje13052025} and quickly recalled in Eq.~\eqref{eq:fnu}. This function $\Fnu$ requires a pair $(h,\beta)$ where $h$ is a strictly increasing continuous function on $[0,1]$ and $\beta\in (0,1)$ such that $\nu_k\leq h(\beta^k)$ for all $k\in\nn$. A key notion is the \emph{usefulness} of $(h,\beta)$, that is, there exists some $k\in\nn$ such that $\nu_k>h(0)$. Without this condition, an upper bound of the greatest index maximizer of $\nu$ cannot be computed from $(h,\beta)$. However, the existence theorem~\cite[Th.2]{adje13052025} (partially recalled in Th~\ref{mainexistence}) is not totally constructive and computing $(h,\beta)\in\funset(\nu)$ has to follow an alternative methodology. We propose alternatives based on classical tools from the stability theory of difference equations. In~\cite{adje2025kllyap}, we compute $(h,\beta)\in\funset(\nu)$ from classical $\klcls$ and Lyapunov functions. However, this approach suffers from a lack of converse theorems. Indeed, some problems of the form~\eqref{pcpprob} (such as Problem~\eqref{eq:runningdpcp}) cannot be solved using these classes of functions. 
}

The inequality $\varphi(T^k(x))\leq h(\beta^k)$ for all $x\in\xin$ and for all $k\in\nn$ can be viewed as a kind of $\mathcal{KL}$ inequality that appears in asymptotic stability studies. Indeed, $h$ is strictly increasing and continuous on $[0,1]$ and $\beta\in (0,1)$ making the sequence $(h(\beta^k))_k$ strictly decreasing and convergent to $h(0)$. From this observation, we propose a more general $\klcls$ class of functions called $\klgen$ where the positivity with respect to the state variable and the continuity assumptions are removed (compared with the $\klcls$ class). We have proved (Th. ~\ref{htoklbounds}) that a $\klgen$-$\fin$ upper bound for $(\xin,T,\varphi)$ exists if and only if $\nuopt$ is finite: Then, we extend Sontag's result to the $\klgen$ class (Lemma~\ref{dectoregular} and Lemma~\ref{alasontag}) to construct $h\in\fset$, useful for $\nu$, from a $\klgen$ function and such that $(h,e^{-1})\in\funset(\nu)$ (Th.~\ref{KLversh}). One of the advantages of the extension of Sontag’s result is that the scalar part of the elements of $\funset(\nu)$ can be fixed, and it remains to construct $h$ from the $\klgen$ certificate. In practice, to decide whether a function $(\gamma,\fin)$ is a $\klgen$-$\fin$ upper bound of $\varphi(T^k(x))$ seems difficult to make. 

One constructive alternative to $\klcls$ certificates consists in Lyapunov functions.    
Classical Lyapunov functions satisfy pure functional inequations and are not defined from the orbits of the dynamical system. One of the most interesting theoretical results for Lyapunov functions is the necessity of their existence for stability (namely, the converse Lyapunov theorems). For our specific problem, we propose a new type of Lyapunov function called the Opt-Lyapunov function, for which we obtain both direct (Th. ~\ref{directoptlyap}) and the converse theorem (Th~\ref{converseoptlyap}). The main difference between classical and Opt-Lyapunov functions is that Opt-Lyapunov functions are nonnegative but not definite (many zeros) and can take infinite values (outside reachable values taken by the dynamical system). The fact that Opt-Lyapunov functions are not definite (and not continuous) renders classical comparison functions impossible to use. In the classical case, comparison functions serve as a link between our objective function $\varphi$ and the classical Lyapunov function. This link (and comparison functions) is replaced by the $(\xin,T,\varphi)$-compatibility notion (Th.~\ref{compatibilityth}). 
Theorem~\ref{summaryth} summarizes the equivalences proved in this study with respect to the peak computation problem.
\begin{theorem}[Summary]
	\label{summaryth}
	The following statements are equivalent.
	\begin{enumerate}
		\item The set $\gsls{\nu}$ is nonempty;
		\item There $(h,\beta)\in\funset(\nu)$ for which $\res(\nu,h)$ is nonempty;
		\item There exists $(h,\beta)\in\funset(\nu)$ such that $\Fnu(\bigks_\nu,h,\beta)$ is \red{well-defined} and finite and $\nuopt=\max_{k\in\nn} \{\nu_k : k=0,\ldots, \red{\Fnu(\bigks_\nu,h,\beta)}\}$;
		\item There exists $\theta\in\fin$ and $\gamma\in\klgen$ such that $\varphi(T^k(x))\leq \gamma(\theta(x),k)$ for all $x\in\xin$ and all $k\in\nn$ and $\inf_{t\geq 0} \gamma(\osp{\theta},t)<\nu_n$ for some $n\in\nn$; 
		\item There exists an Opt-Lyapunov $(\xin,T,\varphi)$ compatible function.
	\end{enumerate}
\end{theorem}   

The proof of Theorem is obtained by sketching the functions constructed in the proofs of the earlier theorems. Figure~\ref{figsummary} provides a graphical view of the different results involved in proving Th. ~\ref{summaryth}. In Figure~\ref{figsummary}, the notation $\mathcal{OL}_{\rm comp}$ denotes the set of $(\xin,T,\varphi)$-compatible Opt-Lyapunov functions. 
\begin{figure}[h]
	\fbox{
		\begin{minipage}{0.97\textwidth}
			\begin{center}
				\begin{tikzpicture}
					\draw (0,0) node[above]{$\funset(\nu)$};
					\draw (-4,-2) node[below]{$\klgen\times \fin$};
					\draw (4,-2) node[below]{$\mathcal{OL}_{\rm comp}\times \scivt$};
					\draw (4.6,0) node[above]{$\mathcal{OL}_{\rm comp}$};
					\draw (0,-2.5) node[font=\Large,below]{$\Fnu$};
					\draw[very thick,->] (-0.4,0) -- (-4.4,-2);
					\draw[very thick,->] (-4.1,-2) -- (-0.1,0);
					\draw[very thick,->] (0.4,0) -- (4.4,-2);
					\draw[very thick,<-] (0.1,0) -- (4.1,-2);
					\draw[very thick,->] (0,-2) -- (0,-2.5);
					\draw[very thick] (0,0) -- (0,-1.5);
					\draw[very thick,dashed,->] (4.6,0) -- (4.6,-1.9); 
					\draw (0,-1.5) node[below]{Th.~\ref{mainexistence}};
					\draw (-2,-1.1) node[right]{Th.~\ref{KLversh}};
					\draw (-2.6,-1.4) node[right]{\textcircled{\small 1}};
					\draw (-2,-0.6) node[left]{Th.~\ref{htoklbounds}};
					\draw (-2,-0.3) node[right]{\textcircled{\small 2}};
					\draw (2,-1.1) node[left]{Th.~\ref{directoptlyap}};
					\draw (1.9,-1.4) node[right]{\textcircled{\small 3}};
					\draw (2,-0.6) node[right]{Th.~\ref{converseoptlyap}};
					\draw (1.4,-0.3) node[right]{\textcircled{\small 4}};
					\draw (4.6,-0.6) node[right]{Th.~\ref{compatibilityth}};
					\draw (4.7,-1.1) node[right]{\textcircled{\small 3}};
				\end{tikzpicture}
			\end{center}
		\end{minipage}
	}
	\caption{The scheme showing the mappings between $\klgen$, $\mathcal{OL}_{\rm comp}$ and $\funset(\nu)$}
	\label{figsummary}
\end{figure}

As $V\in\mathcal{OL}_{\rm comp}$ is $(\xin,T,\varphi)$-compatible, we associate with them a $(\xin,T,\varphi)$-certificate of compatibility. For such a function $V$, we denote by $\alpha_V$ the chosen $(\xin,T,\varphi)$-certificate of compatibility.

We now clarify the circled numbers in Figure~\ref{figsummary}.
\begin{itemize} 
	\item The mapping \textcircled{\small 1} constructs from an element of $\funset(\nu)$ a $\klgen$-$\fin$ upper bound for $(\xin,T,\varphi)$:
	\[
	\begin{array}{ccc}
		\funset(\nu)&\longrightarrow &\klgen\times \fin \\
		(h,\beta)&\longmapsto& \left(\rr\times \rr_+\ni (s,t)\mapsto \displaystyle{\min\{s,h(\beta^t)\}},\displaystyle{\rd\ni x\mapsto \sup_{k\in\nn} \varphi(T^k(x))}\right)
	\end{array}
	\]
	\item The mapping \textcircled{\small 2} produces an element of $\funset(\nu)$ from a $\klgen$-$\fin$ upper bound for $(\xin,T,\varphi)$:
	\[
	\begin{array}{ccc}
		\klgen\times \fin&\longrightarrow &\funset(\nu)\\ 
		(\gamma,\theta)&\longmapsto &(h,e^{-1})
	\end{array}
	\]
	The function $h$ cannot be explicitly formulated and is constructed following the proofs of Lemmas ~\ref{dectoregular} and ~\ref{alasontag}.

	
	\item The mapping \textcircled{\small 3} asks more explanations as the symbol \textcircled{\small 3} appears twice. The mapping \textcircled{\small 3} associates an Opt-Lyapunov $(\xin,T,\varphi)$ compatible with an element of $\funset(\nu)$ but needs an intermediate argument, which is a $(\xin,T,\varphi)$ certificate of compatibility for the Opt-Lyapunov function: 
	\[
	\begin{array}{ccccc}
		\mathcal{OL}_{\rm comp}&\longrightarrow &\mathcal{OL}_{\rm comp}\times \scivt & \longrightarrow & \funset(\nu)\\
		V &\longmapsto & (V,\alpha_V)& \longmapsto &(h_V,\newop{T}(V))
	\end{array}
	\]
	where $h_V$ is defined as 
	\[
	[0,1]\ni s\mapsto h_V(s)=\inv{\alpha_V}{J}(s\osp{V})
	\]
	where $J=\overline{\operatorname{\conv}}(\ivt\cup I)\in \ccdom(\alpha_V)$ where $\alpha_V(I)=[0,1]$.
	
	\item The mapping \textcircled{\small 4} links an element $\funset(\nu)$ with an Opt-Lyapunov $(\xin,T,\varphi)$-compatible function and an associated $(\xin,T,\varphi)$-certificate of compatibility:
	\[
	\begin{array}{ccc}
		\funset(\nu)&\longrightarrow& \mathcal{OL}_{\rm comp}\times \scivt\\
		(h,\beta)&\longmapsto &(V,\alpha_V)
	\end{array}
	\]
	where:
	\begin{itemize}
		\item[$\bullet$] $\displaystyle{\rd\ni x\mapsto V(x)=\sup_{k\in\nn} \beta^{-k} \inv{h}{[0,1]}\left(\omega(\varphi(T^k(x)))\right)}$;
		\item[$\bullet$] $\displaystyle{\rr\ni s\mapsto \omega(s)=\min\{h(1),\max\{s,h(0)\}\}}$;
		\item[$\bullet$]$(-\infty,h(1)]\ni s\mapsto \alpha_V(s)=\left\{\begin{array}{lr} s-h(0) & \text{ if } s\leq h(0)\\ \inv{h}{[0,1]}(s) & \text{ if } s\in [h(0),h(1)]\end{array}\right.$
	\end{itemize}
\end{itemize}
\bibliographystyle{alpha} 
\bibliography{supremumbib}

\newpage

\appendix \section{}
\subsection{Additional proofs for Opt-Lyapunov definition}

We provide details about the proofs of Lemmas ~\ref{poslemma} and ~\ref{lemmafact}.

\newcounter{lemapp}
\counterwithin{lemapp}{section}
\setcounter{lemapp}{2}
\newtheorem{lemmaAppendix}[lemapp]{Lemma}
\begin{lemmaAppendix}
	Let $\gamma$ be strictly greater than 1. Let $f\in\kclsi$ verifying for all $x>0$, $f(x)>x$. Then, there exists a function $g\in\kcls$ which satisfies $g(f(x))=\gamma g(x)$ for all $x\in\rr_+$.
\end{lemmaAppendix}

\begin{proof}
	As we have $f(x)>x$ for all $x>0$, we have $f(1)>1$ and thus the interval $[1,f(1))$ is nonempty. We define $h:[1,f(1))$ as continuous and strictly increasing, with strictly positive values such that $h(1)=1$ and $\lim_{x\mapsto f(1)^{-}} h(x)=\gamma$. For example, we could take $h$ to be the affine map linking $(1,1)$ to $(f(1),\gamma)$. Now, as the sequence $(f^k(1))_k$ is strictly increasing and unbounded (if $M=\sup_{k\in\nn} f^k(1)>0$ is finite, then as $f$ is increasing and continuous, $f(M)=M$ but $M<f(M)$ by the assumption on $f$). As $f$ is invertible on $\rr_+$ and we have $f^{-1}(x)<x$ for all $x>0$, the sequence $(f^{-k}(1))_k$ is strictly decreasing and positive and converges to 0 as $k$ tends to $+\infty$.
	Hence $\rr_+^*$ is union of disjoint intervals $[f^{k}(1),f^{k+1}(1))$ where $k\in\mathbb{Z}$. Then for all $x>0$, there exists a unique $k\in\mathbb{Z}$ such that $x\in [f^{k}(1),f^{k+1}(1))$. We define a function $g$ on $\rr_+$ as follows $g(0)=0$ and for all $x>0$, $g(x):=\gamma^k h(f^{-k}(x))$ where $k$ is the unique integer such that $x\in [f^{k}(1),f^{k+1}(1))$. As $g$ is equal to $h$ on $[1,f(1)]$, then $g$ is strictly increasing, continuous, and strictly positive on $[1,f(1)]$. Then, on the open interval $(f^{k}(1),f^{k+1}(1))$ for $k\neq 0$, $g$ is the composition $t\mapsto \gamma^k t$, $h$ and $f^{-k}$ and thus is strictly increasing, continuous and strictly positive on $(f^{k}(1),f^{k+1}(1))$ for all $k\in \mathbb{Z}$. Now, for all $k\neq 0$, $\lim_{y\to {f^{k}(1)}^-} g(y)=\lim_{y \to {f^k(1)}^-} \gamma^{k-1} h(f^{-k+1}(y))=\lim_{z=f^{-k}(y) \to 1^-} \gamma^{k-1} h(f(z))=\gamma^{k-1+1}=\gamma^k$. We also have $\lim_{y\to {f^{k}(1)}^+} g(y)=\lim_{y \to {f^k(1)}^+} \gamma^{k} h(f^{-k}(y))=\lim_{z \to 1^-} \gamma^{k} h(z)=\gamma^{k}h(1)=\gamma^k$. We conclude that $g$ is continuous on $\rr_+^*$. Now, we prove that $g$ is continuous at 0 (at right). Let us consider a sequence $(x_n)_{n\in\nn}$ of positive reals that converges to 0. Let $n\in\nn$. There exists a unique $k_n\in\mathbb Z$ such that $f^{k_n}(1)\leq x_n\leq f^{k_n+1}(1)$. We get $1\leq f^{-k_n}(x_n)\leq f(1)$ and as $h$ is increasing on $[1,f(1))$ and continuous (at left) at $f(1)$, $1=h(1)\leq h(f^{-k_n}(x_n))\leq h(f(1))=\gamma$. Finally, $\gamma^{k_n}\leq g(x_n)=\gamma^{k_n}h(f^{-k_n}(x_n))\leq \gamma^{k_n+1}$. As $(x_n)_{n\in\nn}$ tends to 0, the sequence $(k_n)_{n\in\nn}$ tends to $-\infty$ and we conclude that $(g(x_n))_{n\in\nn}$ tends to $0=g(0)$ and $g$ is continuous at 0. We have already proven that $g$ is strictly increasing on each interval $[f^{k}(1),f^{k+1}(1))$. It is easy to see that this is also the case when we pick two reals $x,y$ in two different intervals: $x<y$ implies $g(x) <g(y)$.
	
	Finally, we must prove that $g(f(x))=\gamma g(x)$ for all $x\in\rr_+$. As $f(0)=0$, this holds for $x=0$. Now, let us take $x>0$. This real $x$ belongs to some interval $[f^k(1),f^{k+1}(1))$. Thus, $f(x)\in [f^{k+1}(1),f^{k+2}(1))$ and 
	$g(f(x))=\gamma^{k+1}h(f^{-k-1}(f(x))=\gamma \gamma^k h(f^{-k}(x))=\gamma g(x)$. This ends the proof.
\end{proof}

\begin{lemmaAppendix}
	Let $\lambda\in (0,1)$. Let $f\in\kcls$ verifying for all $x>0$, $f(x)<x$. Then, there exists a function $g\in\kcls$ satisfying $g(f(x))=\lambda g(x)$ for all $x\in\rr_+$.
\end{lemmaAppendix}

\begin{proof}
	The function $f^{-1}$ is continuous, strictly increasing, and definite, and satisfies $f^{-1}(x)>x$ for all $x>0$. Then using Lemma~\ref{poslemma}, there exists $g\in\kcls$ such that $g(f^{-1}(x))=\lambda^{-1} g(x)$ for all $x\in\rr$. Then, for all $y\in\rr_+$, there exists a unique $x\in\rr_+$ such that $y=f^{-1}(x)$ or equivalently $x=f(y)$. Then, we have $g(y)=\lambda^{-1} g(f(y))$ and we conclude that $g(f(y))=\lambda g(y)$ for all $y\in\rr_+$. 
\end{proof}

\newcounter{propapp}
\counterwithin{propapp}{section}
\setcounter{propapp}{0}
\newtheorem{propAppendix}[propapp]{Proposition}

\newcounter{coroapp}
\counterwithin{coroapp}{section}
\setcounter{coroapp}{0}
\newtheorem{coroAppendix}[coroapp]{Corollary}
\subsection{Technical details about examples on the application for Leslie models}

We provide technical details about the family of $\klgen$ functions proposed in Equation~\eqref{eq:klgenepsdef} and the $\klgen$-$\fini{\xles}$ upper bound construction proposed in Equation~\eqref{eq:finalklgenfin}. Then, we provide details on the construction of the $(\xinles,L,\philes)$ certificate of compatibility for the Lyapunov function $\lyaples$ proposed in Equation~\eqref{eq:opt-lyap-leslie}. Recall that passing from $\xinles$ to $\xles$ does not the solutions of Problem~\eqref{eq:runningdpcp}.

First we give details about the analysis of coordinate functions. 

\begin{lemmaAppendix}
	For all $x\in(\rr_+^*)^2$, $(\pi_1(L^k x))_{k\in\nn}$ and $(\pi_2(L^k x))_{k\in\nn}$ are strictly increasing (from $k=2$ for $\pi_2$) and both tend to $+\infty$ as $k$ goes to $+\infty$. 
\end{lemmaAppendix}

\begin{proof}
	Let $x\in(\rr_+^*)^2$, we have $\pi_1(Lx)=a\pi_1(x)+b\pi_2(x)>a\pi_1(x)$. We conclude that $(\pi_1(L^k x))_k$ is strictly increasing as $(\rr_+^*)^2$ by $L$. Moreover, we have $\pi_2(L^2 x)=c\pi_1(Lx)=ca\pi_1(x)+cb\pi_2(x)>ac\pi_1(x)=a\pi_2(Lx)$. We conclude that  $(\pi_2(L^k x))_k$ is strictly increasing from $k\geq 2$. 
	
	Now, we also have
	\[
	\pi_1(L^k x)=\dfrac{1}{c(\lambda_2-\lambda_1)}\left(\lambda_1^{k+1}\afct(x)+\lambda_2^{k+1}\bfct(x)\right)=\dfrac{\lambda_2^{k+1}}{c(\lambda_2-\lambda_1)}\left(\left(\dfrac{\lambda_1}{\lambda_2}\right)^{k+1}\afct(x)+\bfct(x)\right)
	\]
	As $\lambda_2>1$ and from~\eqref{eq:abprop}, $\pi_1(L^k x)$ tends to $+\infty$ as $k$ goes to $+\infty$. We use the same techniques to prove that $\pi_1(L^k x)$ tends to $+\infty$ as $k$ goes to $+\infty$.
\end{proof}

\subsubsection{On the \protect$\klgen$-$\fini{\xles}$ certificate construction}

First recall that 
	\[
	H:\{(x,y)\in\rr^2: x+y\neq 0\}\ni (u,v)\mapsto \dfrac{\lambda_2}{c} \dfrac{\dfrac{\lambda_1}{\lambda_2}v+u}{v+u}
\]
and we have for all $x\in \xles$ and all $k\in\nn$
\[
	\philes(L^k x)=\dfrac{\lambda_1^{k+1}\afct(x)+\lambda_2^{k+1}\bfct(x)}{c\left(\lambda_1^k\afct(x)+\lambda_2^k\bfct(x)\right)}=
	\dfrac{\lambda_2}{c}\dfrac{\dfrac{\lambda_1}{\lambda_2} \left(\dfrac{\lambda_1}{\lambda_2}\right)^k+\dfrac{\bfct(x)}{\afct(x)}}{\left(\dfrac{\lambda_1}{\lambda_2}\right)^k+\dfrac{\bfct(x)}{\afct(x)}}=H\left(\dfrac{\bfct(x)}{\afct(x)},\left(\dfrac{\lambda_1}{\lambda_2}\right)^k\right).
\]
We also have:
\begin{itemize}
	\item for all $v>0$ (resp. $v<0$), $u\mapsto H(u,v)$ is strictly increasing on $(-\infty,-v)$ and on $(-v,+\infty)$ (resp. strictly decreasing on $(-\infty,-v)$ and on $(-v,+\infty)$);
	\item for all $u<0$ (resp. $u>0$), $v\mapsto H(u,v)$ is strictly increasing on $(-\infty,-u)$ and on $(-u,+\infty)$ (resp. strictly decreasing on $(-\infty,-u)$ and on $(-u,+\infty)$). 
\end{itemize}

\begin{propAppendix}
For all $\varepsilon>0$, the function 
\[
\gamma_\varepsilon:\rr\times\rr_+\ni (s,t)\mapsto \left\{\begin{array}{lr}
H\left(s,\left|\dfrac{\lambda_1}{\lambda_2}\right|^t\right) & \text{ if } s\leq -\left|\dfrac{\lambda_1}{\lambda_2}\right|^t-\varepsilon\\
\left(s+\left|\dfrac{\lambda_1}{\lambda_2}\right|^t+\varepsilon\right)+H\left(-\left|\dfrac{\lambda_1}{\lambda_2}\right|^t-\varepsilon,\left|\dfrac{\lambda_1}{\lambda_2}\right|^t\right) & \text{ otherwise }  
\end{array}\right.
\]
belongs to $\klgen$.

\end{propAppendix}

\begin{proof}
Let $t\in\rr_+$, for sake of simplicity we write $\omega_t=|\lambda_1/\lambda_2|^t$. Let $s,s'\in\rr$ such that $s<s'$. First suppose that $s<s'\leq -\omega_t-\varepsilon$. As $\omega_t>0$, $u\mapsto H(u,\omega_t)$ is strictly increasing on $(-\infty,-\omega_t)$. Since $s,s'<-\omega_t$, we get $H(s,\omega_t)<H(s',\omega_t)$. Now, suppose that $s'>s>-\omega_t-\varepsilon$. Then, $(s'+\omega_t+\varepsilon>(s+\omega_t+\varepsilon)$ and thus $\gamma_\varepsilon(s',t)>\gamma_\varepsilon(s,t)$ as $H(-\omega_t-\varepsilon,\omega_t)$ does not depend on $s,s'$. Finally, suppose that $s<-\omega_t-\varepsilon<s'$. The result follows readily from the fact that $H(\cdot,\omega_t)$ strictly increases on $(-\infty,-\omega_t)$ and $(s'+\omega_t+\varepsilon)$ is strictly positive.

Now let $s\in\rr$. Let $t,t'\in \rr_+$ such that $t<t'$. Keeping the same notation that is $\omega_t$ for $|\lambda_1/\lambda_2|^t$, we have $\omega_{t'}<\omega_t$. If $s\leq -\omega_t-\varepsilon$, we have $\gamma_\varepsilon(s,t)=H(s,\omega_t)$ and $\gamma_\varepsilon(s,t')=H(s,\omega_{t'})$. The conclusion follows from the fact that $v\mapsto H(s,v)$ strictly increases as $s<0$ on $(-\infty,-s)$ and $\omega_{t'}<\omega_t<\omega_{t}+\varepsilon\leq -s$. Now suppose that $s>-\omega_{t'}-\varepsilon$. 
Actually, we have for all $u,v$ such that $u+v\neq 0$, $H(u,v)=\lambda_2/c(1+(\lambda_1/\lambda_2-1)v(u+v)^{-1})$ and thus $H(-\omega_t-\varepsilon,\omega_t)=\lambda_2/c(1+\varepsilon^{-1}\omega_t(1-\lambda_1/\lambda_2))$ and $H(-\omega_{t'}-\varepsilon,\omega_{t'})=\lambda_2/c(1+\varepsilon^{-1}\omega_{t'}(1-\lambda_1/\lambda_2)$. We conclude that $v\mapsto (s+v+\varepsilon)+\lambda_2/c(1+\varepsilon^{-1}v(1-\lambda_1/\lambda_2)$ strictly increases and $\gamma_\varepsilon (s,t')<\gamma_\varepsilon (s,t)$. Finally, suppose that  $-\omega_{t'}-\varepsilon\geq s>-\omega_t-\varepsilon$. Then $\gamma_\varepsilon(s,t')=H(s,\omega_t')$ and $\gamma_\varepsilon(s,t)=(s+\omega_{t}+\varepsilon)+H(-\omega_{t}-\varepsilon,\omega_{t})$. Then $(s+\omega_{t}+\varepsilon)+H(-\omega_{t}-\varepsilon,\omega_{t})-H(s,\omega_t')=(s+\omega_{t}+\varepsilon)+\lambda_2/c(1+\varepsilon^{-1}\omega_{t}(1-\lambda_1/\lambda_2)-\lambda_2/c(1+(\omega_{t'}+s)^{-1}\omega_{t'}(\lambda_1/\lambda_2-1)=(s+\omega_{t}+\varepsilon)+\lambda_2/c(1-\lambda_1/\lambda_2)(\varepsilon^{-1}\omega_{t}+\omega_{t'}(\omega_{t'}+s)^{-1})$. Now as $-\omega_{t'}-\varepsilon\geq s$, we have $(\omega_{t'}+s)^{-1}\geq -\varepsilon^{-1}$. Thus, $(s+\omega_{t}+\varepsilon)+H(-\omega_{t}-\varepsilon,\omega_{t})-H(s,\omega_t')\geq (s+\omega_{t}+\varepsilon)+\lambda_2/c(1-\lambda_1/\lambda_2)\varepsilon^{-1}(\omega_t-\omega_{t'})>0$ and $\gamma_\varepsilon(s,t)>\gamma_\varepsilon(s,t')$.
\end{proof}

Recall that 
\[
\theles:\afct^{-1}(\rr\backslash\{0\}) \ni x\mapsto -\left|\dfrac{\bfct(x)}{\afct(x)} \right|.
\]

We propose a slightly different approach than that in Example~\ref{ex:klgenfinleslie}. We decompose the supremum of $\theles$ over $\afct^{-1}((0,+\infty))$ into two parts. Those have an additional constraint that is $(\afct-\bfct)^{-1}((-\infty,0))$ for the first and $(\afct-\bfct)^{-1}([0,+\infty))$ for the second. This explains the different cases in the definition of $\varepsilon^{>0}$ in~\eqref{eq:varepssupthe}. 

We recall that $\bfct$ is strictly positive on $(\rr_+^*)^2$. Hence, $\afct(x)\geq \bfct(x)$ implies that $\afct(x)>0$.
\begin{propAppendix}
	\label{propapp:sup}
The following equalities are true:
{\everymath{\displaystyle}
\[
\begin{array}{l} 
\sup_{\substack{\afct(x)>0\\ \afct(x)< \bfct(x)\\ x\in \xles}} \theles(x)=\left\{
\begin{array}{lr}
\min\left\{\theles(\xlu),-1\right\} & \text{ if } \{x\in \xles : 0<\afct(x)< \bfct(x)\}\neq \emptyset\\
-\infty & \text{ otherwise}
\end{array}\right., \\
\sup_{\substack{\afct(x)>0\\ \afct(x)\geq \bfct(x)\\ x\in \xles}} \theles(x)=\left\{
\begin{array}{lr}
\theles(\xlu)& \text{ if } \{x\in \xles : \afct(x)\geq \bfct(x)\}\neq \emptyset\\
-\infty & \text{ otherwise}
\end{array}\right.,\\
\sup_{\substack{\afct(x)<0\\ x\in \xles}} \theles(x)=\left\{
\begin{array}{lr}
\theles(\xul) & \text{ if } \{x\in \xles : \afct(x)<0\}\neq \emptyset\\
-\infty &  \text{ otherwise}
\end{array}\right. .
\end{array}
\]
}
Moreover, if $\afct(\xlu)<\bfct(\xlu)$, $ \{x\in \xles : \afct(x)\geq \bfct(x)\}$ is empty; and if $\afct(\xlu)\geq \bfct(\xlu)$, the set $\{x\in \xles : \afct(x)>0,\ \afct(x)< \bfct(x)\}$, if nonempty, contains elements for which $\theles$ is arbitrarily close to -1. 
\end{propAppendix}

\begin{proof}
Let us consider the supremum of $\theles$ on $C_1:=\{x\in X: \afct(x)>0,\ \afct(x)<\bfct(x)\}$. If $C_1$ is empty, the supremum is known to be $-\infty$. Suppose that $C_1$ is nonempty. First, we must have $\afct(\xlu)>0$. Otherwise, $\afct(\xlu)\leq 0$ which is the same as $\lambda_2/c\leq l_1/u_2$. This implies that $\afct(x)\leq 0$ for all $x\in X$ and $C_1$ cannot be nonempty. Second, $\theles$ is strictly smaller than -1 on $C_1$.
From the monotonicity analysis of $\theles$, we have for all $x$ different from $\xlu$ such that $\afct(x)>0$, $\theles(x)<\theles(\xlu)$. If $0<\afct(\xlu)<\bfct(\xlu)$ then $\xlu\in C_1$ and it is the unique maximizer of $\theles$ on $C_1$. Note that $0<\afct(\xlu)<\bfct(\xlu)$ is equivalent to $\theles(\xlu)<-1$ and so $\min\{\theles(\xlu),-1\}=\theles(x)$. Now, suppose that $\afct(\xul)\geq \bfct(\xul)$. As $\xinles$ is connected and $(\afct-\bfct)$ continuous, $\afct-\bfct(\xinles)$ is an interval. Moreover, $(\afct-\bfct)(\xlu)\geq 0$ and as $C_1$ is nonempty $(\afct-\bfct)(y)<0$ for some $y\in \xinles$ with $\afct(y)>0$. Hence, there exists $z\in\xinles$ such that $\afct(z)=\bfct(z)$. As $\bfct$ is strictly positive on $\xinles$, $\afct(z)$ is strictly positive. Recall that $\xinles$ does not contain 0, and thus $y\neq 0$. Let $\eta>0$ and write $u_\eta=\eta y\norm{y}^{-1}+z$. We have $\afct(u_\eta)=\eta\norm{y}^{-1}\afct(y)+\afct(z)>0$ and $(\afct-\bfct)(u_\eta)= \eta\norm{y}^{-1}(\afct-\bfct)(y)+(\afct-\bfct)(z)=\eta\norm{y}^{-1}(\afct-\bfct)(y)<0$. Then, $u_\eta\in C_1$ and $\theles(u_\eta)$ tends to $\theles(z)=-1$ as $\eta$ tends to 0. We conclude that $\sup_{C_1} \theles=-1$. Note that as $\afct(\xlu)\geq \bfct(\xlu)$ we have $-1\leq \theles(\xlu)$ and thus $\min\{-1,\theles(\xlu)\}=-1$. Simple computations show that $0<\afct(\xlu)<\bfct(\xlu)$ is the same as $a/(2c)<l_1/u_2$. This implies that for all $x\in\xles$, 
 $a/(2c)<x_1/x_2$ or equivalently $0<\afct(x)<\bfct(x)$ for all $x\in\xles$.
 
We prove that the supremum of $\theles$ on $C_2:=\{x\in X: \afct(x)\geq \bfct(x)\}$ equals $\theles(\xlu)$ if $C_2$ is nonempty. It suffices to prove that $\afct(\xlu)\geq \bfct(\xlu)$ if $C_2$ is nonempty. Let $x\in C_2$. This is the same as $a/2c\geq x_1/x_2$. This implies that $a/2c\geq l_1/u_2$ and so $\afct(\xlu)\geq \bfct(\xlu)$.

We prove that the supremum of $\theles$ on $C_3:=\{x\in X: \afct(x)<0\}$ equals $\theles(\xul)$ if $C_3$ is nonempty. From the monotonicity analysis of $\theles$, it suffices to prove that $\afct(\xul)<0$ if $C_3$ is nonempty. Let $x\in C_3$. This is the same as $\lambda_2/c< x_1/x_2$. This implies that $\lambda_2/c< u_1/l_2$ and so $\afct(\xul)<0$.
\end{proof}

\begin{coroAppendix}
	\label{coro:maxtheles}
	The maximum of $\theles$ over $\xles$ can be computed as 
	{\everymath{\displaystyle}
		\[
		\max_{x\in \xles} \theles(x)=\left\{
			\begin{array}{lr}
				\theles(\xul) & \text{ if } \afct(\xlu)\leq 0\\
				\theles(\xlu)& \text{ if } \afct(\xul)\geq 0\\
				\max\{\theles(\xul),\theles(\xlu)\} & \text{ if } \afct(\xul)<0\text{ and } \afct(\xlu)>0
			\end{array}\right.
		\]
	}
\end{coroAppendix}

Now recall that
	\[
	\varepsilon^{<0}=\left\{\begin{array}{lr} \dfrac{1}{2}(-1-\theles(\xul)) & \text{ if } \afct(\xul)<0\\
		+\infty & \text{ if } \afct(\xul)\geq 0\end{array}\right.
	\text{ and } 
	\varepsilon^{>0}=\left\{\begin{array}{lr} \dfrac{1}{2}(-1-\theles(\xlu)) & 0<\afct(\xlu)<\bfct(\xlu)\\
		\dfrac{1}{2}\left(\dfrac{\lambda_1}{\lambda_2}-\theles(\xlu)\right) & \text{ if } \afct(\xlu)\geq\bfct(\xlu)\\
		+\infty & \text{ if } \afct(\xlu)\leq 0\end{array}\right.
\]
and 
\[
	\epsles=\min\{\varepsilon^{<0},\varepsilon^{>0}\}.
\]
\begin{propAppendix}
	We have for all $x\in \xles$ and all $k\in\nn$:
	\[
	\philes(L^k x)\leq \gamma_{\epsles}\left(\theles(x),\left|\dfrac{\lambda_1}{\lambda_2}\right|^k\right)
	\]
\end{propAppendix}
\begin{proof}
First, from~\eqref{eq:abprop}, we have, for all $k\in\nn$, 
\[
\forall\, x\text{ s.t. } \afct(x)<0,\ \dfrac{\bfct(x)}{\afct(x)}<-\dfrac{\lambda_1}{\lambda_2}^k\text{ and } \forall\, x\text{ s.t. } \afct(x)>0,\ -\dfrac{\bfct(x)}{\afct(x)}<-\dfrac{\lambda_1}{\lambda_2}^k.
\]
Hence,
\[
\forall\, k\in\nn,\ \forall\, x\text{ s.t. } \afct(x)<0,\ \theles(x)<-1\leq \left|\dfrac{\lambda_1}{\lambda_2}\right|^k\text{ and } \forall\, k\in\nn^*,\ \forall\, x\text{ s.t. } \afct(x)>0,\ \theles(x)<\dfrac{\lambda_1}{\lambda_2}\leq \left|\dfrac{\lambda_1}{\lambda_2}\right|^k.
\]
By definition of $\varepsilon^{<0},\varepsilon^{>0}$ and $\epsles$, we have 
\[
\forall\, k\in\nn,\ \forall\, x\text{ s.t. } \afct(x)<0,\ \theles(x)<-1-\epsles\leq \left|\dfrac{\lambda_1}{\lambda_2}\right|^k-\epsles
\]
and 
\[
\forall\, k\in\nn^*,\ \forall\, x\text{ s.t. } \afct(x)>0,\ \theles(x)<\dfrac{\lambda_1}{\lambda_2}-\epsles\leq \left|\dfrac{\lambda_1}{\lambda_2}\right|^k-\epsles.
\]
Finally,
\[
\forall\, k\in\nn,\ \forall\, x\text{ s.t. } \afct(x)<0,\ \gamma_{\epsles}(\theles(x),k)=H\left(\theles(x),\left|\dfrac{\lambda_1}{\lambda_2}\right|^k\right)
\]
and 
\[
\forall\, k\in\nn^*,\ \forall\, x\text{ s.t. } \afct(x)>0,\  \gamma_{\epsles}(\theles(x),k)=H\left(\theles(x),\left|\dfrac{\lambda_1}{\lambda_2}\right|^k\right).
\]

Now, from Eq.~\eqref{eq:varphiH} and as for all $x\in\xles$, $v\mapsto H(\theles(x),v)$ strictly increases on $(-\infty,-\theles(x))$, we conclude for all $x\in \xles$ such that $\afct(x)<0$ and for all $k\in\nn$: 
\[
\philes(L^k x)=H\left(\theles(x),\left(\dfrac{\lambda_1}{\lambda_2}\right)^k\right)\leq H\left(\theles(x),\left|\dfrac{\lambda_1}{\lambda_2}\right|^k\right)=\gamma_{\epsles}(\theles(x),k).
\] 
Recall that $\afct\circ L=\lambda_1\afct$ and $\bfct\circ L=\lambda_2\bfct$. Thus, for $\afct(x)>0$ we have $\afct(Lx)<0$  and we can apply the previous inequality to $Lx$. We have for all $x\in \xles$ such that $\afct(x)>0$ and for all $k\in\nn^*$
{\everymath{\displaystyle}
\[
\begin{array}{ll}
\philes(L^k x)=\philes(L^{k-1}L x)&=H\left(\dfrac{\bfct(Lx)}{\afct(L x)},\left(\dfrac{\lambda_1}{\lambda_2}\right)^{k-1}\right)\\
&\leq H\left(\dfrac{\bfct(Lx)}{\afct(Lx)},\left|\dfrac{\lambda_1}{\lambda_2}\right|^{k-1}\right)=H\left(-\dfrac{\bfct(x)}{\afct(x)},\left|\dfrac{\lambda_1}{\lambda_2}\right|^{k}\right)=H\left(\theles(x),\left|\dfrac{\lambda_1}{\lambda_2}\right|^{k}\right)=\gamma_{\epsles}(\theles(x),k).
\end{array}
\]
}
It remains to prove the inequality at $k=0$ for all $x\in \xles$ such that $\afct(x)>0$, that is, $\philes(x)\leq \gamma_{\epsles}(\theles(x),0)$ for all $x\in \xles$ such that $\afct(x)>0$.

Suppose that $0<\afct(\xlu)<\bfct(\xlu)$. It implies that for all $x\in\xles$, $0<\afct(x)<\bfct(x)$ which is the same as $\theles(x)<-1$ and by definition of $\varepsilon^{>0}$ and $\epsles$, $\theles(x)\leq -1-\epsles$ for all $x\in \xles$. Then, for all $x\in \xles$, $\gamma_{\epsles}(\theles(x),0)=H(\theles(x),1)$. Since
\[
H\left(\theles(x),1\right)-\varphi(x)=H\left(\theles(x),1\right)-H\left(-\theles(x),1\right)=2\theles(x)\dfrac{\lambda_2}{c}\dfrac{1-\frac{\lambda_1}{\lambda_2}}{\left(1+\theles(x)\right)\left(1-\theles(x)\right)}>0,
\]
we conclude that, if $0<\afct(\xlu)<\bfct(\xlu)$, for all $k\in\nn$, for all $x\in \xles$ such that $\afct(x)>0$:
\[
\philes(L^k x)\leq H\left(\theles(x),\left|\dfrac{\lambda_1}{\lambda_2}\right|^{k}\right)=\gamma_{\epsles}(\theles(x),k).
\]	
Now assume that $\afct(\xlu)\geq \bfct(\xlu)$ and let $x\in\xles$ such that $\afct(x)>0$. If $\{x\in \xles : 0<\afct(x)< \bfct(x)\}$ is nonempty, we can have $\theles(y)>-1-\epsles$ for some $y\in\xles$ such that $0<\afct(y)< \bfct(y)$ as proved in Proposition~\ref{propapp:sup}. For the elements $y\in \{x\in \xles : 0<\afct(x)< \bfct(x)\}$ for which $\theles(y)\leq -1-\epsles$, the inequality $\philes(y)\leq \gamma_{\epsles}(\theles(y),0)$ has been proved. Therefore, suppose that $\theles(x)>-1-\epsles$. We remark that, for all $u>0$ and $\varepsilon>0$
\[
H(-u-\varepsilon,u)= \dfrac{\lambda_2}{c}\dfrac{u+\varepsilon-\frac{\lambda_1}{\lambda_2}u}{\varepsilon}=\dfrac{\lambda_2}{c}\left(1+\dfrac{u}{\varepsilon}\left(1-\frac{\lambda_1}{\lambda_2}\right)\right)>\dfrac{\lambda_2}{c}.
\]
It follows that:
\[
\gamma_\varepsilon(\theles(x),0)=\theles(x)+1+\epsles+H(-1-\epsles,1)>\dfrac{\lambda_2}{c}>\dfrac{x_1}{x_2}=\philes(x).
\]
This ends the proof.
\end{proof}
\subsubsection{Details on certificate of compatibility}

We bring details about the $(\xles,L,\philes)$ certificate of compatibility 
proposed in Example~\ref{ex:leslyapcertif}. 

\begin{lemmaAppendix}
	\label{lem:defM}
	Let $z=\max\{\omeles_0,\omeles_1\}$ and let $\eta>0$ such that $z-\eta>\lambda_2/c$. 
	
	Recall that the function $N:\afct^{-1}(\rr\backslash\{0\})\to\mathbb Z$ is defined as follows:
	\[
	\afct^{-1}(\rr\backslash\{0\}) \ni x\to N(x):=
	\left\lfloor\dfrac{\ln\left(\dfrac{\bfct(x)(c(z-\eta)-\lambda_2)}{|\afct(x)|(c(z-\eta)-\lambda_1)}\right)}{\ln\left(\left|\dfrac{\lambda_1}{\lambda_2}\right|\right)}\right\rfloor.
	\]
	Then, 
	\[
	\overline{N}:=\max_{x\in \xles} N(x)=\left\{
	\begin{array}{lr}
		N(\xul) & \text{ if } \afct(\xlu)\leq 0\\
		N(\xlu) & \text{ if } \afct(\xul)\geq 0\\
		\max\{N(\xlu),N(\xul)\} & \text{ if } \afct(\xlu)> 0 \text{ and } \afct(\xul)< 0
	\end{array}		
	\right.
	\] 
	and
	\[
	M:=\max_{x\in \xinles} \pi_1(L^{\overline{N}} x)=\pi_1(L^{\overline{N}} \xuu)
	\]
	satisfies $M\geq \pi_1(L^j x)$ for all $x\in\xles$, for all $j\in\nn$ such that $\philes(L^j x)\geq z-\eta$.
\end{lemmaAppendix}
\begin{proof}
	Let $x\in \xles$ and $k\in\nn$. Then, from \eqref{eq:abprop}, 
	{\everymath{\displaystyle}
		\[
		\begin{array}{ll}
			\philes(L^k x)\geq z-\eta&\iff \dfrac{\lambda_2}{c}\dfrac{\afct(x)\left(\dfrac{\lambda_1}{\lambda_2}\right)^{k+1}+\bfct(x)}{\afct(x)\left(\dfrac{\lambda_1}{\lambda_2}\right)^{k}+\bfct(x)}\geq z-\eta\\
			&\iff \left(\dfrac{\lambda_1}{\lambda_2}\right)^k\afct(x)(\lambda_1-c(z-\eta))\geq \bfct(x)(c(z-\eta)-\lambda_2)
		\end{array}
		\]
	}
	Now, if $\afct(x)<0$, we get, applying the natural logarithm,
	{\everymath{\displaystyle}
		\[
		\begin{array}{ll}
			\philes(L^k x)\geq z-\eta&\implies \left|\dfrac{\lambda_1}{\lambda_2}\right|^k\geq \dfrac{\bfct(x)(c(z-\eta)-\lambda_2)}{|\afct(x)|(c(z-\eta)-\lambda_1)}\\
			&\implies k\leq N(x)
		\end{array}
		\]
	}
	
	If $\afct(x)>0$, we have $\afct(Lx)<0$. Thus, $\philes(L^k x)\geq z-\eta \implies \philes(L^{k-1}y)\geq z-\eta$ where $y=Lx$ $\implies k-1\leq N(y)=K(Lx)$. Moreover, for all $u\in\afct^{-1}(\rr\backslash\{0\})$ (and thus $\afct(Lu)\neq 0$):
	{\everymath{\displaystyle}
		\[
		\begin{array}{ll}
			N(Lu)=\left\lfloor\dfrac{\ln\left(\dfrac{\bfct(Lx)(c(z-\eta)-\lambda_2)}{|\afct(Lx)|(c(z-\eta)-\lambda_1)}\right)}{\ln\left(\left|\dfrac{\lambda_1}{\lambda_2}\right|\right)}\right\rfloor&=\left\lfloor\dfrac{\ln\left(\dfrac{\lambda_2\bfct(x)(c(z-\eta)-\lambda_2)}{|\lambda_1||\afct(x)|(c(z-\eta)-\lambda_1)}\right)}{\ln\left(\left|\dfrac{\lambda_1}{\lambda_2}\right|\right)}\right\rfloor\\
			&=\left\lfloor\dfrac{-\ln\left(\left|\dfrac{\lambda_1}{\lambda_2}\right|\right)+\ln\left(\dfrac{\bfct(x)(c(z-\eta)-\lambda_2)}{|\afct(x)|(c(z-\eta)-\lambda_1)}\right)}{\ln\left(\left|\dfrac{\lambda_1}{\lambda_2}\right|\right)}\right\rfloor=N(u)-1
		\end{array}
		\] 
	}
	We conclude that for all $x\in\afct^{-1}(\rr\backslash\{0\})$, $\philes(L^k x)\geq z-\eta\implies k\leq N(x)$.
	
	Now, we prove the formula for the computation of $\overline{N}$. As $\lambda_1/\lambda_2\in (-1,0)$, the function $\rr_+^*\ni t\mapsto \ln(ct (z-\eta)-\lambda_2)/(c(z-\eta)-\lambda_1))/\ln(|\lambda_1/\lambda_2|)$ strictly decreases. Thus, $N$ is maximal on $\xles$ when $\bfct/|\afct|=-\theles$ is minimal on $\xles$. We have already computed the supremum of $\theles$ over $\xles$ in Corollary~\ref{coro:maxtheles} which corresponds to the opposite of the infimum $-\theles$ over $\xles$. Hence, minimizers on $\xles$ of $-\theles$ are the maximizers of $\theles$ over $\xles$. This yields the definition of $\overline{N}$. 
	
	To prove that $\overline{N}$ is nonnegative, it suffices to prove that $\philes(\xul)\geq z-\eta$ when $\afct(\xlu)\leq 0$, $\philes(L\xlu)\geq z-\eta$ when $\afct(\xul)\geq 0$ and $\philes(\xul)\geq z-\eta$ or $\philes(L\xlu)\geq z-\eta$ when $\afct(\xlu)>0$ and $\afct(\xul)<0$. We have already proved that $\omeles_{2k}=\philes(L^{2k}\xul)\text{ and }\omeles_{2k+1}=\philes(L^{2k+1}\xlu)$.
	It follows that $\omeles_0=\philes(\xul)$ and $\omeles_1=\philes(L\xlu)$. 
	First, we assume that $z=\omeles_0$. Then, as $z$ is strictly greater than $\lambda_2/c$, it is also the case for $\omeles_0$. Since $\omeles_0=\philes(\xul)$, we have $\afct(\xul)<0$. And finally, $N(\xul)\geq 0$ if and only $\bfct(\xul) (c(z-\eta)-\lambda_2)\leq  -\afct(\xul)(c(z-\eta)-\lambda_1)$ that is equivalent to $c(z-\eta)\leq (\lambda_1\afct(\xul)+\lambda_2 \bfct(\xul))/(\afct(\xul)+\bfct(\xul))$ that is the same as $z-\eta\leq \varphi(\xul)=z$. Hence, we conclude that $N(\xul)\geq 0$ if $z=\omeles_0$.
	Now assume that $z=\omeles_1=\philes(L\xlu)$. As $z$ is strictly greater than $\lambda_2/c$, 
	$\afct(\xlu)>0$. It follows that $N(\xlu)\geq 0$ if and only if $c(z-\eta)(\bfct(\xlu)-\afct(\xlu))\leq -\lambda_1\afct(\xlu)+\lambda_2\bfct(\xlu)$. If $\bfct(\xlu)-\afct(\xlu)\leq 0$, the inequality is valid as $\lambda_1<0$ and $\lambda_2$, $\afct(\xlu)$ and $\bfct(\xlu)$ are strictly positive. Assume that $\bfct(\xlu)-\afct(\xlu)>0$. Then $N(\xlu)\geq 0$ if and only if $z-\eta\leq (\lambda_2/c)(-\lambda_1 \afct(\xlu)/\lambda_2+\bfct(\xlu))/(-\afct(\xlu)+\bfct(\xlu))$. Now
	\[
	\begin{array}{ll}
		\dfrac{-\frac{\lambda_1}{\lambda_2}\afct(\xlu)+\bfct(\xlu)}{-\afct(\xlu)+\bfct(\xlu)}> \dfrac{-\frac{\lambda_1}{\lambda_2}\afct(\xlu)+\bfct(\xlu)}{\frac{\lambda_1}{\lambda_2}\afct(\xlu)+\bfct(\xlu)}&=\dfrac{\left(\frac{\lambda_1}{\lambda_2}\right)^2\afct(\xlu)-\left(\frac{\lambda_1}{\lambda_2}\right)^2\afct(\xlu)-\frac{\lambda_1}{\lambda_2}\afct(\xlu)+\bfct(\xlu)}{\frac{\lambda_1}{\lambda_2}\afct(\xlu)+\bfct(\xlu)}\\
		&=\varphi(L\xlu)-\dfrac{\frac{\lambda_1}{\lambda_2}\left(1+\frac{\lambda_1}{\lambda_2}\right)}{\frac{\lambda_1}{\lambda_2}\afct(\xlu)+\bfct(\xlu)}
	\end{array}
	\]
	We conclude that $(\lambda_2/c)(-\lambda_1 \afct(\xlu)/\lambda_2+\bfct(\xlu))/(-\afct(\xlu)+\bfct(\xlu))>\varphi(L\xlu)=z>z-\eta$ and $N(\xlu)\geq 0$. 
	Then, we can compute $M=\max_{x\in \xinles}\pi_1(L^{\overline{N}} x)$. First, from the discussion of Subsection~\ref{subsubsec:rapidanal}, for all $k\leq \overline{N}$, $\pi_1(L^k x)\leq \pi_1(L^{\overline{N}} x)\leq M$. Now as for all $x\in\rr^2$ and all $k\in\nn$, $\pi_1(L^k x)=(c(\lambda_2-\lambda_1))^{-1}(\lambda_1^{k+1}\afct(x)+\lambda_2^{k+1}\bfct(x))$ and thus $\pi_1(L^k x)$ increases in $x_1$ and $x_2$ and we conclude that $\pi_1(L^k x)\leq \pi_1(L^k \xuu)$ for all $x\in\xinles$ and for all $k\in\nn$. Finally, $M=\pi_1(L^{\overline{N}} \xuu)$.
\end{proof}
\begin{propAppendix}
	Let $z=\max\{\omeles_0,\omeles_1\}$ and $x^*\in \xles\cup L(\xles)$ such that $\philes(x^*)=z$. 
	Let us consider 
	\[
	\alpha_1:\rr\ni s\to \dfrac{C}{M^2(1+s^{-2})}.
	\]
	Then for all $\eta>0$ such that $z-\eta>\lambda_2/c$ and for all $\gamma$ greater than
	\[
	\max\left\{\max_{s\in [z-\eta,z]} \alpha_1'(s),\dfrac{\lyaples(x^*)}{\eta}\right\},
	\]
	the function  
	\[
	\alpha:\rr\ni s\mapsto \left\{
	\begin{array}{lr}
		\dfrac{C}{M^2(1+s^{-2})} & \text{ if } s\geq z\\
		\gamma(s-z)+\dfrac{C}{M^2(1+z^{-2})} & \text{ if } s<z
	\end{array}
	\right.
	\]
	is a $(\xles,L,\philes)$-certificate of compatibility for $\lyaples$.
\end{propAppendix}
\begin{proof}
From Lemma~\ref{lem:defM}, for all $x\in\xles$ and for all $k\in\nn$ such that $\varphi(L^k x)\geq z-\eta$ we have 
\[
\lyaples(L^k x)=\dfrac{C}{\pi_1(L^k x)^2+\pi_2(L^k x)^2}=\dfrac{C}{\pi_1(L^k x)^2(1+\philes(L^k x)^{-2})}\geq \dfrac{C}{M^2(1+\philes(L^k x)^{-2})}
\]

We define 
\[
\alpha_1:s\mapsto \dfrac{C}{M^2(1+s^{-2})}
\]
Now let us take any $\gamma$ greater than
\[
\max\left\{\max_{s\in [z-\eta,z]} \alpha_1'(s),\dfrac{\lyaples(x^*)}{\eta}\right\}.
\]
Note that as $\alpha_1$ is strictly increasing and $\lyaples$ is strictly positive, $\gamma$ is strictly positive.
From the mean value theorem, for all $s\in [z-\eta,z]$, $\alpha_1(z)-\alpha_1(s)=(z-s)\alpha_1^{'}(t)$ for some $t\in (s,\omega)$. Then by definition of $\gamma$, 
$\alpha_1(z)-\alpha_1(s)\leq \gamma(z-s)$ or equivalently $\alpha_1(z)+\gamma(s-z)\leq \alpha_1(z)$. By definition of $\alpha_1$, we have $\alpha_1(z)+\gamma(\philes(L^k x)-z)\leq V(L^k z)$ for all $x\in\xinles$ and all $k\in\nn$ such that $\philes(L^k x)\in [z-\eta,z]$. Moreover, there exists $u\in [z-\eta,z)$ such that $\gamma(u-z)+\alpha_1(z)=0$. 
Indeed, $V(L^j x)\geq \alpha_1(\varphi(L^j x))$ for all $j\in\nn$ and $x\in\xles$ such that $\varphi(L^j x)\geq z-\eta$ and $\varphi(x^*)=z$ with either $x^*\in\xles$ if $z=\omeles_0$ or $x^*\in L(\xles)$ if $z=\omeles_1$. From this, we have 
$\gamma(z-\eta-z)+\alpha_1(z)\leq -V(x^*)+\alpha_1(z)\leq -\alpha_1(z)+\alpha_1(z)=0$. We also have $\gamma(z-z)+\alpha_1(z)>0$. Either, $u=z-\eta$ or $-\eta\gamma+\alpha_1(z)<0$ and $u\in (z-\eta,z)$. It follows from the strict positivity of $\lyaples$ and the strict increase of $s\mapsto \alpha_1(z)+\gamma(s-z)$ that for all $x\in\xles$ and all $k\in\nn$ such that $\philes(L^k x)\leq u$, $\alpha_1(z)+\gamma(\philes(L^k x)-z)\leq \lyaples(L^k x)$. We conclude that 
\[
\alpha:s\mapsto \left\{
\begin{array}{lr}
	\dfrac{C}{M(1+s^{-2})} & \text{ if } s\geq z\\
	\gamma(s-z)+\dfrac{C}{M^2(1+z^{-2})} & \text{ if } s<z
\end{array}
\right.
\]
is a $(\xles,L,\philes)$-certificate of compatibility for $\lyaples$. 
\end{proof}
\end{document}